\documentclass[11pt]{article}

\usepackage[utf8]{inputenc}

\usepackage[mathscr]{eucal}
\usepackage{enumerate} 
\usepackage{subcaption}
\usepackage[margin=1in]{geometry}
\usepackage{amsrefs}
\usepackage{amsmath}
\usepackage{amsthm}
\usepackage{graphicx}
\usepackage{amsfonts}
\usepackage{amssymb}
\usepackage{color}
\usepackage{bm}
\usepackage{hyperref}
\usepackage{enumerate} 
\usepackage[margin=1in]{geometry}
\usepackage{relsize}
\numberwithin{equation}{section}
\usepackage{footnote}
\usepackage{footmisc}
\usepackage{setspace}
\usepackage{verbatim}
\usepackage{tocloft}
\usepackage{tikz}
\usetikzlibrary{fit, shapes.geometric,mindmap}

\newtheorem{theorem}{Theorem}[section]
\newtheorem{lemma}[theorem]{Lemma}
\newtheorem{proposition}[theorem]{Proposition}
\newtheorem{corollary}[theorem]{Corollary}

\theoremstyle{definition}
\newtheorem{example}[theorem]{Example}
\newtheorem{definition}[theorem]{Definition}
\newtheorem{remark}[theorem]{Remark}

\newtheorem{assumptio}{Assumption}

\newtheorem*{sassu}{Standing Assumption}

\newtheorem{assumptionalt}{Assumption}[theorem]
\newenvironment{assumptionp}[1]{
  \renewcommand\theassumptionalt{#1}
  \assumptionalt
}{\endassumptionalt}

\newcommand\blfootnote[1]{%
\begingroup
\renewcommand\thefootnote{}\footnote{#1}%
\addtocounter{footnote}{-1}%
\endgroup
}

\newcommand{\Rb}{{\mathbb R}}
\newcommand{\Nb}{{\mathbb N}}
\newcommand{\PP}{{\mathbb P}}
\newcommand{\FF}{{\mathbb F}}
\newcommand{\Pc}{{\mathcal P}}
\newcommand{\Xc}{{\mathcal X}}
\newcommand{\Zb}{{\mathbb Z}}
\newcommand{\Fc}{{\mathcal F}}

\newcommand{\pspace}{(\Omega, {\mathcal F}, \Pmb)} 
\newcommand{\pzspace}{(\wh{\Omega},\wh{\mathcal F},\wh{\Pmb})}

\newcommand{\fpspace}{(\Omega, {\mathcal F}, \FF, \Pmb)}

\newcommand{\filt}{\Fmb}
\newcommand{\filtm}{\Fc}

\newcommand{\hfilt}{\Hmb}
\newcommand{\hfiltm}{\Hmc}

\newcommand{\fpp}{filtration-Poisson process pair}

\newcommand{\leb}{{\rm Leb}}
\newcommand{\Sm}{\varsigma}
\newcommand{\gm}{\zeta}

\newcommand{\neigh}[1]{\Nmc_{#1}}				
\newcommand{\gneigh}[2]{\Nmc_{#1}(#2)}			
\newcommand{\cl}[1]{\te{cl}_{#1}}				
\newcommand{\gcl}[2]{\te{cl}_{#1}(#2)}				
\renewcommand{\root}{o}
\newcommand{\subg}[1]{[#1]}	
\newcommand{\nm}[1]{[#1]}
\newcommand{\ppmk}{\mathbf{\zeta}}
\newcommand{\vms}{\mathbf{\kappa}}	
\newcommand{\dvms}{\vartheta}
\newcommand{\ems}{\ov{\mathbf{\kappa}}}		

\newcommand{\vmsn}[1]{\mathbf{\kappa}^{#1}}		
\newcommand{\emsn}[1]{\ov{\mathbf{\kappa}}^{#1}}	
\newcommand{\pare}[2]{\pi_{#1}(#2)}	
\newcommand{\chil}[2]{c_{#1}(#2)}
\newcommand{\perc}[2]{\te{perc}_{#1}(#2)}
\newcommand{\sO}{{\mathcal O}}
\newcommand{\tree}{\Tmc}
\newcommand{\trnc}[1]{B_{#1}}	
\newcommand{\cmpn}[1]{\mathscr{C}_{#1}}

\newcommand{\cad}{\Dmc}
\newcommand{\Pol}{\Zmc}						
\newcommand{\borel}{\Bmc}					
\newcommand{\fset}[1]{\Lambda_{#1}}
\newcommand{\ic}[1]{\langle #1\rangle}

\newcommand{\gs}{\Gms_*} 
\newcommand{\gsonenew}{{\Gms}_{*,1}} 
\newcommand{\gsone}{\alt{\Gms}_{*,1}}
\newcommand{\wgsone}{\wh{\Gms}_{*,1}}
\newcommand{\vmksp}{\Kmc}
\newcommand{\emksp}{\ov{\Kmc}}

\newcommand{\etra}{\odot}
\newcommand{\et}[1]{#1_{\etra}}
\renewcommand{\sp}[1]{[#1]}	
\newcommand{\clos}{\mathrm{clo}}
\newcommand{\mps}{\Wms}
\newcommand{\mpsmp}{\beta}
\newcommand{\Range}{\wh{\mathcal R}}
\newcommand{\Space}{{\mathcal R}}

\newcommand{\locrate}{\bar{r}}
\newcommand{\poiss}{\mathbf{N}}
\newcommand{\x}{\xi}		
\newcommand{\rate}{r}
\newcommand{\poissv}[1]{_{#1}}		
\newcommand{\stpara}[1]{_{#1}}					
\newcommand{\gvpara}[2]{^{#1,#2}}	

\newcommand{\defeq}{:=}
\newcommand{\skipLine}{\vspace{12pt}}
\newcommand{\jmp}[2]{{\rm Disc}_{#2}\left(#1\right)} 
\newcommand{\jmps}{\Jmc}
\newcommand{\ov}{\overline}

\newcommand{\te}{\text}

\newcommand{\indic}[1]{\mb{I}_{\left\{#1\right\}}}
\newcommand{\indi}[1]{\mb{I}_{#1}}	
\newcommand{\ex}[1]{\mb{E}\left[#1\right]}		
\newcommand{\deq}{\overset{\text{(d)}}{=}}			
\newcommand{\wh}[1]{\widehat{#1}}
\newcommand{\rpp}{P}
\newcommand{\delt}{\Delta}
\newcommand{\gemp}{\pi}	
\newcommand{\mb}{\mathbb}

\newcommand{\va}{R}
\newcommand{\Oo}{\mathcal{S}}
\newcommand{\Cc}{\mathcal{A}}
\newcommand{\minset}{\Mmc^{\Psi}}
\newcommand{\alt}[1]{\widetilde{#1}}
\newcommand{\law}{\mathcal{L}}
\newcommand{\ind}{\hspace{24pt}}
\newcommand{\Map}{\Theta}

\newcommand{\Fmb}{{\mathbb{F}}}

\newcommand{\Hmb}{{\mathbb{H}}}

\newcommand{\Nmb}{{\mathbb{N}}}

\newcommand{\Pmb}{{\mathbb{P}}}

\newcommand{\Vmb}{{\mathbb{V}}}

\newcommand{\Bmc}{{\mathcal{B}}}

\newcommand{\Dmc}{{\mathcal{D}}}
\newcommand{\Emc}{{\mathcal{E}}}

\newcommand{\Hmc}{{\mathcal{H}}}
\newcommand{\Imc}{{\mathcal{I}}}
\newcommand{\Jmc}{{\mathcal{J}}}
\newcommand{\Kmc}{{\mathcal{K}}}

\newcommand{\Mmc}{{\mathcal{M}}}
\newcommand{\Nmc}{{\mathcal{N}}}

\newcommand{\Pmc}{{\mathcal{P}}}

\newcommand{\Smc}{{\mathcal{S}}}
\newcommand{\Tmc}{{\mathcal{T}}}

\newcommand{\Zmc}{{\mathcal{Z}}}

\newcommand{\Gms}{{\mathscr{G}}}

\newcommand{\Nms}{{\mathcal M}_{\Nb}}

\newcommand{\Wms}{{\mathscr{W}}}

\allowdisplaybreaks
\title{Hydrodynamic Limits of non-Markovian Interacting Particle Systems on Sparse Graphs}
\author{Ankan Ganguly* and Kavita Ramanan**}
    
\begin{document}

\maketitle

\abstract{Consider an interacting particle system indexed by the vertices of a (possibly random) locally finite graph whose vertices and edges are equipped with weights or marks that represent parameters of the model, such as the environment and initial conditions. Each particle takes values in a countable state space and evolves according to a pure jump process whose jump rates depend only on its own state (or history) and marks, and states (or histories) and marks of particles and edges in its neighborhood. Under mild conditions on the jump rates, it is shown that if the sequence of (marked) interaction graphs converges in probability in the local weak sense to a limit (marked) graph that satisfies a certain finite dissociability property, then the corresponding sequence of empirical measures of the particle trajectories converges weakly to the law of the marginal dynamics at the root vertex of the limit graph. The proof of this hydrodynamic limit relies on several auxiliary results of potentially independent interest. First, such interacting particle systems are shown to be well-posed on (almost surely) finitely dissociable graphs, which include graphs with uniformly bounded maximum degrees and any Galton-Watson tree whose offspring distribution has a finite first moment. A counterexample is also provided to show that well-posedness can fail for dynamics on graphs outside this class. Next, given any sequence of graphs that converges in the local weak sense to a finitely dissociable graph, it is shown that the corresponding sequence of jump processes also converges in the same sense to a jump process on the limit graph. Finally, the dynamics are also shown to exhibit an (annealed) asymptotic correlation decay property. These results complement recent work on hydrodynamic limits of locally interacting probabilistic cellular automata and diffusions on sparse random graphs. However, the analysis of jump processes requires very different techniques, including percolation arguments and notions such as consistent spatial localization and causal chains.}

\blfootnote{Both authors were supported by the United States Army Research Office under grant number W911NF2010133, and the second author was also supported 
by the Office of Naval Research under
the Vannevar Bush Faculty Fellowship N0014-21-1-2887.}
\blfootnote{*Department of Mathematics and Statistics, Boston University. Email: \href{mailto:ankang@bu.edu}{ankang@bu.edu}}
\blfootnote{**Division of Applied Mathematics, Brown University. Email: \href{mailto:kavita\_ramanan@brown.edu}{kavita\_ramanan@brown.edu}}

\setcounter{tocdepth}{1}
\tableofcontents

\section{Introduction} 
\label{intro} 
\subsection{Motivation and Description of Results} 
We establish hydrodynamic limits, namely limits of empirical measures, of a general class of interacting particle systems (IPS) on large sparse graphs. The IPS we consider are pure jump processes that take values in a countable state space and describe the evolution of the states of particles indexed by the vertices of a (possibly random) locally finite graph $G$ that encodes the interaction structure between particles. Specifically, the jump rates of each particle at any given time depend only on its own state (or history), the states (or histories) of neighboring particles in the graph $G$, and possibly a random environment governed by vertex and edge weights in the particle neighborhood. Such IPS describe phenomena in a wide variety of fields including statistical physics \cite{MosSly13}, epidemiology \cite{JanLucWin14, Pem92, Bhaetal21,CocRam23}, neuroscience \cite{Tru16}, social science \cite{Lig85, CarTorMig16,HuoDur19}, engineering, and operations research \cite{AghRam19,BudMukWu19,Ram22a}. 

As a concrete example, consider the SIR model, which is an idealized stochastic model of the spread of disease through a population that serves as the basis for several more complex epidemiological models. In this model, each particle lies in one of three states, S, I, or R, that indicate that the particle is healthy but susceptible to the disease (S), infected (I), or recovered from and immune to the disease (R). Without loss of generality, we identify S, I, and R with the integers $0, 1$, and $2$, respectively. Given a finite (deterministic) graph $G = (V,E)$ that represents the social contact network of a population and measurable functions $\lambda$ and $\rho$ that represent (possibly time-varying) infection and recovery rates, the SIR model on $G$ is a Feller process $X^G$ where for each $v \in V$, the coordinate process $X^G_v$ has state space ${\mathcal X} = \{0,1,2\}$, and jumps of only size $1$, whose rates at time $t$ given by 
\begin{equation} 
\label{intro:contrates} 
\begin{array}{ll} 0\to 1 & \te{ at rate } \quad \lambda(t)\sum_{u:\{u,v\} \in E} \indic{X^G_u(t)= 1}, \\ 
1 \to 2 & \te{ at rate } \quad \rho(t). 
\end{array} 
\end{equation} 
Note that the rate at which a particle gets infected depends on the states of its neighbors, but the rate at which it recovers does not. In order to capture heterogeneity in the interactions, the graph could additionally carry vertex marks $\kappa_v\in \Rb_+, v \in V,$ and edge marks $\bar{\kappa}_e\in \Rb_+, e \in E$. In the context of the SIR model, for example, the vertex marks could capture an individual's susceptibility to infection, and the edge weights could reflect the frequency of interaction between pairs of individuals in a social contact network. The jump rates of the particle $v$ would then depend on these weights and could, for instance, take the form: 
\begin{equation} 
\label{intro:contrates2} 
\begin{array}{ll} 0\to 1 & \te{ at rate } \quad \lambda (t)\sum_{u:\{u,v\} \in E} \bar{\kappa}_{uv} \indic{X^G_u(t)= 1}, \\ 
1 \to 2 & \te{ at rate } \quad \rho(t) \kappa_v. 
\end{array} 
\end{equation} 

By broadening the above framework further to allow the jump rates of a particle to depend on its past evolution, it is also possible to describe a non-Markovian SIR process with general, non-exponential recovery times (see Example \ref{ex:NMcont}). In any of the above models, the definition of the process can be extended to any (almost surely) finite marked random graph $G$ in a natural way by evolving the dynamics according to \eqref{intro:contrates} for each realization of the (marked) random graph.

The basic SIR model has been well studied when $G_n$ is the complete graph on $n$ vertices. The infection rate is constant and scaled as $\lambda_n(t) = \lambda(t)/n$ to keep the net influence of all particles on any fixed particle of order one. In this setting, for any vertex $v_n$ chosen uniformly at random from $G_n$, the asymptotic dynamics of $X^{G_n}_{v_n}$, in the limit as $n\to\infty$, can be described by classical mean-field theory \cite[Theorem 2]{Oel84}. In particular, under mild (exchangeability) conditions on the initial states of the IPS, neighboring particles become asymptotically independent (a phenomenon referred to as propagation of chaos), and the sequence of random empirical measures of the trajectories converges to a deterministic limit as $n \rightarrow \infty$, 
\begin{equation} 
\label{intro:MFhydro} 
\gemp^{G_n,X^{G_n}}\defeq \frac{1}{n}\sum_{v\in V_{G_n}} \delta_{X^{G_n}_v} \to \law(X^*) \te{ in probability}, 
\end{equation} 
where $\delta_y$ represents the Dirac delta measure at the element $y$, $X^*$ is a nonlinear jump Markov process, referred to as the mean-field limit, that describes the limiting evolution of a typical particle in the graph, and $\law(X^*)$ denotes its law on the space of c\`{a}dl\`{a}g functions. At time $t$, $X^*$ transitions from $0$ to $1$ at rate $\lambda(t) \PP(X^*(t) = 1)$, and from $1$ to $2$ at rate $\rho(t)$, and thus the evolution of $\law(X^*)$ can be described by a coupled system of nonlinear ordinary differential equations (ODEs). In the time-homogeneous case, the latter ODE coincides with the deterministic SIR model introduced by Kermack and McKendrick in 1927 to describe the macroscopic evolution of diseases in a population \cite{KerMcK27}. Recent work has established corresponding propagation of chaos and convergence results to (adjusted) mean field limits for a class of IPS, including the time-homogeneous SIR model on sufficiently dense graph sequences $\{G_n\}$ (see \cite{BayWu21,DunCai21}).

The focus of this article is on the complementary case when the underlying graph is truly sparse (i.e., with uniformly bounded average degree), which is often a more realistic model of real-world networks. In this regime, the rates are not scaled (and the net influence of all particles on any one particle still remains of order one), so neighboring particles continue to exert a strong influence on each other and do not become independent even in the limit as the graph size goes to infinity. Hence, the limit depends on the topology of the interaction graph, and the techniques used to establish mean-field limits are no longer applicable. Specifically, one cannot expect the empirical measure process to converge just by sending the number of particles to infinity. Instead, the notion of local convergence of sparse rooted graphs\footnote{a root is a distinguished vertex of the graph} introduced by Benjamini and Schramm \cite{BenSch01} serves as a natural alternative mode of convergence that respects the graph topology; see Section \ref{nota:LWC} for a precise definition.

We now summarize our main results, which apply to a general class of (possibly non-Markovian) IPS on countable state spaces evolving in a random environment. Let $(G,\ems,\vms,X^G(0))$ denote the graph $G$ marked respectively with (possibly random) edge and vertex weights and initial conditions. Suppose we are given a collection of (possibly history-dependent) jump rates, and a sequence $\{(G_n,\ems^n,\vms^n,X^{G_n}(0))\}_{n \in \Nb}$ of finite marked graphs that converge locally weakly (see Definition \ref{def:LWC}) to $(G,\ems,\vms,X^G(0))$. Then, we establish the following: 

\begin{description} 
\item[Result 1: ] If the jump rates are predictable, satisfy a basic consistency property (see \eqref{eq:standing}) and a mild boundedness condition (see Assumption \ref{mod:assu}), and the graph $G$ is almost surely \emph{finitely dissociable}, then the IPS $X^G$ on $G$ with those jump rates is well defined (Theorem \ref{WP:WP}). We also provide a counterexample (see Appendix \ref{illpf}) to demonstrate that when the graph $G$ is not finitely dissociable, there may be multiple solutions to the IPS associated with the same jump rates. 

Finite dissociability is a percolation condition on the graph, which we show is satisfied by Galton Watson (GW) trees whose offspring distributions have finite mean, \emph{unimodular Galton-Watson} (UGW) trees whose offspring distributions have finite variance, and graphs with uniformly bounded maximum degree (see Proposition \ref{find:GWtree}, Corollary \ref{find:UGWpf}, and Proposition \ref{find:finite}). UGW trees are of particular interest since they arise naturally as local limits of many random graph sequences such as Erd\"os-R\'enyi graphs and configuration models (see Examples \ref{ex:ER}-\ref{ex:RR} and \cite[Theorems 3.12 and 3.15]{Bor16}). 

\item[Result 2: ] If the jump rates satisfy an additional mild continuity condition (Assumption \ref{mod:cont}), then the sequence $\{(G_n, X^{G_n})\}_{n \in \Nb}$ of graphs marked with the trajectories of the IPS converges locally weakly to $(G,X^G)$, the limit graph $G$ marked with the trajectories $X^G$ of the IPS on $G$ (see Theorem \ref{LWC:LWC}). 

It is worth pointing out that Assumption \ref{mod:cont} is trivially satisfied if the IPS is Markov. 

\item[Result 3: ] If $\{(G_n,\ems^n,\vms^n,X^{G_n}(0))\}_{n \in \Nb}$ converges to $(G,\ems,\vms,X^G(0))$ in a slightly stronger sense, namely locally in probability (see Definition \ref{def-inprob}), then the sequence of empirical measures of neighborhoods of vertices, marked with the corresponding IPS trajectories, converges to $\law (X_{\root \cup N_\root}^G)$, the law of the marginal of the IPS on the root $\root$ of the limit graph $G$ and its neighborhood $N_\root$. This implies the following hydrodynamic limit for the root marginal (see Theorem \ref{GEM:GEMconv} and Corollary \ref{GEM:empconv}): 
\begin{equation} 
\label{intro:EMPConv} 
\gemp^{G_n,X^{G_n}} \defeq \frac{1}{|V_{G_n}|}\sum_{v \in V_{G_n}} \delta_{X^{G_n}_v} \to \law(X^G_{\root})\te{ in probability}. 
\end{equation} 
Identification of the more general root neighborhood convergence allows one to capture the limiting dependence structure between neighboring vertices. 

Furthermore, when $G$ is a regular tree, an autonomous description of $\law(X^G_{\root \cup N_\root})$ is obtained in \cite[Chapter 6]{Gan22}; see also \cite{GanRam-LETDet24}. Descriptions of marginal dynamics on more general random trees are provided in forthcoming work. As in the mean-field case, the marginal dynamics are described by a nonlinear process (whose evolution depends on its law), but unlike in the mean-field case, it is in general a {\em non-Markovian} nonlinear process even when the dynamics on the full tree $G$ are Markov. However, for a class of IPS on UGW trees that includes the SIR models introduced above, it is shown in \cite{CocRam23} (see also references therein) that the root neighborhood marginal process evolves according to a {\em Markovian} nonlinear process, whose law is characterized by a system of coupled nonlinear ODEs, which differ from the mean-field or Kermack-McKendrick ODE. A comparison of these two different ODEs and an analysis of the former to characterize the dependence of the outbreak size on the topology of the graph can be found in \cite[Theorem 3.1]{CocRam23}. 
\end{description}

In addition to the SIR models introduced above and variants such as SEIR models, the class of IPS for which Results 1-3 hold includes commonly studied IPS such as the contact process, the voter model and its variants, Glauber dynamics for the Ising and Potts models as well as their non-Markovian analogs (see Section \ref{mod:examples}). Our framework can also be extended to cover models with directed interactions such as (non-Markovian) neuronal Hawkes models (see Remark \ref{LWC:directed}), and some of our intermediate results hold in even greater generality than the main results. Our results, however, do not cover IPS, such as the exclusion process, in which multiple particles jump simultaneously. This is addressed in forthcoming work \cite{RamYas24}, which also establishes large deviation principles for such IPS.

\subsection{Comments on the Proofs and Comparison with Prior Work} 
We start by discussing the proof of well-posedness (Result 1). Although several recent works studying IPS on random graphs provide intuitive descriptions of IPS on random graphs \cite{NamNguSly22,Bhaetal21,JanLucWin14,PemSta01,HuaDur20}, there appears to be no general result that rigorously establishes well-posedness of even Markovian IPS on a general class of random graphs. While well-posedness of IPS on finite graphs is standard under our assumptions, on infinite graphs, the issue is more subtle and, as illustrated by the simple example in Appendix \ref{illpf}, well-posedness can, in fact, fail to hold for even Markovian IPS. Previous well-posedness results for IPS on infinite graphs have almost exclusively focused on graphs with uniformly bounded maximum degrees. For example, on lattices, an analytical proof of well-posedness of a large class of Feller IPS via examination of their semigroups can be found in the seminal paper of Liggett \cite{Lig72} (see also \cite{Lig85}), and a probabilistic proof of well-posedness of IPS with nearest-neighbor interactions using percolation arguments can be found in the classical work of Harris \cite{Har72}. The latter argument can be extended to locally interacting IPS on any translation invariant graph but crucially relies on the graph having a uniformly bounded maximum degree. Another approach to well-posedness involves a standard Picard iteration argument applied to the (jump) stochastic differential equation (SDE) representation of the IPS dynamics; see \eqref{mod:infpart}. This approach is effective when the jump rates of any individual particle satisfy a strong Lipschitz continuity property, that is, when they are uniformly Lipschitz with respect to the state (or trajectory in the non-Markovian setting) of each of the neighboring particles, with the (single-neighbor) Lipschitz coefficient being \emph{inversely proportional} to the degree of the vertex (see also \cite{DelFouHof16} for a slightly weaker averaged version of this Lipschitz condition). However, for even standard Markovian IPS such as the abovementioned SIR process, the Lipschitz constants of the jump rates with respect to the states (or trajectories and marks) of each neighboring particle do not decrease with the degree of the vertex of the particle, but remain of the same order. In particular, the strong Lipschitz continuity property of jump rates does not hold on infinite (random) graphs that have unbounded maximum degree, such as GW trees with Poisson offspring distributions. Nevertheless, we are able to establish strong well-posedness under a mild boundedness condition on the jump rates.

Our proof of strong well-posedness of the jump stochastic differential equation (SDE) associated with the IPS (see \eqref{mod:infpart} and Definition \ref{mod:WP}) consists of three main ingredients. First, we introduce the notion of {\em spatial localization} of the IPS dynamics (see Definition \ref{WP:locunif} and Figure \ref{fig-Spatloc}). Roughly speaking, a graph $G$ is said to spatially localize an IPS with given jump rates if given any $T < \infty$ and a finite subset $\sO$ of the vertices, there exists a (possibly random) almost surely finite set $U$ containing $\sO$ such that on the interval $[0,T]$ the marginal evolution of the IPS on $\sO$ is not influenced by the evolution of the IPS outside the larger set $U$. Invoking strong well-posedness of the IPS on (almost surely) finite random graphs, we then conclude strong well-posedness of the IPS on any graph $G$ that spatially localizes the IPS SDE (see Proposition \ref{WP:WP2}). Second, under a mild boundedness condition on the jump rates (Assumption \ref{mod:assu}), we show that the IPS SDE is spatially localized by any finitely dissociable graph (see Proposition \ref{WP:locality2}). This proof entails the analysis of so-called causal chains that capture the propagation of influence of the IPS dynamics from a vertex (see Section \ref{perc:pfwp}). Finally, we introduce and analyze a certain (inhomogeneous) site percolation to show that GW trees and graphs of bounded maximal degree are almost surely finitely dissociable (see Section \ref{perc:findis}). To the best of our knowledge, the only other work that proves well-posedness of a (jump) IPS on a graph with unbounded maximal degree appears to be the recent work of Gantert and Schmidt \cite{GanSch20}, which establishes well-posedness of the simple exclusion process on a GW tree whose offspring distribution has finite mean by crucially exploiting the special structure of the exclusion process to reduce the problem to the study of a standard bond percolation problem. Our result does not subsume that of \cite{GanSch20} but is applicable to a wide class of possibly non-Markovian models and does not rely on specific features of the IPS. 

Our next result (Result 2) on local weak convergence of the dynamics is proved via coupling arguments that entail establishing a certain consistent spatial localization property of the sequence of interaction graphs (see Definition \ref{WP:loccons}), which requires a more careful analysis of causal chains and their behavior under isomorphisms of the graph. Lastly, our proof of Result 3 involves establishing an asymptotic spatial decay of correlations of the trajectories of the IPS that is annealed (or averaged over the randomness of the graph). Specifically, in Theorem \ref{GEM:deccor}, we show that although neighboring vertices remain strongly correlated for sparse graph sequences (in contrast to dense graph sequences), finite neighborhoods of two independent randomly chosen vertices become asymptotically independent as the number of particles goes to infinity. The proof of this asymptotic correlation decay property involves suitable coupling arguments and also exploits the local convergence result of Result 2. Along the way, in several of the proofs, to avoid working with more cumbersome isomorphism classes of graphs (in terms of which local weak convergence is defined), we also introduce (in Appendix \ref{msbl}) an "equivalent" space of measurable representative graphs equipped with a topology that is compatible with local weak convergence (see Appendix \ref{msbl}). The latter result may be of independent interest. 

The present article complements recent work by Oliveira et al. \cite{OliReiSto20}, which establishes local convergence of interacting diffusions with (possibly random) pairwise interactions on locally convergent sequences of finite graphs, and the works of Lacker et al. \cite{LacRamWu23,LacRamWuLE23,Ram22b} which establish hydrodynamic limits for homogeneously interacting cellular automata and diffusions with general (not necessarily pairwise) symmetric interactions. The hydrodynamic limit in \cite{LacRamWu23} is also shown by first establishing local weak convergence and then asymptotic correlation decay, but the proofs of these results rely crucially on the previously mentioned strong Lipschitz continuity conditions on the drift and diffusion coefficients, which, though reasonable for interacting diffusions, exclude many interesting classes of IPS. Our weaker assumptions change the nature of the correlation decay established in comparison with the diffusion setting (see the discussion in Section \ref{LWC:CPLWS} for an elaboration of this point).

\subsection{Organization of the Rest of the Paper} 
The article is organized as follows. In Section \ref{nota}, we introduce common notation that is used throughout the paper and also provide examples of locally converging graphs in Section 2.4.2. In Section \ref{mod}, we introduce the class of IPS we consider and its SDE formulation, state the basic assumptions on the jump rates, and properly define notions of strong and weak solutions for IPS on random graphs. Examples of IPS that lie within our framework are presented in Section \ref{mod:examples}. The main results are stated in Section \ref{Res}, and the ramifications of our results for our running example of the SIR model are discussed in Section \ref{contact-motivating}. The rest of the article is devoted to proofs of the main results: Section \ref{perc} introduces the notions of (consistent) spatial localization, causal chains, and finite dissociability and contains the proof of well-posedness (Result 1); local weak convergence of IPS (Result 2) is proved in Section \ref{LWCpf}; asymptotic correlation decay is established in Section \ref{GEM}, and the hydrodynamic limit (Result 3) is deduced from it in Section \ref{GEM:lim}. Appendix \ref{illpf} contains an example of a simple IPS that fails to be well-posed. Appendix \ref{msbl} contains auxiliary technical results related to canonical measurable representatives of random (marked) graph isomorphism classes. Appendix \ref{aver} presents a few useful technical results: Appendix \ref{saver} presents a generalization of our main well-posedness result (Theorem \ref{WP:WP}) to heterogeneous IPS, Appendix \ref{cond} includes the proof of a technical result (Lemma \ref{mod:condwp}) describing how to establish the well-posedness of IPS whose initial conditions are random, and the examples of Section \ref{mod:examples} are verified in Appendix \ref{varex}. In Appendix \ref{WPfin}, we include (for completeness) a simple proof of the strong well-posedness of IPS on finite graphs.

\section{Preliminaries and Notation} 
\label{nota} 
Let $\mathbb{R}$ denote the reals, let $\mathbb{Z}$ denote the integers, and let $\mathbb{N}_0 := \{0, 1, \ldots\}$ denote the nonnegative integers. For any set $A$, let $|A|$ denote its cardinality.

\subsection{Graph Notation} 
\label{nota:grph}
Given an undirected graph $G = (V,E)$ with vertex set $V$ and edge set $E$, for $v \in V$, let $\neigh{v} = \gneigh{v}{G}\defeq\{u \in V: \{u,v\} \in E\}$ denote the neighbors of $v$ in $G$ and let $\cl{v} = \gcl{v}{G} \defeq \{v\}\cup\neigh{v}$. For any $U \subseteq V$, set $\neigh{U} \defeq \cup_{v\in U} (\neigh{v}\setminus U)$ and $\cl{U} \defeq U\cup \neigh{U}$. For clarity, we may write $\cl{U}(G)$ to emphasize that the closure is taken with respect to edges in $G$. We define $\fset{G} \defeq \{U\subseteq V:|U| <\infty\}$ to be the set of finite subsets of the vertices in $G$. Recall that the degree of a vertex $v$ is equal to $|\neigh{v}|$. The graph $G$ is said to be locally finite if each of its vertices has a finite degree. We always assume graphs are simple (i.e., they do not have self-loops or multi-edges) and locally finite.

A graph $G = (V,E)$ equipped with a distinguished vertex $\root \in V$, denoted the root, is called a rooted graph and denoted by $(G,\root)\defeq (V,E,\root)$. When the root is clear from context, we simply write $G$ instead of $(G,\root)$. For $U \subseteq V$, we denote by $G\subg{U}$ the induced subgraph of $G$ on $U$, that is, $G\subg{U} = (U,E\subg{U})$ where $E\subg{U} = E\cap \{ \{x,y\}: x,y \in U \} $. For $u, v \in V$, a path between $u$ and $v$ in $G$ is defined to be a sequence of vertices $\Gamma = (u = v_0, v_1,\dots,v_{n-1},v_n = v)$ such that for all $i \in\{1,\dots,n\}$, $\{v_{i-1},v_i\} \in E$ and $v_i \neq v_j$ whenever $i \neq j$ except possibly when $ \{i,j\} = \{n,0\}$, in which case the path is said to be a cycle. A graph is said to be acyclic if it has no cycles. The length of the path, denoted $|\Gamma|$, is the number of edges in the path. We let $d_G(u,v)$ denote the usual graph distance, which is the length of the shortest path between $u$ and $v$ in $G$. When $G$ is a finite rooted graph, its radius is the maximum distance from any vertex to the root. Let $\gsone$ denote the set of rooted graphs of radius 1. \subsection{Configurations and Path Space Notation} \label{nota:skor} Given a Polish space $\Pol$ and $U \subseteq V$, we define the configuration space 
\begin{equation} 
\label{skor:prod} 
\Pol^U = \{(z_v)_{v \in U}: z_v \in \Pol\te{ for all } v \in U\} 
\end{equation} 
and equip it with the product topology. For any $z \in \Pol^V$, we write $z_U = (z_v)_{v\in U} \in \Pol^U$ to mean the restriction of $z$ to $\Pol^U$. Given two vertex sets $V_1$ and $V_2$, a map $\varphi:V_1 \to V_2$, a subset $U \subseteq V_1$, and configurations $x \in \Pol^{V_1}$, $y \in \Pol^{V_2}$, we write $x_U = y_{\varphi(U)}$ to mean $x_v = y_{\varphi(v)}$ for all $v \in U$. Vertex set indices are assumed to be ordered.

Let $\Xc$ denote the countable state space of the IPS, which we identify with a subset of $\mb{Z}$ and equip with the discrete topology. For any $U \subseteq V$ and $t \in (0,\infty]$, let $\cad^U_t\defeq \cad([0,t]; \Xc^U)$ (respectively, $\cad^U_{t-}\defeq \cad([0,t); \Xc^U)$) be the space of c\`adl\`ag functions from $[0,t]$ (respectively, $[0,t)$) to $\Xc^U$, equipped with the product J1 topology, which makes it a Polish space \cite[Section 11.5]{Whi02}. Also, let $\cad^U \defeq \cad^U_{\infty-}$ denote the space of c\`adl\`ag functions from $[0,\infty)$ to $\Xc^U$, equipped with the topology such that $x^n$ converges to $x$ in $\cad^U$ if and only if for each $t \in [0,\infty)$, the restriction of $x^n$ to $[0,t]$ converges to the restriction of $x$ to $[0,t]$ in $\cad_t^U$. When $|U| = 1$, we denote $\cad^U_t$ or $\cad^U$ simply by $\cad_t$ or $\cad$, respectively. If $x \in \cad^U$ and $v\in U$, then $x_v(t)$ denotes the value of the $v$th component of $x$ at time $t$. The restrictions of $x$ to $[0,t]$ and $[0,t)$ are denoted by $x[t] \in \cad^U_t$ and $x[t)\in \cad^U_{t-}$, respectively. For $0 \leq s\leq t\leq \infty$, $W\subseteq U\subseteq V$, and $x \in \cad^U_t$, let 
\begin{equation} 
\label{skor:disc} 
\begin{array}{rcll} \jmp{x_W}{s} & \defeq & \{s' \in [0,s]: x_W(s') \neq x_W(s'-)\} 
\end{array} 
\end{equation} 
denote the set of discontinuities of $x_W$. 

\subsection{Measure Notation and Point Processes} 
\label{nota:msr}
Given a Polish space $\Pol$, let $\borel(\Pol)$ be the Borel $\sigma$-algebra on $\Pol$, and let $\Pmc(\Pol)$ be the space of probability measures on $(\Pol,\borel(\Pol))$ equipped with the topology of weak convergence, that is, $\mu_n$ converges to $\mu$ weakly if and only if $\lim_{n\to\infty} \int_{\Pol} f d\mu_n = \int_{\Pol} f d \mu$ for every bounded, continuous function $f$ on $\Pol$. Given any $\gm\in \Pmc(\Pol)$ and $\Pol$-valued random elements $Z$ and $Y$, $\law(Z)$ is used to denote the distribution (equivalently, law) of $Z$, the notation $Z \sim \gm$ means $\law(Z) = \gm$, and $Y\deq Z$ denotes $\law(Y) = \law (Z)$. We additionally define $\Nms(\Pol)$ to be the space of locally finite, non-negative integer-valued measures on $(\Pol,\borel(\Pol))$. We equip $\Nms(\Pol)$ with the weak topology. As is well known, $\Pmc(\Pol)$ and $\Nms(\Pol)$ are Polish spaces (see \cite[Theorem 6.8]{Bil99} and \cite[Proposition 9.1.IV (iii)]{DalVer08}, respectively, as well as \cite{Mor18}). For any interval $I \subseteq \Rb$, probability measure $\gm \in \Pc(I)$, and $[a, b) \subset I$ or $(a,b) \subset I$, we write $\gm[a,b)$ and $\gm(a,b)$ for $\gm([a,b))$ and $\gm((a,b))$, respectively.

A random element $\rpp$ taking values in $\Nms(\Pol)$ is called a point process on $\Pol$. Given any point process $P$ on $\Pol$, for every compact set $K \subseteq \Pol$, there exists an almost surely finite set of points $\{z_i\}_{i=1}^N \subseteq K$, referred to as \emph{events}, such that $\rpp(\{z_i\}) > 0$ for all $i=1,\dots, N$, and $\rpp(K \setminus \{z_i\}_{i=1}^N) = 0$. In this paper, we assume all point processes are \emph{simple}, that is, $\sup_{z \in \Pol} \rpp(\{z\}) \in \{0,1\}$. Given any measure $\gm$ on $\Pol$ that is finite on each compact set $K \in\borel(\Pol)$, a \emph{Poisson point process} on $\Pol$ with \emph{intensity measure} $\gm$ is a point process $\rpp$ such that for any disjoint sets $A, B\in \borel(\Pol)$, $\rpp(A)$ and $\rpp(B)$ are independent and $\ex{\rpp(A)} = \gm(A)$.

We work with point processes equipped with a time component. Let $\hat{\Pol} \defeq I\times \Pol$, where $I \subseteq \Rb$ is an interval and $\Pol$ is a Polish space. We refer to a point process $\rpp$ on $\hat{\Pol}$ as a \emph{marked} point process on $I$ with marks in $\Pol$. If $\rpp$ has events $\{(t_i,\kappa_i)\}_{i=1}^N$, then we call $\{\kappa_i\}_{i=1}^N$ the \emph{marks} of $\rpp$. We say a marked point process $\rpp$ on $I$ defined on the filtered probability space $(\Omega, {\mathcal H}, \hfilt = \{\hfiltm_t\}_{t \in I}, \mathbb{P})$ is $\hfilt$-adapted if for every $t\in I$ and $A \in \borel([0,t] \cap I \times \Pol)$, $\rpp(A)$ is $\hfiltm_t$-measurable. Furthermore, an $\hfilt$-adapted marked Poisson point process $\rpp$ on $I$ with marks in $\Pol$ is said to be an $\hfilt$-Poisson marked point process if for every $t\in I$ and $A \in \borel\left((t,\infty)\cap I \times \Pol\right)$, $\rpp(A)$ is independent of $\hfiltm_t$. Such point processes are used to describe the noise driving the IPS.

\subsection{Local Convergence} 
\label{nota:LWC}
\subsubsection{Definitions} 
\label{nota:LWCdef} 
Since we represent our IPS as marked graphs we briefly review the notions of local convergence of graphs and marked graphs, which were introduced in \cite{BenSch01}. Let $G_i= (V_i,E_i), i=1,2,$ be (unrooted) graphs. A mapping $\varphi: V_1 \to V_2$ is said to be an \emph{isomorphism} from $G_1$ to $G_2$ if it is a bijection and $e = \{u,v\} \in E_1$ if and only if $\varphi(e)\defeq \{\varphi(u),\varphi(v)\} \in E_2$. Given roots $\root_i \in V_i,i=1,2,$ $\varphi$ is an isomorphism from the rooted graph $(G_1,\root_1)$ to the rooted graph $(G_2,\root_2)$ if, in addition, $\varphi(\root_1) = \root_2$. Recall that when denoting the rooted graph, we often omit the explicit dependence on the root. Given rooted graphs $G_1$ and $G_2$, let $I(G_1,G_2)$ denote the collection of isomorphisms from $G_1$ to $G_2$. If $I(G_1,G_2)$ is non-empty, then $G_1$ and $G_2$ are said to be isomorphic, which is denoted $G_1 \cong G_2$. Let $\gs$ be the space of isomorphism classes of connected, locally finite, rooted graphs. Then, for any connected, locally finite rooted graph $G$, we let $\ic{G} \in \gs$ denote the isomorphism class of $G$, namely $\ic{G}$ is the collection of connected locally finite rooted graphs isomorphic to $G$. Conversely, we refer to $G$ as a \emph{representative graph} of $\ic{G}$. Clearly, if $H\cong G$ then $H \in \ic{G}$. For each $m \in \Nb$, let $\trnc{m}(G)$ be the induced subgraph of $G$ consisting of all vertices within (graph) distance $m$ of the root. We equip $\gs$ with the topology of \emph{local convergence} in which $\{\ic{G_n}\}_{n \in \Nb} \subset \gs$ is said to converge locally to $\ic{G}\in \gs$ if for every $m \in \Nb$, there exists $n_m < \infty$ such that $\trnc{m}(G_n) \cong \trnc{m}(G)$ for all $n \geq n_m$, $G_n\in \ic{G_n}$, and $G\in\ic{G}$.

Next, fix any two Polish spaces $\emksp$ and $\vmksp$ that represent the edge and vertex mark spaces, respectively, and consider a (not necessarily connected) marked rooted graph $G = (V,E,\root,\ems,\vms)$, where $(V,E,\root)$ is a rooted graph, $\ems \in \emksp^E$, and $\vms \in \vmksp^V$. Then $G$ is a \emph{$\sp{\emksp,\vmksp}$-graph}. Unless explicitly mentioned otherwise, a $\sp{\emksp,\vmksp}$-graph is assumed to be rooted. Also, let $\nm{G_*} = (V,E,\root)$ denote the $\sp{\emksp,\vmksp}$-graph $G$ with its marks removed, and let $\nm{G}$ denote $G$ with its root and marks removed. For $m \in \Nb$, let $\trnc{m}(G)$ be the induced marked rooted subgraph of $G$ consisting of all vertices within (graph) distance $m$ of the root, equipped with the same marks and root. We slightly abuse notation at times by allowing $\trnc{m}(G)$ to also denote the set of vertices within graph distance $m$ of the root. We say the marked graphs $G \defeq (V,E,\root,\ems,\vms)$ and $G^\prime \defeq (V^\prime,E^\prime,\root^\prime,\ems^\prime,\vms^\prime)$ with edge and vertex marks in $\emksp$ and $\vmksp$, respectively, are isomorphic, and write $G \cong G^\prime$ if there exists an isomorphism $\varphi \in I(\nm{G_*},\nm{G^\prime_*})$ such that $\vms_v = \vms^\prime_{\varphi(v)}$ and $\ems_{e} = \ems^\prime_{\varphi(e)}$ for all $v \in V$ and $e \in E$. We let $I(G,G')$ denote the set of isomorphisms between $G$ and $G'$, and let $\gs\sp{\emksp,\vmksp}$ denote the collection of isomorphism classes of graphs with edge and vertex marks in $\emksp$ and $\vmksp$, respectively. Once again, for any such marked graph $G$, $\ic{G} \in \gs\sp{\emksp,\vmksp}$ denotes the isomorphism class of $G$. Likewise, for any $\ic{H} \in \gs\sp{\emksp,\vmksp}$, the marked graph $H$ (with edge and vertex marks in $\emksp$ and $\vmksp$, respectively) denotes an arbitrary representative of $\ic{H}$. Also, given a (possibly marked or rooted) graph $G$, we let $V_G$ and $E_G$, respectively, denote the vertex and edge sets of $G$. We also occasionally abuse notation by letting $G$ denote its own vertex set.

We equip $\gs\sp{\emksp,\vmksp}$ with the topology of local convergence, defined as follows:

\begin{definition}[Local convergence] 
\label{def-locconvnoiso} 
\sloppy The sequence $\ic{G_n} \in \gs\sp{\emksp,\vmksp}$, $n \in \Nb$, is said to converge locally to $\ic{G} \in \gs\sp{\emksp,\vmksp}$ if for every sequence of representatives $G_n = (V_n,E_n,\root_n,\ems^n,\vms^n)$, $n \in \Nb$, and $G = (V,E,\root,\ems,\vms)$, and for every $m \in \Nb$ there exists $n_m < \infty$ and a sequence $\varphi_{n,m} \in I(\trnc{m}(\nm{G_{*}}), \trnc{m}(\nm{G_{n,*}})), n > n_m, n \in \Nb,$ such that for every $v \in V_{\trnc{m}(\nm{G_{*}})}$ and $e \in E_{\trnc{m}(\nm{G_{*}})}$, $\vms^n_{\varphi_{n,m}(v)} \to \vms_v$ and $\ems^{n}_{\varphi_{n,m}(e)} \to \ems_e$ as $n \rightarrow \infty$. 
\end{definition}

We also let $\gsonenew\sp{\emksp,\vmksp} \subset \gs\sp{\emksp,\vmksp}$ denote the (closed) space of isomorphism classes of graphs in $\gsone$ with edge and vertex marks in $\emksp$ and $\vmksp$, respectively, equipped with the topology induced by $\gs\sp{\emksp,\vmksp}$. Note that one can also view $\gs = \gs\sp{\{1\},\{1\}}$ as a space of (isomorphism class of) marked graphs with trivial marks, that is, when both mark spaces are equal to the trivial Polish space $\{1\}$, in which case the local convergence defined above coincides with the notion defined earlier on $\gs$. It is well known that $\gs\sp{\vmksp,\emksp}$ and hence, $\gsonenew\sp{\emksp,\vmksp}$ and $\gs$ (equipped with the topology of local convergence), are Polish spaces (see \cite[Lemma 3.4]{Bor16}). 

\begin{definition}[Local weak convergence] 
\label{def:LWC} 
If $\{\gm_n\}_n \subset \Pc(\gs\sp{\emksp,\vmksp})$ converges in distribution to $\gm$, then it is said that $\gm_n$ converges to $\gm$ \emph{in distribution in the local weak sense}, denoted $\gm_n \Rightarrow \gm$. For notational conciseness, we often refer to this convergence as \emph{local weak convergence}, or alternatively, say that $\gm_n$ converges to $\gm$ \emph{locally weakly}. Furthermore, if $G_n \sim \gm_n$ for every $n \in \Nb$ and $G \sim \gm$, we also write $G_n$ $\Rightarrow$ $G$ to denote that $G_n$ converges to $G$ locally weakly. Lastly, we refer to the topology on $\Pc(\gs\sp{\emksp,\vmksp})$ induced by local weak convergence as the \emph{local weak topology}. 
\end{definition}

We refer the reader to \cite{Bor16}, \cite{Hof24}, and \cite[Appendix A]{LacRamWu23} for general results on local weak convergence and Section \ref{grph:ex} below for examples.

\begin{remark}[IPS on Directed graphs] 
\label{LWC:directed} 
While the above discussion focused on the local convergence of undirected graphs, there exist several frameworks for working with the local convergence of directed graphs. One possible approach is to define a space of isomorphism classes of marked directed graphs in which all isomorphisms additionally respect edge orientation (in the spirit of \cite[Exercise 2.17]{Hof24}). Using an argument similar to that used in the proof of \cite[Lemma 3.4]{Bor16}), it is possible to then show that this space is Polish. Furthermore, one can construct a random map that ``lifts'' marked directed graphs from this new space to marked undirected graphs in a suitably continuous way. This construction can be combined with the results of this paper to obtain convergence and hydrodynamic limit results for a large class of IPS on directed sparse graphs, including neuronal Hawkes processes \cite{Tru16}. A fully rigorous justification is somewhat technical, and hence omitted from this work. 
\end{remark}

\subsubsection{Examples of graphs that converge locally} 
\label{grph:ex}
We provide a few examples of random graphs that converge locally weakly. Many of these random graph models converge to a unimodular Galton-Watson (UGW) tree, which we now define.

\begin{definition}[GW and UGW trees] 
\label{LWC:GWdef} 
Given a probability measure $\rho \in \Pc(\Nb_0)$, a random rooted tree $(\tree, \root)$ is a GW tree if the numbers of offspring of all vertices in $\tree$ are i.i.d. $\rho$-distributed random variables. If $\rho$ has a finite first moment, then the UGW tree (or size-biased GW tree), denoted UGW$(\rho)$, is defined in the same way, except that the offspring distributions of all non-root vertices are now given by $\wh{\rho}\in\Pc(\Nb_0)$, where 
\begin{equation} 
\label{def:hatrho} 
\widehat{\rho}(\{k\}) \defeq \frac{(k+1)\rho(\{k+1\})}{\sum_{n=0}^\infty n \rho(\{n\})}, \quad k \in \Nb_0. 
\end{equation}
\end{definition}

We now present several examples of standard sequences of random graphs that converge locally weakly. 

\begin{example}[Erd\"os-R\'enyi Graphs] 
\label{ex:ER} 
Suppose that $G_n$ is a sequence of Erd\"os-R\'enyi graphs with $n$ vertices and edge probability $p_n$ such that $np_n \to \theta \in (0,\infty)$ as $n \to \infty$. If each $G_n$ is equipped with a uniform random root vertex, and $\te{Poiss}(\theta)$ represents the Poisson distribution with parameter $\theta$, then $\ic{G_n}$ converges locally weakly to the $\te{UGW}(\te{Poiss}(\theta))$ tree. See \cite[Theorem 3.12]{Bor16} or \cite[Theorem 2.18]{Hof24} for proofs. 
\end{example}

\begin{example}[Configuration Models] 
\label{ex:CM} 
For each $n \in \Nb$, let $\mathbf{d}^n\defeq (d^n_1,\dots,d^n_n) \in \{0,\dots,n-1\}^n$ be a vector of non-negative integers such that $\sum_{k=1}^n d^n_k$ is even. Let $G_n$ be a uniformly chosen random multigraph with self-loops on $n$ edges with degree sequence $\mathbf{d}^n$ and a random root chosen uniformly at random from $G_n$. This can be encoded in our framework by equipping each edge of $G_n$ with a mark indicating the multiplicity of the edge and equipping each vertex with a mark that represents the number of self-loops at that vertex. Then $G_n$ is said to be a \emph{configuration model} with degree distribution $\mathbf{d}^n$. If the degree distributions $\rho^n\defeq \frac{1}{n}\sum_{k=1}^n \delta_{d^n_k}$ converge weakly to some probability measure $\rho \in \PP(\Nb_0)$ with a finite first moment and if the first moments of $\rho^n$ converge to the first moment of $\rho$, then $\ic{G_n}$ converges locally weakly to the $\te{UGW}(\rho)$ tree. See \cite[Theorem 3.15]{Bor16} or \cite[Theorem 4.1]{Hof24} for proofs. 
\end{example}

\begin{example}[Random Regular Graphs] 
\label{ex:RR} 
Fix any $d \in \Nb$. Let $G_n$ be uniformly distributed on the set of \emph{d-regular} graphs with uniform degree $d$ and a fixed vertex set of size $n$ (where we assume $nd$ is even for $G_n$ to be well defined). Let $G_n$ be equipped with a root chosen uniformly at random from its vertex set. Then $\ic{G_n}$ converges locally weakly to the infinite regular $d$-tree, or equivalently, the $\te{UGW}(\delta_d)$ tree. See (\cite[Theorem 2.17]{Hof24} for a proof. 
\end{example}

There are many other examples of random graph models that converge to random trees, such as preferential attachment graphs that converge to the P\'olya point tree \cite{Beretal14}.

\begin{example}[Graphs with I.I.D. Marks] 
\label{ex:markedconv} 
Suppose $\ic{G_n}\Rightarrow \ic{G}$ locally weakly (respectively, in probability) and $(G_n,\ems^n,\vms^n)$ and $(G,\ems,\vms)$ are equipped with i.i.d. edge and vertex marks such that for any $n \in \Nb$, $v \in V_{G_n}$, $v' \in V_G$, $e \in E_{G_n}$, and $e' \in E_G$, 
\[\vms^n_v \deq \vms_{v'}\te{ and } \ems^n_e \deq \ems_{e'}.\] 
Then it is easily shown that $\ic{(G_n,\ems^n,\vms^n)} \Rightarrow \ic{(G,\ems,\vms)}$ locally weakly (respectively, in probability). This is proven in the case in which edge marks are trivial in \cite[Corollary 2.16]{LacRamWu23}, but the inclusion of edge marks does not significantly alter the proof. 
\end{example}

\section{Model Description} 
\label{mod} 
In Section \ref{mod:dyn}, we introduce the IPS model and assumptions on the jump rates, and in Section \ref{mod:examples}, we provide several examples that satisfy our assumptions. In Section \ref{mod:wpdef}, we introduce the SDE, associated with the jump rates, that describes the IPS dynamics on any given graph, and in Section \ref{mod:WPsec}, we introduce notions of well-posedness of the SDE.

\subsection{A Standing Assumption on the Jump Rates} 
\label{mod:dyn} 
We consider IPS in which each particle takes values in a countable state space $\Xc$, which we identify with a subset of $\Zb$ and equip with the discrete topology. We let $\jmps \subseteq \{i - j: i,j \in \Xc, i \neq j\}$ denote the set of transitions or jump sizes of any particle that are allowable. Note that $\jmps$ may be a strict subset of $\{i-j:i,j\in\Xc,i\neq j\}$. For instance, in the SIR model from \eqref{intro:contrates}-\eqref{intro:contrates2}, $\Xc = \{0,1,2\}$ and $\jmps = \{1\}\subsetneq \{-2,-1,1,2\}$. In addition, let $\emksp$ and $\vmksp$ be two Polish spaces that serve as state spaces for respective edge and vertex marks that specify static parameters of the model such as random environments such as in \eqref{intro:contrates2} (see also Example \ref{ex:RE}), histories before time zero for non-Markovian processes (see Example \ref{ex:NMcont}), or heterogeneities and other graph attributes. To describe the model on any interaction graph, we need to specify the rate at which each particle makes a jump of any given size. At any time, this depends only on the particle's own state or history, the states or histories of neighboring particles in the interaction graph, and the environment marks on that vertex and the vertices and edges in its neighborhood. In particular, this jump rate depends on the vertex $v$ associated with the particle only via its neighborhood structure. This structure can vary from vertex to vertex but is always captured by a graph in the space $\gsone$ of finite rooted graphs $H = (V_H, E_H, \root_{H})$ of radius 1 (introduced in Section \ref{nota:grph}). Thus, the IPS model is completely specified by a family of {\em local jump rates} 
\[{\bf \locrate} := \{\locrate^{H}_j: \Rb_+ \times \cad^{V_H} \times \emksp^{E_H} \times \vmksp^{V_H} \rightarrow \Rb_+, H = (V_H, E_H, \root_H) \in \gsone, j \in \jmps\},\] 
where recall from Section \ref{nota:skor} that $\cad$ is the space of c\`{a}dl\`{a}g functions taking values in $\Xc$. Then given any marked $[\emksp,\vmksp]$-graph $G$, the dynamics of that IPS model on $G$ is governed by the jump rates 
\[{\bf \rate^G} := \{\rate\gvpara{G}{v}\stpara{j}: \Rb_+ \times \cad^{V_G} \to [0,\infty); v\in V_G, j \in \jmps\},\] 
with the rate $\rate\gvpara{G}{v}\stpara{j}(t,x)$ of a particle $v$ having a jump of size $j$, at a time $t$ when the configuration of particle trajectories is $x$, given by 
\begin{equation} 
\label{eq:standing} 
\rate\gvpara{G}{v}\stpara{j} (t,x) = \locrate^{H_v}_j(t,x_{H_v}, (\ems_{e})_{e \in E_{H_v}},(\vms_{u})_{u \in V_{H_v}}), \quad \mbox{ for every } (t,x) \in \Rb_+ \times \cad^V, 
\end{equation} 
where $H_v \defeq (G\subg{\cl{v}}, v)$ is the induced subgraph of $G$ on the closure $\cl{v}$ of $v$, with $v$ as its root. 

Throughout, we require that the local jump rates satisfy some mild regularity conditions specified in Definition \ref{def-regular} below. The first is a symmetry condition that essentially says that the rates are invariant to labelings of the neighboring vertices and are thus symmetric functions of the neighborhood states and marks. This ensures that solutions to the IPS possess well defined isomorphism classes (see Remark \ref{rem-classfn}). The second condition ensures that the jump rates only depend on the past histories of the particles. This condition is trivially satisfied for Markovian IPS. Examples of IPS models with regular local jump rates are given in Section \ref{mod:examples}.

\begin{definition}[Regularity of local jump rates] 
\label{def-regular} 
Given $j \in \jmps$, the family of functions $\locrate^{H}_j: \Rb_+ \times \cad^{V_H} \times \emksp^{E_H} \times \vmksp^{V_H} \rightarrow \Rb_+$, $H = (V_H, E_H, \root_H) \in \gsone$, is said to be regular if $\locrate^{H}_j$ is Borel measurable for each $H \in \gsone$ and additionally satisfies the following two properties: 
\begin{enumerate} 
\item (Symmetry): for every $t > 0$, $H \mapsto \locrate^H_{j}(t, \cdot, \cdot, \cdot)$ is a class function in the sense that for any $(x, \ems, \vms) \in \cad^{V_H} \times \emksp^{E_H} \times \vmksp^{V_H}$, $\widehat{H} \in \gsone$ with $\widehat{H} \cong H$, and isomorphism $\varphi \in I(\widehat{H},H)$, 
\[\locrate^H_{j} (t, (x_v)_{v \in V_H}, (\ems_e)_{e \in E_H}, (\vms_v)_{v \in V_H}) = \locrate^{\widehat{H}}_{j} (t, (x_{\varphi(v)})_{v \in V_{\wh{H}}}, (\ems_{\varphi(e)})_{e \in E_{\wh{H}}}, (\vms_{\varphi(v)})_{v \in V_{\wh{H}}});\] 
\item (Predictability): for every $H \in \gsone$, $(\ems, \vms) \in \emksp^{E_H} \times \vmksp^{V_H}$ $(t,x) \mapsto \locrate^H_{j} (t, x, \ems, \vms)$ is predictable in the sense that for every $t > 0$ and $x, y \in \cad^{V_H}$, 
\[x(s) = y(s) \quad \forall s \in [0,t) \quad \Rightarrow \quad \locrate^H_{j} (t, x, \ems, \vms) = \locrate^H_{j} (t, y, \ems, \vms).\] 
\end{enumerate} 
\end{definition} 

We now state our standing assumption on the IPS models we consider. 

\begin{sassu} 
\label{mod:consist} 
Given any $\sp{\emksp,\vmksp}$-graph $G = (V,E, \root, \ems,\vms)$, the jump rates ${\bf \rate^G}$ for the IPS dynamics on the graph $G$ satisfy \eqref{eq:standing} for local jump rates ${\bf \locrate} = \{\locrate^{H}_j: \Rb_+ \times \cad^{V_H} \times \emksp^{E_H} \times \vmksp^{V_H} \rightarrow \Rb_+\}_{H = (V_H, E_H, \root_H) \in \gsone}, j \in \jmps,$ that are regular in the sense of Definition \ref{def-regular}.
\end{sassu}

\begin{remark}[Isomorphism invariance of jump rates] 
\label{rem-classfn} 
As a consequence of the class function property of the IPS model or local jump rates ${\bf \locrate}$ specified in property 1 of Definition \ref{def-regular}, the associated jump rates for the IPS dynamics on any marked graph $G$ given by \eqref{eq:standing} satisfy the following analogous class property. Given any $\sp{\emksp,\vmksp}$-graphs $G,G'$ with $G \cong G'$ and isomorphism $\varphi \in I(G',G)$, for any $v \in V_G$, $t \in \Rb_+$, and $x \in \cad^{G}$, 
\begin{equation} 
\label{mod:rsym0} 
\rate\gvpara{G}{v}\stpara{j}(t,(x_{u})_{u \in V_G}) = \rate\gvpara{G'}{\varphi^{-1}(v)}\stpara{j}(t, (x_{\varphi(u)})_{u \in V_{G'}}). 
\end{equation} 
\end{remark}

\subsection{Examples of Interacting Particle Systems} 
\label{mod:examples} 
We now provide several examples of IPS models, or equivalently local jump rates $\bf{\locrate}$, that satisfy the regularity conditions imposed in our standing assumption. 

\begin{example}[Countable-state Markovian IPS] 
\label{mod:Markov} 
Suppose the mark spaces $\emksp$ and $\vmksp$ are trivial, and suppose that for every $j \in \jmps$, there exist functions $\alt{\rate}_j^H: [0,\infty) \times \Xc^{V_H}\to \Rb_+, H \in \gsone,$ such that the local jump rates $\locrate_j^H, H = (V_H,E_H,\root_H) \in \gsone,$ satisfy 
\begin{equation} 
\label{rate-mkov} 
\locrate_j^H (t,x) = \alt{\rate}^H_j(t, x(t-)) \quad \forall (t,x) \in \Rb_+ \times \cad^{V_H}. 
\end{equation}

We present the form of $\alt{\rate}^H_j, j \in \jmps, H \in \gsone,$ for several classical IPS models below, in all cases denoting $\root_H$ simply by $\root$ for notational simplicity. 
\begin{enumerate} 
\item {\bf The contact process} with measurable infectivity and recovery rates $\lambda, \rho: [0,\infty) \mapsto (0,\infty)$; the time-homogeneous version (with constant $\lambda, \rho$) was introduced in \cite{Har74}, see also \cite{Lig99} and, for more recent studies on (locally) tree-like graphs, see \cite{Pem92}, \cite{Sta01}, \cite{Bhaetal21}, and references therein. This model has $\Xc = \{0,1\}$, $\jmps = \{-1,1\}$, and 
\[\alt{\rate}^H_1(t,z) = \indic{z_{\root} = 0}\lambda(t-)\sum_{u \in \gneigh{\root}{H}} z_u, \qquad \alt{\rate}^H_{-1}(t,z) = \rho(t-) \indic{z_{\root} = 1}.\] 
\item {\bf Branching random walk} with branching rate $\lambda$; see \cite[page 1564]{PemSta01}. This model has $\Xc = \Nb_0$, $\jmps = \{-1,1\}$, and homogeneous rates $\alt{\rate}^H_1(z) = \alt{\rate}^H_1(t,z)$ given by 
\[\alt{\rate}^H_1(z) = \lambda \sum_{u \in \gneigh{\root}{H}} z_u, \qquad \alt{\rate}^H_{-1}(z) = z_{\root}.\] 
\item {\bf Voter models} which model opinion dynamics and has homogeneous rates $\alt{\rate}^H_j(z) = \alt{\rate}^H_j(t,z)$; see \cite{Lig99} and \cite[page 2]{CarTorMig16}. The voter model has many variations; for example, in the linear voter model, each particle updates its opinion at times dictated by a unit Poisson process by copying the opinion of one of its neighbors chosen uniformly at random. In this model. $\Xc = \{0,1\}$, $\jmps = \{-1,1\}$, and 
\[\alt{\rate}^H_1(z) = \frac{\indic{z_{\root} = 0}}{|\gneigh{\root}{H}|}\sum_{u \in \gneigh{\root}{H}} z_u, \quad \alt{\rate}^H_{-1}(z) = \frac{\indic{z_{\root} = 1}}{|\gneigh{\root}{H}|}\sum_{u \in \gneigh{\root}{H}} (1 - z_u),\] 
In the {\em majority process} \cite[page 443]{AmiBalBei23}, each particle instead updates its opinion by adopting the majority opinion of its neighbors (resolving ties randomly). In this model, $\Xc = \{-1,1\}$, $\jmps = \{-2,2\}$, and
\[\alt{\rate}^H_{j}(z) = \indic{z_{\root} = -\frac{j}{2}}\left(\indic{j \sum_{u \in \gneigh{\root}{H}} z_u > 0} + \frac{1}{2}\indic{\sum_{u \in \gneigh{\root}{H}} z_u = 0}\right), \quad j \in \{-2,2\}.\] 
\item {\bf Glauber dynamics for the Ising model} with inverse temperature $\beta \in \Rb$ (see, e.g., \cite[Definition 2]{MosSly13}) has $\Xc = \{-1,1\}$, $\jmps = \{-2,2\}$, and $\alt{\rate}^H_j(z) = \alt{\rate}^H_j(t,z)$ with 
\[\alt{\rate}^H_{j}(z) = \indic{z_o = -j/2} \frac{\exp\left(-\beta z_{\root}\sum_{u \in \neigh{\root}} z_u\right)}{\exp\left(-\beta z_{\root}\sum_{u \in \neigh{\root}} z_u\right) + \exp\left(\beta z_{\root}\sum_{u \in \neigh{\root}} z_u\right)}, \quad j \in \{-2,2\}.\] 
\end{enumerate} 
\end{example}

\begin{example}[SIR model in a heterogeneous environment] 
\label{ex:RE} 
The SIR model introduced in \eqref{intro:contrates2} with measurable infectivity and recovery rates $\lambda, \rho: [0,\infty) \mapsto (0,\infty)$ has $\Xc = \{0,1,2\}$, $\jmps = \{1\}$, and for $H = (V_H, E_H, \root_H) \in \gsone$, $\locrate^H_1 (t,x, \ems, \vms) = \alt{r}^H_1(t, x(t-), \ems, \vms),$ where $\alt{r}^H_1: \Rb_+ \times \Xc^{V_H} \times \emksp^{E_H} \times \vmksp^{V_H} \to \Rb$ is given by (once again writing $\root$ for $\root_H$) 
\[\alt{r}^H_1(t,z, \ems, \vms) = \left(\lambda (t-) \sum_{u \in \neigh{\root}(H)} \ems_{\{u,\root\}} z_u\right)\indic{z_{\root} = 0} + \rho(t-) \vms_{\root} \indic{z_{\root} = 1}\] 
See \cite{JanLucWin14} and \cite{CocRam23} for recent work on SIR models on locally tree-like graphs, with the latter also considering heterogeneous environments. Clearly, heterogeneous versions of the other examples can also be described in an analogous fashion. 
\end{example}

\begin{example}[Non-Markovian SIR model] 
\label{ex:NMcont} 
We now introduce a class of non-Markovian SIR models that fall within our framework. As before, we have $\Xc = \{0,1,2\}$ and $\jmps = \{1\}$. Let $\beta: \Rb_+\to \Rb_+$ and $\gamma: \Rb_+\to \Rb_+$ be functions for which $\beta(s)$ and $\gamma(s)$ are the rates at which a particle that has been infected for $s$ units of time infects a fixed neighboring particle or recovers, respectively. A natural choice for $\beta$ and $\gamma$ would be the respective hazard rate functions of the infection and recovery distributions. To define the local jump rates, consider the functional $\tau: \Rb_+\times \cad \times \vmksp\to \Rb_+$ given by 
\[\tau(t,x,\vms) \defeq 
\begin{cases} 
t - \sup\{s \in [0,t): x(s) = 0\} & \te{ if } \{s \in [0,t): x(s) = 0\} \neq \emptyset, \\ 
t + \kappa & \te{ if } \{s \in [0,t): x(s) = 0\} = \emptyset. 
\end{cases} 
\] 
In other words, $\tau (t,x,\vms)$ describes the amount of time a particle with trajectory $x$ has been infected at time $t$, assuming it had already been infected for $\vms$ units of time at time $0$. Naturally, if $x_v(0) \neq 1$, then it is assumed that $\vms_v = 0$. Then the local jump rate $\locrate^{H}_{1}: \Rb_+\times \cad^{V_H} \times\vmksp^{V_H} \to \Rb_+$ is given by 
\begin{align*} 
\locrate^{H}_1(t,x,\vms) \defeq \left(\sum_{\substack{u \in \gneigh{\root}{H}:\\x_u(t-) = 1}} \beta(\tau(t,x_u,\vms_u))\right)\indic{x_\root(t-) = 0} + \gamma(\tau(t,x_\root,\vms_{\root}))\indic{x_\root(t-) = 1}. 
\end{align*} 
\end{example}

\subsection{Dynamics and Notions of Solutions} 
\label{mod:wpdef} 
While an IPS model can be intuitively specified through its jump rates, we need to more rigorously describe the associated dynamics to verify whether it is well defined. Fix an IPS model with regular local jump rates $\bf{r}$ in the sense of Definition \ref{def-regular}. In addition, fix a random, possibly unrooted, $\sp{\emksp,\vmksp\times \Xc}$-graph $(G,\x)$, henceforth referred to as the \emph{initial data}, that encodes both the random $\sp{\emksp,\vmksp}$-graph describing the interaction structure of the IPS as well as the $\Xc$-valued initial conditions encoded by $\x$. Also, equip $\jmps$ with a finite measure $\Sm \in \Pc(\jmps)$ that assigns strictly positive mass to all elements of $\jmps$ and satisfies $\sum_{j \in \jmps} |j|\Sm(j) < \infty$, where we write $\Sm(j)$ for $\Sm(\{j\})$. When $\jmps$ is finite, $\Sm$ is usually taken to be the counting measure on $\jmps$; when $\jmps$ is countable, $\Sm$ specifies the probabilities with which different jump sizes are considered. Then, the dynamics of the associated IPS are described by the following jump SDE: 
\begin{equation} 
\label{mod:infpart} 
X^{G,\x}_v(t) = \x_v + \int_{(0,t]\times \Rb_+\times \jmps} j\indic{r\leq \rate\gvpara{G}{v}\stpara{j}(s,X^{G,\x})}\,\poiss^G\poissv{v}(ds,dr,dj),\quad v \in V,t \in [0,\infty), 
\end{equation} 
where $\poiss^G$ is the so-called driving noise, comprised of a collection of i.i.d Poisson processes described in Definition \ref{mod:dnoise} below. Here, for each $v \in V$ and $j \in \jmps$, an event $(s,r)$ of the Poisson process $\poiss^G\poissv{v}(\cdot, \cdot, j)$ on $[0,\infty)^2$ represents a potential jump of size $j$ of the particle at $v$ at time $s$, where the particle actually realizes that jump at that time if and only if the parameter $r$ lies below the jump rate of the particle at that time. 

\begin{remark}[Generic Driving Noise Construction] 
\label{rem:drivingnoise} 
When $(G,\x)$ is deterministic, the driving noise $\poiss^G = \{\poiss^G_v\}_{v \in V_G}$ is simply a collection of i.i.d. adapted Poisson processes on $[0,\infty)^2 \times \jmps$, and \eqref{mod:infpart} reduces to a standard Poisson-driven SDE. When $G$ is random but nevertheless $V_G \subseteq \Vmb$ for some countable, deterministic set $\Vmb$ (as is often the case), then the driving noise can once again be easily specified as $\poiss^G\defeq \{\poiss_v\}_{v \in V_G}$, where $\{\poiss_v\}_{v \in\Vmb}$ is a collection of i.i.d. adapted Poisson processes on $[0,\infty)^2 \times \jmps$. When $(G,\x)$ is a general random graph, more care has to be taken in clarifying notions of solutions and associated measurability conditions.
\end{remark}

For general random graphs $(G,\x)$, with no assumptions imposed on the vertex set (e.g., if the initial data $(G,\x)$ is defined to be a general measurable representative graph of the isomorphism class $\ic{(G,\x)}$), more care has to be taken in clarifying notions of solutions, and associated measurability conditions. In this case, the driving Poisson processes are best represented as vertex marks on the graphs (analogous to how initial conditions are treated), though now taking values in the measure space $\Nms(\Rb_+^2\times\jmps)$. However, since the noises also have a temporal component, when the vertex set of the graph is an arbitrary random set, the measurability and adaptedness properties of the driving noises have to be described with more care. This is spelled out in Definition \ref{mod:dnoise}. 

\begin{definition}[Driving Noise for General Random Graphs] 
\label{mod:dnoise} 
Given a complete, filtered probability space $\fpspace$ such that $\filt$ satisfies the usual conditions and a (possibly random) $\filtm_0$-measurable $\sp{\emksp,\vmksp\times\Xc}$-graph $(G,\x)$, referred to as the initial data, an $\filt$-driving noise (compatible with $G$) is a $\sp{\emksp,\vmksp\times\Nms(\Rb_+^2\times\jmps)}$-random graph $(G,\poiss^G)$ that satisfies the following properties: 
\begin{enumerate} 
\item conditioned on $\filtm_0$, $\poiss^G = \{\poiss^G_v\}_{v \in V_G}$ is a collection of i.i.d. Poisson processes on $\Rb_+^2\times\jmps$ with intensity measure $\leb^2\otimes\Sm$, indexed by the vertex set $V_G$ of $G$; 
\item for any $t > 0$ and $A \in \borel((t,\infty)\times \Rb_+\times \jmps)$, if $\poiss^G(A) \defeq \{\poiss^G_v(A)\}_{v \in V_G}$, then the $\sp{\emksp,\vmksp\times \Nb_0}$-random graph $(G,\poiss^G(A))$ is conditionally independent of $\filtm_t$ given $\filtm_0$; 
\item for any $\filtm_0$-measurable $v \in V_G$, $\poiss^G_v$ is an $\filt$-adapted point process in the sense described in Section \ref{nota:msr}. 
\end{enumerate} 
With a slight abuse of notation, we often denote the driving noise $(G,\poiss^G)$ just by $\poiss^G$. 
\end{definition}

When $(G,\x)$ is random, analogous to what is done with initial conditions and the driving noise, it is natural to also encode the trajectories of the IPS, or equivalently any solution to \eqref{mod:infpart}, as additional vertex marks on the random graph. This leads to the following definitions of weak and strong solutions to the SDE \eqref{mod:infpart}. Given a filtered probability space $\fpspace$ that supports a $\sp{\emksp,\vmksp\times \Xc\times\Nms(\Rb_+^2\times\jmps)}$-random graph $(G,\x,\poiss^G)$, define $\mathcal H_t \defeq \sigma ((G,\x,\poiss^G (A)): A \in \borel([0,t]\times \Rb_+\times\jmps))$ for $t \geq 0$, and define $\filt^{G,\x,\poiss^G}$ to be the augmentation of the filtration $\mathbb{H} = \{\mathcal H_t\}_{t\in\Rb_+}$, that is, $\filt^{G,\x,\poiss^G}$ is the smallest complete, right-continuous filtration such that $\filtm^{G,\x,\poiss^G}_t \supseteq {\mathcal H}_t$ for every $t \geq 0$, and $\filtm^{G,\x,\poiss^G}_0$ contains all sets $N \subset A \in \filtm$ with $\PP(A) = 0$. When $(G,\x)$ is deterministic, we denote $\filt^{G,\x,\poiss^G}$ simply by $\filt^{\poiss^G}$. 

\begin{definition}[Weak and Strong Solutions] 
\label{mod:sol} 
\sloppy Given (possibly random) initial data $(G,\x)$, a weak solution to \eqref{mod:infpart} is a tuple $((G,X^{G,\x},\poiss^G),\fpspace)$ such that 
\begin{enumerate} 
\item $\fpspace$ is a complete, filtered probability space such that $\filt$ satisfies the usual conditions; 
\item $(G,\poiss^G)$ is an $\filt$-driving noise compatible with $G$ in the sense of Definition \ref{mod:dnoise}; 
\item $(G,X^{G,\x})$ is a random $\sp{\emksp,\vmksp\times \cad}$-graph with $\filt$-adapted vertex marks (i.e., for any $\filtm_0$-measurable $v \in V_G$, $X^{G,\x}_v$ is an $\filt$-adapted process) such that $X^{G,\x}$ satisfies \eqref{mod:infpart} $\PP$-a.s.. 
\end{enumerate} 
Given a complete probability space $\pspace$ and an $\filtm^{G,\x,\poiss^G}$-driving noise $(G,\poiss^G)$ compatible with $G$ in the sense of Definition \ref{mod:dnoise}, an $\poiss^G$-strong solution $(G,X^{G,\x})$ is a $\sp{\emksp,\vmksp\times\cad}$-random graph such that $((G,X^{G,\x},\poiss^G),(\Omega,\filtm,\filt^{G,\x,\poiss^G},\PP))$ is a weak solution to \eqref{mod:infpart}. 
\end{definition}

\begin{remark}[Filtration-Poisson process pairs] 
\label{mod:concise} 
For conciseness, we often omit mention of the whole probability space $\fpspace$ and simply refer to $(\filt,\poiss^G)$ as a \fpp, and say $(G,X^{G,\x})$ is an $(\filt,\poiss^G)$-weak solution to \eqref{mod:infpart} if $((G,X^{G,\x},$ $\poiss^G),\fpspace)$ is a weak solution to \eqref{mod:infpart}. We also say that $(G,X^{G,\x})$ is a weak solution to \eqref{mod:infpart} if there exists a \fpp\ $(\filt,\poiss^G)$ such that $(G,X^{G,\x})$ is an $(\filt,\poiss^G)$-weak solution to \eqref{mod:infpart}. We would also like to emphasize that if $X$ and $X'$ are both $(\filt,\poiss^G)$-weak solutions, then they are implicitly defined on the same filtered probability space with the same driving noise. 
\end{remark}

\subsection{Notions of Well-Posedness} 
\label{mod:WPsec} 
Given the definitions of weak and strong solutions in the previous section, we now introduce the definitions of uniqueness and well-posedness, which are suitably modified versions of the parallel concepts for SDEs on fixed deterministic vertex sets with a deterministic interaction structure.

\begin{definition}[Uniqueness notions] 
\label{mod:unique} 
The SDE \eqref{mod:infpart} is said to be \emph{unique in law} for the initial data $(G,\x)$ if for any $(G',\x')$, possibly defined on a different probability space, with $(G',\x') \deq (G,\x)$, we have $(G,X) \deq (G',X')$ for all weak solutions $(G,X)$ and $(G',X')$ to the SDE for the respective initial data $(G,\x)$ and $(G',\x')$. The SDE \eqref{mod:infpart} is said to be \emph{pathwise unique} for $(G,\x)$ if for any \fpp\ $(\filt,\poiss^G)$ and any two $(\filt,\poiss^G)$-weak solutions $(G,X^1)$ and $(G,X^2)$ to \eqref{mod:infpart}, $(G,X^1)=(G,X^2)$ a.s.. 
\end{definition} 

\begin{definition}[Well-posedness] 
\label{mod:WP} 
We say that the SDE \eqref{mod:infpart} is \emph{well-posed} for the initial data $(G,\x)$ if there exists at least one weak solution to \eqref{mod:infpart} and the SDE \eqref{mod:infpart} is unique in law for $(G,\x)$. We say the SDE \eqref{mod:infpart} is \emph{strongly well-posed} for the initial data $(G,\x)$ if there exists at least one weak solution to \eqref{mod:infpart} and the SDE \eqref{mod:infpart} is pathwise unique for $(G,\x)$. 
\end{definition}

The next lemma, which we prove in Appendix \ref{cond}, establishes an intuitive sufficient condition for strong well-posedness of \eqref{mod:infpart} for random initial data.

\begin{lemma}[Strong well-posedness for random initial data] 
\label{mod:condwp} 
The SDE \eqref{mod:infpart} is strongly well-posed for the random initial data $(G,\x)$ if it is strongly well-posed for a.s. every realization of $(G,\x)$.
\end{lemma} 

In order to discuss local convergence of solutions, it will be convenient to work with isomorphism classes of initial data and solutions. To this end, note that (strong) well-posedness of (3.3) depends only on the isomorphism class of the initial data. More precisely, given $\ic{(G,\x)}\in \gs\sp{\emksp,\vmksp\times\Xc}$, let $(G_i,\x_i),i=1,2,$ be two different representatives of $\ic{(G,\x)}$ and fix some isomorphism $\varphi\in I((G_2,\x_2),(G_1,\x_1))$. Let $(\filt,\poiss^{G_1})$ be a \fpp\ compatible with $G_1$ and let $X^{G_1,\x_1}$ be a $(\filt,\poiss^{G_1})$-weak solution to \eqref{mod:infpart} for the initial data $(G_1,\x_1)$. Define, \[\poiss^{G_2}_v \defeq \poiss^{G_1}_{\varphi(v)} \quad \te{and} \quad X^{G_2,\x_2}_v\defeq X^{G_1,\x_1}_{\varphi(v)},\quad \te{for all }v \in V_{G_2}.\] Then $(\filt, \poiss^{G_2})$ is a \fpp\ compatible with $G_2$ and by Remark \ref{rem-classfn} and the form of the SDE \eqref{mod:infpart}, it follows that $X^{G_2,\x_2}$ is a $(\filt,\poiss^{G_2})$-weak solution to \eqref{mod:infpart}. Thus, \eqref{mod:infpart} is (strongly) well-posed for $(G_1,\x_1)$ if and only if it is (strongly) well-posed for $(G_2,\x_2)$, showing that (strong) well-posedness is a class property.

\begin{definition}[Strong well-posedness for isomorphism classes] 
\label{mod:sols} 
We say the SDE \eqref{mod:infpart} is (strongly) well-posed for the (possibly random) initial data $\ic{(G,\x)}$ taking values in $\gs\sp{\emksp, \vmksp\times\Xc}$ if there exists a (possibly random) $\sp{\emksp,\vmksp\times\Xc}$-graph $(G,\x)$ with $(G,\x) \in \ic{(G,\x)}$ a.s. such that \eqref{mod:infpart} is (strongly) well-posed for $(G,\x)$. Furthermore, we say that $\ic{(G,X^{G,\x})}$ is a strong (resp. weak) solution to \eqref{mod:infpart} for $\ic{(G,\x)}$ if there exists a (random) representative $(G,\x)$ that lies in $\ic{(G,\x)}$ a.s. and a strong (resp. weak) solution $(G,X^{G,\x})$ to \eqref{mod:infpart} for $(G,\x)$ such that $(G,X^{G,\x}) \in \ic{(G,X^{G,\x})}$ a.s.. 
\end{definition}

\begin{remark}[Implications of Definition \ref{mod:sols}]
\label{isotograph} 
It is worthwhile observing that if the SDE \eqref{mod:infpart} is (strongly) well-posed for initial data given in terms of an isomorphism class $\ic{(G,\x)}$, as in Definition \ref{mod:sols}, it is also (strongly) well-posed, in the sense of Definition \ref{mod:WP}, for any $\filtm_0$-measurable representative graph $(G,\x)$ of $\ic{(G,\x)}$. Indeed, fix $(G,\x)$ and any {\fpp } $(\filt,\poiss^G)$. Consider any $\filtm_0$-measurable random element $(G',\x')$ that is a.s. isomorphic to $(G,\x).$ Then there exists an $\filtm_0$-measurable isomorphism $\varphi\in I((G,\x),(G',\x'))$. Define the {\fpp } $(\filt,\poiss^{G'})$ by setting $\poiss^{G'}_{\varphi(v)} = \poiss^G_v$ a.s. for all $v \in V_G$. Then, building on the isomorphism invariance of jump rates mentioned in Remark \ref{rem-classfn}, there exists a one-to-one correspondence between $(\filt,\poiss^G)$ weak solutions to \eqref{mod:infpart} for the initial data $(G,\x)$ and $(\filt,\poiss^{G'})$ weak solutions to \eqref{mod:infpart} for the initial data $(G', \x)$. 
\end{remark}

We conclude this section with a Yamada-Watanabe type result.

\begin{lemma}[Yamada-Watanabe type result] 
\label{mod:stunique} 
For the initial data $\ic{(G,\x)}$, if the SDE \eqref{mod:infpart} is strongly well-posed, then it is also well-posed. Furthermore, \eqref{mod:infpart} is strongly well-posed for $\ic{(G,\x)}$ if and only if the set of weak solutions to \eqref{mod:infpart} for $\ic{(G,\x)}$ is non-empty and coincides with the set of strong solutions to \eqref{mod:infpart}. \end{lemma} 
\begin{proof} 
\sloppy Fix the $\sigma(\ic{(G,\x)})$-measurable $\sp{\emksp,\vmksp\times\Xc}$-random graph $(G,\x) $ such that $(G,\x) \in \ic{(G,\x)}$ a.s.. By Definition \ref{mod:sols} and Lemma \ref{mod:condwp}, it suffices to prove the lemma for any deterministic $\sp{\emksp,\vmksp\times \Xc}$-graph initial data $(G,\x)$ instead of $\ic{(G,\x)}$. The lemma can be deduced by showing that pathwise uniqueness as defined in Definition \ref{mod:unique} matches the definition of pathwise uniqueness in \cite{Kur07}. Then by \cite[Theorem 3.14]{Kur07}, strong well-posedness is equivalent to all weak solutions being strong as desired. The same theorem also shows that strong well-posedness implies well-posedness. 
\end{proof} 

\begin{remark}[Consequences of strong well-posedness] 
\label{mod:coupling} 
Definition \ref{mod:sol} and Lemma \ref{mod:stunique} imply that under strong well-posedness of the SDE, $(G,X^{G,\x})$ is characterized by the initial data $(G,\x)$ and a compatible driving noise $(G,\poiss^G)$. Hence, a coupling of the initial data and driving noises of a sequence of IPS immediately yields a coupling of the respective solutions to \eqref{mod:infpart} (see Sections \ref{LWCpf} and \ref{GEM}). It is worth emphasizing that \emph{strong} well-posedness is key to facilitating the construction of such couplings. 
\end{remark}

\section{Statements of Main Results} 
\label{Res}
We assume throughout that the IPS model satisfies the standing assumption from Section \ref{mod:dyn}. We first state our results on the strong well-posedness of the SDE \eqref{mod:infpart} governing the IPS dynamics in Section \ref{Res:WP}. We then state our local convergence and hydrodynamic limit results in Section \ref{Res:LWC} and Section \ref{Res:GEM}, respectively. These results hold under some mild additional conditions (Assumption \ref{mod:assu} and Assumption \ref{mod:cont}) on the jump rates that are shown in Appendix \ref{varex} to be satisfied by all the examples of Section \ref{mod:examples}. 

\subsection{Well-Posedness Results} 
\label{Res:WP}
We start by imposing a fairly mild degree-dependent boundedness condition on the jump rates. 

\begin{assumptio}[Bounds on the jump rates] 
\label{mod:assu} 
There exists a family of constants 
\[\mathbf{C} := \{C_{k,T}\}_{k \in \Nb, T \in \Rb_+ } \subset (0,\infty)\] 
with $(k,T) \mapsto C_{k,T}$ being componentwise non-decreasing such that for any $\sp{\emksp,\vmksp}$-graph $G$ and $T \in \Rb_+$, the jump rates ${\bf r}^G$ satisfy 
\begin{equation} 
\label{eq-assubd} 
\rate\gvpara{G}{v}\stpara{j}(t,x) \leq C_{|\gcl{v}{G}|,T}, \qquad v \in V_G, j\in \jmps, t \in [0,T], x \in \cad^G.
\end{equation} 
\end{assumptio}

\begin{remark}[Local jump rate formulation of Assumption \ref{mod:assu}] 
\label{rem:modassu} 
Assumption \ref{mod:assu} can equivalently be phrased as a condition on the local jump rates ${\bf \locrate}$. Indeed, in view of the relation \eqref{eq:standing} between ${\bf \rate^G}$ and ${\bf \locrate}$, it is clear that Assumption \ref{mod:assu} holds with the family of constants ${\bf C}$ if and only if for every $H \in \gsone$ and $T \in \Rb_+,$ 
\[\locrate^H (t,x,\ems, \vms) \leq C_{|V_H|,T}, \quad \mbox{ for all } t \in [0,T], x \in \cad^{V_H}, \ems \in \emksp^{E_H}, \vms \in \vmksp^{V_H}.\] 
We say that the local jump rates ${\bf \bar{\rate}}$ satisfy Assumption \ref{mod:assu} for an associated family of constants ${\bf C}$ when the above inequality holds. 
\end{remark} 

We recall the notion of strong well-posedness introduced in Definition \ref{mod:sols}. In the case of \emph{finite initial data}, that is, when the initial data consists of a finite (possibly unrooted) graph, Assumption \ref{mod:assu} implies the rates are uniformly bounded and strong well-posedness is easily established via a simple recursive construction (see Appendix \ref{WPfin}). For subsequent reference, we state this as a proposition. 

\begin{proposition}[Strong well-posedness for finite initial data] 
\label{WP:fin} 
Suppose Assumption \ref{mod:assu} holds. Then, the SDE \eqref{mod:infpart} is strongly well-posed for all finite initial data. 
\end{proposition}

We now turn to the main case of infinite graphs with possibly unbounded maximal degrees. In this case, the well-posedness of even Markovian SDEs of the form (3.3) is subtle and may fail to hold. Indeed, in Appendix \ref{illpf}, we construct a simple Markovian IPS with uniformly bounded jump rates that admits multiple strong solutions (with different laws) on certain graphs with super-exponential growth. Establishing well-posedness is further complicated because the jump rates of many commonly studied IPS (on graphs of unbounded maximal degree) are unbounded, as demonstrated by the examples in Section \ref{mod:examples}. Nevertheless, the following main result of this section shows that strong well-posedness does hold for IPS on a large subset of (random) graphs that satisfy an (almost sure) \emph{finite dissociability} condition. This condition is an (inhomogeneous) site percolation condition, governed by the jump rates and driving noises on the graph and expressed in terms of the family of constants ${\bf C}$ of Assumption \ref{mod:assu}. Roughly speaking, it ensures that there exists a time interval such that the graph can be a.s. decomposed into a collection of finite (random) subgraphs for which the driving noises on the boundary exhibit no jumps on that interval. We show that this implies that on this time interval, the dynamics of the IPS on each subgraph are not influenced by the dynamics outside of it. Subsequently, in Section \ref{perc:pfwp}, we also show that finite dissociability implies a more general \emph{spatial localization} property that holds on all finite time intervals. The precise definition of a finitely dissociable graph is deferred to Section \ref{perc:findis}, and the notion of spatial localization is made precise in Definition \ref{WP:locunif}. 

\begin{theorem}[Strong well-posedness on finitely dissociable graphs] 
\label{WP:WP} 
Suppose Assumption \ref{mod:assu} holds with the family of constants ${\bf C}$. Then the SDE \eqref{mod:infpart} is strongly well-posed for any a.s. finitely dissociable $\gs\sp{\emksp,\vmksp \times \Xc}$-random element $\ic{(G,\x)}$ with respect to {$\bf C$} (in the sense of Definition \ref{findis:decom} and Remark \ref{findis:WD}). 
\end{theorem}

Theorem \ref{WP:WP} follows from two auxiliary results, Proposition \ref{WP:locality2} and Proposition \ref{WP:WP2}. Proposition \ref{WP:locality2} shows that the almost sure finite dissociability of a graph implies the spatial localization property of the dynamics mentioned above. Proposition \ref{WP:WP2} shows that when the IPS is well-posed for all finite initial data (as is guaranteed under Assumption \ref{mod:assu} by Proposition \ref{WP:fin}), spatial localization on a graph $G$ implies well-posedness of the IPS on that graph. The latter result is of independent interest and can be used to establish well-posedness in some situations where Assumption \ref{mod:assu} does not hold. 

\begin{remark}[Strong well-posedness of Section \ref{mod:examples} examples] 
\label{rem-SIRwellposed} 
When applied to the IPS models listed in Section \ref{mod:examples}, Theorem \ref{WP:WP} in conjunction with Appendix \ref{varex} (which contains verification of Assumption \ref{mod:assu} for some constants $\mathbf{C}$ for these models) rigorously justifies that each of these models is well defined on any finitely dissociable graph (with respect to $\mathbf{C}$). In Section \ref{perc:findex}, we show that such graphs include GW and UGW trees whose respective offspring distributions have finite first and second moments and all graphs of bounded maximal degree (which includes all finite graphs). 
\end{remark}

\begin{remark}[Strong well-posedness without regular local jump rates] 
\label{rem-semiregularity} 
The standing assumption, in particular the symmetry assumptions in Definition \ref{def-regular}, may make it appear that our framework is restricted to spatially homogeneous dynamics. However, we show in Theorem \ref{saver:all} that by adding marks to the graph to make the group of graph automorphisms trivial, Theorem \ref{WP:WP}, in fact, implies well-posedness of a large class of IPS that are highly spatially heterogeneous. This more general well-posedness result is used in \cite{GanRam-MRF22} to establish certain Markov random field properties of the IPS. 
\end{remark}

\subsection{Local Weak Convergence of the Dynamics} 
\label{Res:LWC}
We now address the local weak convergence of processes. Given well-posedness, this is equivalent to establishing continuity (in the local weak topology) of the law of the isomorphism class $\ic{(G,X^{G,\x})}$ of the graph marked with the trajectory of the unique strong solution to the SDE \eqref{mod:infpart} with respect to the $\gs\sp{\emksp,\vmksp\times \Xc}$-valued initial data $\ic{(G,\x)}$. This requires the following additional mild continuity assumption on the local jump rates with respect to the ``environment'' marks. It holds trivially when the mark spaces $\emksp$ and $\vmksp$ are discrete. 

\begin{assumptio}[Weak continuity of jump rates with respect to marks] 
\label{mod:cont} 
\sloppy The initial data $\ic{(G,\x)}$ and jump rate functions $\mathbf{r}^G$ are such that if $(G,X^{G,\x}) = (V,E,\root,\ems,\vms,X^{G,\x})$ is any representative of a strong solution $\ic{(G,X^{G,\x})}$ to \eqref{mod:infpart} for the initial data $\ic{(G,\x)}$, then a.s. for every $(j,v) \in \jmps\times V_G$ and Lebesgue-a.e. $t \in \Rb_+$, $(\ems,\vms)$ is a continuity point of the map: 
\[\emksp^{E_G}\times \vmksp^{V_G} \ni (\ov{\dvms},\dvms) \mapsto \rate^{(V_G,E_G,\root,\ov{\dvms},\dvms),v}_j(t,X^{G,\x}).\] 
\end{assumptio}

\begin{remark}[Why $\ic{(G,X^{G,\x})}$ is well-defined in Assumption \ref{mod:cont}] 
\label{rm-A2welldef} 
Throughout the article, we only apply Assumption \ref{mod:cont} under conditions that also imply strong well-posedness of \eqref{mod:infpart}, in which case we always denote the unique strong solution to \eqref{mod:infpart} by $\ic{(G,X^{G,\x})}$. Thus, the reader may assume that $\ic{(G,X^{G,\x})}$ in Assumption \ref{mod:cont} has a well defined law.
\end{remark} 

\begin{theorem}[Local weak convergence of IPS] 
\label{LWC:LWC} 
Suppose Assumption \ref{mod:assu} holds, and $\ic{(G,\x)}$, $\ic{(G_n,\x^n)}, n \in \Nb,$ are a.s. finitely dissociable $\gs\sp{\emksp,\vmksp\times \Xc}$-random elements (in the sense of Definition \ref{findis:decom}) such that $\ic{(G,\x)}$ satisfies Assumption \ref{mod:cont}. If $\ic{(G_n,\x^n)} \Rightarrow\ic{(G,\x)}$ in $\gs\sp{\emksp,\vmksp \times \Xc}$, then $\ic{(G_n,X^{G_n,\x^n})} \Rightarrow \ic{(G,X^{G,\x})}$ in $\gs\sp{\emksp,\vmksp \times \cad}$. 
\end{theorem}

Theorem \ref{LWC:LWC} follows from a more general almost sure version of this statement proved in Proposition \ref{LWCpf:LWC} under weaker assumptions that only require Assumption \ref{mod:assu2}, a consistent spatial localization condition introduced in Section \ref{subs-pfwp}, and a weaker finite convergence condition, Assumption \ref{LWCpf:finconv}, in place of Assumption \ref{mod:cont}. The weaker Assumption \ref{LWCpf:finconv} is useful for some applications. 

\subsection{Hydrodynamic Limit and Correlation Decay} 
\label{Res:GEM}
Given a Polish space $\Pol$ and a finite, unrooted $\sp{\emksp,\vmksp}$-graph $G$, define the (global) empirical measure of the finite, unrooted $\sp{\emksp,\vmksp\times \Pol}$-graph $(G,z)$ by 
\[\gemp^{G,z} (\cdot) := \frac{1}{|G|}\sum_{v \in V_G} \delta_{z_v}(\cdot),\]
where $\delta_{z_v}(\cdot)$ is the Dirac delta measure concentrated at $z_v$. Also consider the more general \emph{neighborhood empirical measure} given by 
\begin{equation} 
\label{def:nemp} 
\ov{\gemp}^{G,z}_{\ov{\root}} (\cdot) \defeq \frac{1}{|G|}\sum_{v \in V_G} \delta_{\ic{\trnc{1}(\cmpn{v}(G,z))}}(\cdot), 
\end{equation} 
where $\cmpn{v}(G,z)$ denotes the rooted $\sp{\emksp,\vmksp\times\Pol}$-graph obtained by restricting $(G,z)$ to the connected component of $v$ in $\nm{G}$, equipped with $v$ as its root. Note that $\gemp^{G,z} = \gemp^{G',z'}$ and $\ov{\gemp}^{G,z}_{\ov{\root}} = \ov{\gemp}^{G',z }_{\ov{\root}}$ whenever $(G',z')\cong (G,z)$. In other words, both empirical measures are determined only by the isomorphism class $\ic{(G,z)}$ of $(G,z)$. In particular, $\ov{\gemp}^{G,z}_{\ov{\root}}$ is a ${\mathcal P}(\gsonenew[\emksp,\vmksp \times \Pol])$-valued random element that describes the empirical measure of the isomorphism class of a uniformly distributed root in $(G,\ems,\vms,z)$ and its neighborhood, and the ${\mathcal P}(\Pol)$-valued random element $\gemp^{G,z}$ is the $z$-root mark marginal of $\ov{\gemp}^{G,z}_{\ov{\root}}$.

Since $\gemp^{G,z}$ and $\ov{\gemp}^{G,z}_{\ov{\root}}$ are global quantities, their asymptotic behavior cannot be deduced from the local convergence result established in Theorem \ref{LWC:LWC}. Moreover, as discussed in the introduction, unlike in the case of IPS on dense graphs, states of neighboring vertices of IPS on sequences of converging sparse graphs remain strongly correlated and do not become asymptotically independent; that is, propagation of chaos typically fails. Hence, the analysis of the convergence of $\gemp^{G,z}$ is more subtle for IPS on sparse (as opposed to dense) graph sequences, and due to this strong dependence between neighboring vertices, more complex empirical quantities such as $\ov{\gemp}^{G,z}_{\ov{\root}}$ are also of interest. Nevertheless, we show that under a slightly stronger convergence condition on the initial data than that imposed in Theorem \ref{LWC:LWC}, these empirical quantities do have a deterministic limit; see Theorem \ref{GEM:GEMconv} and Corollary \ref{GEM:empconv}, which are presented in Section \ref{GEM:lim}. A key ingredient of the proof is a certain asymptotic correlation decay property, which is first stated in Section \ref{LWC:CPLWS} (see Theorem \ref{GEM:deccor}). 

\subsubsection{An Annealed Correlation Decay Property}
\label{LWC:CPLWS}
We start by introducing a slightly stronger notion of local convergence that applies to graphs that are not necessarily connected and which, in a sense, has a more global flavor. 

\begin{definition}[Local convergence in probability] 
\label{def-inprob} 
Consider a sequence $\{G_n\}_{n \in \Nb}$ of finite, unrooted $\sp{\emksp,\vmksp}$-random graphs. Then $G_n$ converges to a $\gs\sp{\emksp,\vmksp}$-random element $\ic{G}$ in probability in the local weak sense (abbreviated to locally in probability) if for every bounded, $\law(\ic{G})$-a.s. continuous mapping $f: \gs\sp{\emksp,\vmksp} \to \Rb$, 
\begin{equation} 
\label{CPLWS:convdef} 
\frac{1}{|G_n|} \sum_{v \in V_{G_n}} f(\ic{\cmpn{v}(G_n)}) \to \ex{f(\ic{G})} \te{ in probability,} 
\end{equation} 
as $n\to\infty$. 
\end{definition}

Most definitions of convergence in probability only require that \eqref{CPLWS:convdef} hold for bounded and everywhere continuous functions $f$ (e.g., \cite[Definition 2.6]{LacRamWu23}). However, it can be shown that these definitions are equivalent. The version of the definition given above is more convenient for direct application to certain examples where one has to use the fact that \eqref{CPLWS:convdef} holds for the larger class of bounded, a.s. continuous functions $f$, such as Examples A and C of Section \ref{contact-motivating}. 

\begin{remark}[Examples of graphs converging locally in probability] 
\label{rem:cplwsexamples} 
All the examples of sequences of finite random graphs provided in Section \ref{grph:ex}, equipped with roots chosen uniformly at random, turn out to converge in this stronger sense (and not just locally in distribution). 
\end{remark}

We now state an asymptotic correlation decay property.  

\begin{theorem}[Asymptotic correlation decay] 
\label{GEM:deccor} 
Suppose Assumption \ref{mod:assu} holds, and let $\ic{(G,\x)}$ be a $\sp{\emksp,\vmksp\times \Xc}$-random element that satisfies Assumption \ref{mod:cont}, and is a.s. finitely dissociable in the sense of Definition \ref{findis:decom}. Suppose there exists a countable deterministic set $\Smc$ and a sequence $\{(G_n,\x^n)\}_{n \in \Nb}$ of finite, unrooted $\sp{\emksp,\vmksp\times \Xc}$-random graphs, each of whose vertex sets a.s. lie in $\Smc$. Also, for each $n\in \Nb$, let $(G_n,X^{G_n,\x^n})$ denote the strong solution to \eqref{mod:infpart} for the initial data $(G_n,\x^n)$.  Suppose that $|G_n|\to\infty$ in probability, and $(G_n,\x^n)$ converges locally in probability to $\ic{(G,\x)}$. If $o^i_n$, $i=1,2,$ are independent, uniformly distributed vertices of $G_n$ for all $n \in \Nb$, then for any bounded continuous functions $f_i: \gs\sp{\emksp,\vmksp\times\cad}\to \Rb,$ $i=1, 2,$ 
\begin{equation} 
\label{GEM:corlim} 
\lim_{n\to \infty} \te{Cov}\left(f_1\left(\ic{\cmpn{o_n^1}(G_n,X^{G_n,\x^n})}\right),f_2\left(\ic{\cmpn{o_n^2}(G_n,X^{G_n,\x^n})}\right)\right) = 0, 
\end{equation} 
where $\te{Cov}$ represents the covariance functional. 
\end{theorem}

The proof of Theorem \ref{GEM:deccor} is given in Section \ref{GEM}. The assumption in the statement of the theorem that the vertex sets of each $G_n$ a.s. lie in a countable deterministic set is not restrictive because it is satisfied by most common random graph sequences of interest including Erd\"{o}s-R\'{e}nyi graphs, configuration models, and the Barab\`{a}si-Albert model. It is imposed merely for technical convenience. It enables the driving noises on the graph sequence $\{G_n\}$ to be coupled in a measurable way, thereby allowing the application of the results of Section \ref{subs-coupling} in the proof (in particular, of the intermediate result stated in Lemma \ref{EMP:inspce}). 

When the jump rates of the SDE \eqref{mod:infpart} are strongly Lipschitz continuous in the sense described in the introduction, then arguments similar to those used for diffusions in \cite[Lemma 5.2]{LacRamWu23} can be applied to obtain stronger quantitative quenched (i.e., conditioned on the graph) bounds on the decay of correlation of IPS that are uniform with respect to graphs, and only depend on the cardinality of the sets of particles being compared and the graph distance between the sets. Under Assumption \ref{mod:assu}, such a strong Lipschitz condition holds for Markov IPS on graphs with bounded maximal degree but fails to hold for many interesting IPS on graphs with unbounded maximal degree. In such situations, one does not expect there to be a similar quenched correlation bound that is uniform over all graphs (or even all finitely dissociable graphs), and thus, the arguments we use to prove Theorem \ref{GEM:deccor} are crucially different from those used in \cite{LacRamWu23}. Specifically, to establish the \emph{annealed asymptotic} correlation decay property \eqref{GEM:corlim}, which is averaged over the randomness of the initial data, we first use the fact that the initial data satisfies an analogous asymptotic correlation decay property (due to the assumed local convergence in probability), then carefully construct an appropriate coupling and leverage Proposition \ref{LWCpf:LWC}, which establishes a stronger almost sure version of the local convergence result from Theorem \ref{LWC:LWC}, to extend the correlation decay property to the solution process.

\subsubsection{Hydrodynamic Limits} 
\label{GEM:lim}
The existence of the hydrodynamic limit follows as a simple consequence of well-posedness of the limit, local convergence, and asymptotic correlation decay.

\begin{theorem}[Local convergence in probability of IPS] 
\label{GEM:GEMconv} 
Suppose Assumption \ref{mod:assu} holds and the sequence of finite, unrooted $\sp{\emksp,\vmksp\times \Xc}$-random graphs $\{(G_n,\x^n)\}_{n \in \Nb}$ and the $\gs\sp{\emksp,\vmksp\times \Xc}$-random element $\ic{(G,\x)}$ satisfy the conditions of Theorem \ref{GEM:deccor}. Let $(G_n,X^{G_n,\x_n})$, $n\in\Nb$, and $(G,X^{G,\x})$ be weak solutions to \eqref{mod:infpart} for $(G_n, \x_n)$, $n\in\Nb$, and $(G,\x)$, respectively defined on a common probability space. Then the sequence $\{(G,X^{G_n,\x^n})\}_{n \in \Nb}$ converges locally in probability to $\ic{(G,X^{G,\x})}$. 
\end{theorem} 
\begin{proof} 
First note that, given any driving noise, $\ic{(G,X^{G,\x})}$ and $(G_n,X^{G_n,\x^n})$, $n \in \Nb$, are well defined unique strong solutions to \eqref{mod:infpart} by Theorem \ref{WP:WP} and Lemma \ref{mod:stunique}. For each $n \in \Nb$, let $(o^1_n,o^2_n)$ be two i.i.d. uniformly distributed vertices in $G_n$. Then, by Theorem \ref{GEM:deccor} and Theorem \ref{LWC:LWC}, for any bounded, continuous functions $f_1,f_2: \gs\sp{\emksp,\vmksp\times\cad} \to \Rb$, 
\begin{align*} 
\lim_{n\to\infty}\ex{\prod_{i=1}^2f_i(\ic{\cmpn{o^i_n}(G_n,X^{G_n,\x^n})})} & = \lim_{n\to\infty}\prod_{i=1}^2\ex{f_i(\ic{\cmpn{o^i_n}(G_n,X^{G_n,\x^n})})} \\ 
&= \prod_{i=1}^2\ex{f_i(\ic{(G,X^{G,\x})})}. 
\end{align*} 
That this implies the desired result follows from \cite[Lemma 2.8]{LacRamWu23}. 
\end{proof}

We now show that Theorem \ref{GEM:GEMconv} implies the convergence in probability of both the empirical measure and the empirical neighborhood measure to corresponding deterministic limits. 

\begin{corollary}[Hydrodynamic limit] 
\label{GEM:empconv} 
Suppose Assumption \ref{mod:assu} holds and the sequence of finite, unrooted $\sp{\emksp,\vmksp\times \Xc}$-random graphs $\{(G_n,\x^n)\}_{n \in \Nb}$ and the $\gs\sp{\emksp,\vmksp\times \Xc}$-random element $\ic{(G,\x)}$ satisfy the conditions of Theorem \ref{GEM:deccor}. Also, let $(G_n,X^{G_n,\x_n})$ and $(G,X^{G,\x})$ be weak solutions to \eqref{mod:infpart} for $(G_n, \x_n)$ and $(G,\x)$, respectively. Then the $\Pc(\cad)$-valued random empirical measure sequence $\{\gemp^{G_n,X^{G_n,\x^n}}\}_{n\in \Nb}$ converges in probability to $\law(X^{G,\x}_{\root}) \in \Pc(\cad)$, and the $\Pc(\gsonenew\sp{\emksp,\vmksp\times \cad})$-valued random empirical neighborhood measure sequence $\{\ov{\gemp}^{G_n,X^{G_n,\x^n}}_{\ov{\root}}\}_{n \in \Nb}$ converges in probability to $\law(\ic{\trnc{1}(G,X^{G,\x})})\in\Pc(\gsonenew\sp{\emksp,\vmksp\times \cad})$.
\end{corollary} 
\begin{proof} 
By Theorem \ref{GEM:GEMconv}, the sequence $\{(G_n,X^{G_n,\x^n})\}_{n\in\Nb}$ converges locally in probability to $\ic{(G,X^{G,\x})}$. By \cite[Lemma 2.7]{LacRamWu23} (or rather its immediate extension to the setting with edge marks), this directly implies that the sequence $\{\gemp^{G_n,X^{G_n,\x^n}}\}_{n \in \Nb}$ of $\Pc(\cad)$-valued random measures satisfies 
\begin{equation} 
\label{add:eqn} 
\gemp^{G_n,X^{G_n,\x^n}} \mbox{ converges in probability to } \law(X^{G,\x}_{\root}). 
\end{equation} 
This proves the first assertion of the corollary.

Note that the map $\gs\sp{\emksp,\vmksp\times \cad}\ni \ic{(G,x)} \mapsto \ic{\trnc{1}(G,x)} \in \gs\sp{\emksp,\vmksp\times \cad}$ is continuous. Thus, for any bounded, continuous function $f: \gsonenew\sp{\emksp,\vmksp\times \cad} \to \Rb$, the function $g: \gs\sp{\emksp,\vmksp\times \cad} \to \Rb$ given by $g(\ic{(G,x)}) \defeq f(\ic{\trnc{1}(G,x)})$ is also bounded and continuous. Hence, recalling the definition of $\ov{\gemp}^{G_n,X^{G_n,\x^n}}_{\ov{\root}}$ from \eqref{def:nemp}, we have as $n\to\infty$, 
\begin{eqnarray*} 
\int_{\gsonenew\sp{\emksp,\vmksp\times\cad}} f(\ic{(G,x)})\,\ov{\gemp}^{G_n,X^{G_n,\x^n}}_{\ov{\root}}(d\ic{(G,x)}) & = & \frac{1}{|G_n|}\sum_{v \in V_{G_n}} f(\ic{B_1(\cmpn{v}((G_n, X^{G_n,\x^n})}) \\ 
& = & \frac{1}{|G_n|}\sum_{v \in V_{G_n}} g(\ic{\cmpn{v}(G_n,X^{G_n,\x^n})}) \\ 
& \to & \ex{g(\ic{(G,X^{G,\x})})} \\ 
& = & \ex{f(\ic{\trnc{1}(G,X^{G,\x})})},
\end{eqnarray*} 
where the convergence $\to$ in the penultimate line is in probability and justified by \eqref{add:eqn} and \eqref{CPLWS:convdef}. Thus, $\ov{\gemp}^{G_n,X^{G_n,\x^n}}_{\ov{\root}}\to \law(\ic{\trnc{1}(G,X^{G,\x})})$ in probability in $\Pc(\gsonenew\sp{\emksp,\vmksp\times\cad})$. This proves the second assertion of the corollary.  
\end{proof}

As was already observed in \cite[Theorem 6.4 and Proposition 7.7]{LacRamWu23} in the context of interacting diffusions, the stronger local convergence in probability is, in general, necessary to obtain deterministic hydrodynamic limits that coincide with $\law(X_{\root}^{G,\x})$ because if one only has local weak convergence of the initial data as in Theorem \ref{LWC:LWC}, the hydrodynamic limit can fail to be deterministic or fail to coincide with $\law(X_{\root}^{G,\x})$ even when deterministic.

\subsection{Illustrative Example: Ramifications for the SIR Model} 
\label{contact-motivating}
In this section, we use our running example of the heterogeneous SIR model from \eqref{intro:contrates2} to highlight some direct implications of Theorems \ref{WP:WP}, \ref{LWC:LWC}, \ref{GEM:GEMconv}, and Corollary \ref{GEM:empconv}. Similar conclusions can be drawn for the other examples presented in Section \ref{mod:examples} and other IPS models that satisfy our assumptions.

Recall that $\Xc = \{0,1,2\}$ is the state space of the model, $\vmksp = \Rb_+$ is the vertex mark space representing the recovery rates of different individuals, and $\emksp = \Rb_+$ is the edge mark space containing marks that capture the rate at which a vertex infects its neighbor. Let $\nu_v$ and $\nu_e$ denote two probability distributions with bounded support on $\Rb_+$. Recall that the (possibly time-varying) functions $\lambda,\rho: \Rb_+\to\Rb_+$ represent the respective rates at which an infected particle infects its neighboring particles and at which a particle recovers from infection. Assume that the two functions are bounded on compact sets: $\sup_{t \in [0,T]} \max\{\lambda(t),\rho(t)\} < \infty$ for all $T < \infty$. Let $G_n$ be a sequence of Erd\"{o}s-R\'{e}nyi graphs on $n$ vertices with edge probability $c/n\wedge 1$ for some $c \in (0,\infty)$, and let $\root_n$ be uniformly distributed in $V_{G_n}$. Also, let $G$ be the UGW({\rm Poiss}($c$)) tree, with $\root$ denoting its root. In addition, let $\ems, \ems^n, n \in \Nb,$ be i.i.d distributed according to $\nu_e$ and let $\vms, \vms^n, n \in \Nb,$ be i.i.d. distributed according to $\nu_v$. Fix $p \in (0,1)$ and let $\x^n$, $n \in \Nb$, and $\x$ be vectors of i.i.d. $\{0,1\}$ Bernoulli ($p$) random variables. This model is precisely Example \ref{ex:RE}, so it follows from Appendix \ref{varex} that Assumptions \ref{mod:assu} and \ref{mod:cont} hold for this model, and from Examples \ref{ex:ER} and \ref{ex:markedconv} that $(G_n, \x_n)$ converges locally (and in probability) to $(G,\x)$. Then by Theorem \ref{WP:WP}, $\ic{(G,X)}$ is well defined. And by Theorem \ref{LWC:LWC} (respectively, Theorem \ref{GEM:GEMconv}) $\ic{(G_n,\root_n,X^n)}$ converges locally weakly (respectively, in probability) to $\ic{(G, X)}$. 

We present examples of several macroscopic quantities of interest whose convergence can be deduced from Corollary \ref{GEM:empconv}. When $G$ is a UGW tree, a tractable characterization of the limit has recently been obtained in \cite{CocRam23}. 

\begin{description} 
\item[A. Proportion of Infected Particles:] We start with a macroscopic quantity that depends only on time marginals. Fix $t \in \Rb_+$, and let $P_n(t)$ be the proportion of infected particles at time $t$ for the SIR model on $G_n$. Then we claim that 

\[P_n(t) = \frac{1}{|V_{G_n}|}\sum_{v \in V_{G_n}} \indic{X^n_v(t) = 1} \to \PP(X_{\root}(t) = 1) \te{ in probability}.\]

{\em Justification}: Since discontinuities of the function $f_1: \cad \to \Xc$ defined by $f_1(x) = \indic{x(t) = 1}$ lie in the $\law(X_\root)$-null set $\{x: t \in \jmp{x}{\infty}\}$, $f_1$ is a.s. continuous with respect to $\law(X_{\root})$. Hence, applying Corollary \ref{GEM:empconv}, it follows that in probability, 

\[\int f_1(x)\,\gemp^{G_n,X^n}(dx) = \frac{1}{|V_{G_n}|}\sum_{v \in V_{G_n}} \indic{X^n_v(t) = 1} \to \int f_1(x)\,\mu[\{\root\}](dx) = \PP(X_{\root}(t) = 1). \]

\item[B. The Average Duration of Infection of Particles:] We now look at a macroscopic quantity that depends on the trajectories of particles. Fix $0\leq t_1 < t_2 < \infty$, and let $T_n(t)$ denote the average duration of the infection time of a particle in the interval $[t_1,t_2]$ for the SIR model on $G_n$. Then we claim that 
\[T_n(t) = \frac{1}{|V_{G_n}|} \sum_{v \in V_{G_n}} \int_{t_1}^{t_2} \indic{X^n_v(t) = 1}\,dt \to \int_{t_1}^{t_2} \PP(X_\root(t) = 1) \,dt \quad \te{ in probability.}\] 
{\em Justification:} Note that the above display is equivalent to the statement that $\int f_2(x)\,\gemp^{G_n,X^n}(dx) \to \int f_2(x)\, {\rm Law} (X_\root)(dx)$ in probability, where $f$ is the bounded, continuous function $f_2: \cad \to \Xc$ defined as $f_2 (x) := \int_{t_1}^{t_2}\indic{x(t)=1}\,dt.$ The claim then follows from Corollary \ref{GEM:empconv}, which ensures that $\gemp^{G_n,X^n}$ converges to ${\rm Law} (X_\root)$ in probability. 

\item[C. The Average Proportion of Infected Neighbors of a Susceptible Particle:] We now look at a quantity that depends on the neighborhood graph structure. Fix \\ $t \in \Rb_+$, and let $A_n(t)$ denote the average proportion of neighbors of susceptible particles in the graph $G_n$ that are infected at time $t$ (interpreted as equal to zero if no particle is susceptible at time $t$). Then we claim that 
\[A_n(t) = \frac{1}{|\{v \in V_{G_n}: X^n_v(t) = 0\}|}\sum_{v: X^n_v(t)=0} \frac{1}{|\gneigh{v}{G_n}|}\sum_{u \in \gneigh{v}{G_n}} \indic{X^n_u(t)=1}\] 
converges in probability to 
\[A(t) := \ex{\frac{1}{|\gneigh{\root}{G}|}\sum_{u \in \gneigh{\root}{G}} \indic{X_u(t)=1}\middle|X_{\root}(t) = 0}\] 
where $0/0$ is assigned the value $0$ ($A(t)$ is well defined by the remark below). In other words, the limit is the expected proportion of neighbors of the root $\root$ for the SIR model on $G$ that are infected at time $t$, conditioned on the event that $X_{\root}$ is susceptible at time $t.$

{\em Justification:} Recall that $\gsonenew\sp{\{1\},\cad}\subset \gs\sp{\{1\},\cad}$ is the space of isomorphism classes of graphs of radius $1$ with vertex marks in $\cad$. Fix $t > 0$, and let $f_3,f_4: \gsonenew\sp{\{1\},\cad} \to [0,1]$ be defined by 
\[f_3(\ic{(H,x)}) := \frac{\indic{x_{\root_H}(t) = 0}}{|V_H|-1}\sum_{u \in V_H\setminus \{\root_H\}} \indic{x_u(t)=1} \quad \te{and} \quad f_4(\ic{(H,x)}) := \indic{x_{\root_H}(t)=0}.\] 
Also, recall from \eqref{def:nemp} that 
\[\ov{\gemp}^{G_n,X^n}_{\ov{\root}} = \frac{1}{|G_n|} \sum_{v \in V_{G_n}} \delta_{\ic{\trnc{1}(\cmpn{v}(G,z))}},\] 
and set $\ov{\mu} \defeq \law(\ic{\trnc{1}{(G,X)}}).$ Then observe that 
\[A_n(t) = \frac{\int f_3(\ic{(H,x)})\,\ov{\gemp}^{G_n,X^n}_{\ov{\root}}(d\ic{(H,x)})}{\int f_4(\ic{(H,x)})\,\ov{\gemp}^{G_n,X^n}_{\ov{\root}}(d\ic{(H,x)})}\] 
and 
\[A(t) = \frac{\int f_3(\ic{(H,x)})\,\ov{\mu}[\cl{\root}](d\ic{(H,x)})}{\int f_4(\ic{(H,x)})\,\ov{\mu}[\cl{\root}](d\ic{(H,x)})} = \frac{\ex{\frac{\indic{X_{\root}(t)=0}}{|\gneigh{\root}{G}|}\sum_{u\in\gneigh{\root}{G}}\indic{X_u(t)=1}}}{\PP(X_\root(t) = 0)},\] 
once more assigning the value of $0$ to $0/0$. Since the dynamics are driven by i.i.d. Poisson processes, it is not hard to see that the discontinuities of both $f_3$ and $f_4$ lie in the $\ov{\mu}[\cl{\root}]$-null set $\{\ic{(H,x)}\in \gsonenew\sp{\emksp,\vmksp\times\cad}: t \in \jmp{x}{\infty}\}$, so both functions are bounded and a.s. continuous with respect to $\ov{\mu}$. Hence, by Corollary \ref{GEM:empconv} and the continuous mapping theorem, $\langle f_j, \ov{\gemp}^{G_n,X^n}_{\ov{\root}} \rangle \to \langle f_j, \ov{\mu}[\cl{\root}]\rangle$ in probability for $j = 3, 4.$ Below, we show that the denominator of the limit is a.s. positive, so the claim follows. 

{\bf Remark.} We claim that $A(t)$ is well defined because $\PP(X_\root(t) = 0) > 0$. To see why this holds, note that because $|\gneigh{\root}{G}|$ is an a.s. finite, non-negative (integer) random variable, there must exist $k \in \Nb_0$ and $q \in (0,1]$ such that $\PP(|\gneigh{\root}{G}| = k)= q$. Moreover, $X_{\root}(0) = 0$ with probability $p$, which was assumed to lie in $(0,1).$ Lastly, by assumption the distribution $\nu_e$ has bounded support, and so there exists a constant $\ems^* < \infty$ such that $\max\{\ems_e,\ems^n_{e'}\} \leq \ems^*$ a.s. for all $n \in \Nb$, $e \in E_G$, and $e' \in E_{G_n}$. Finally, the maximal rate at which $X_{\root}(0)$ transitions to $1$ is bounded from above at each time $s \in [0,t]$ by $|\gneigh{\root}{G}|\ems^* \lambda(s)$. It therefore follows that 
\[\PP(X_{\root}(t) = 0) \geq pq\exp\left(-k\ems^*\int_0^t \lambda(s)\,ds\right) \geq pq\exp\left(-kt \ems^*\left(\sup_{s \in [0,t)} \lambda(s)\right)\right) > 0,\] 
thus proving the claim.
\end{description}

\section{Spatial Localization and the Proof of Well-Posedness} 
\label{perc}
Throughout this section, we assume, sometimes without explicit mention, that we are given a family of regular local jump rates ${\bf \bar{\rate}}$ and that for any graph $G$, the jump rates ${\bf \rate}^G$ satisfy the standing assumption, that is, are given in terms of the local jump rates ${\bar{\rate}}$ via \eqref{eq:standing}. In Section \ref{subs-pfwp}, we introduce the notions of spatial localization and consistent spatial localization (see Definitions \ref{WP:locunif} and \ref{WP:loccons}) and in Section \ref{subs-wellposed1} show that the SDE \eqref{mod:infpart} is strongly well-posed for any marked graph that ``spatially localizes'' the IPS. In Section \ref{perc:findis}, we define the class of finitely dissociable graphs and provide examples of graphs in this class. In Section \ref{perc:pfwp}, we show that sequences of such graphs consistently spatially localize the SDE \eqref{mod:infpart}.

\subsection{Spatial Localization and Consistent Spatial Localization} 
\label{subs-pfwp}
Throughout this section, we assume the local jump rates ${\bf \locrate}$ satisfy the following mild assumption, which is clearly weaker than Assumption \ref{mod:assu} in light of Proposition \ref{WP:fin}.

\begin{assumptionp}{\ref{mod:assu}$'$}[Strong well-posedness for finite initial data] 
\label{mod:assu2} 
The family of local jump rates ${\bf \locrate }$ is such that for any finite initial data $(G,\xi)$, the SDE \eqref{mod:infpart} with jump rates ${\bf \rate^G}$ defined in terms of ${\bf \locrate}$ via \eqref{eq:standing} is strongly well-posed. 
\end{assumptionp}

We now introduce the notion of spatial localization for a given rooted graph $G$ (and local jump rates ${\bf \bar{\rate}}$). Intuitively, we would like to say that a deterministic $\sp{\emksp,\vmksp}$-graph $G$ spatially localizes the SDE \eqref{mod:infpart} if for any finite $T > 0$ and $\sO \subseteq V_G$ there exists an (in general random) {\em finite} vertex subset $U \subseteq V_G$ that contains $\sO$ with the property that on the time interval $[0,T]$, the $\sO$-marginal of any weak solution to the IPS on $G$ coincides with the $\sO$-marginal of the IPS dynamics on the subgraph $G\subg{U}$. However, for such a characterization to make sense, the IPS dynamics $X^{G\subg{U}, \x_{U}}$ on $G\subg{U}$ needs to be well defined, that is, the IPS needs to be well-posed for the initial data $(G\subg{U},\x_{U})$. Unfortunately, the subtlety is that even though $U$ is (almost surely) finite, Assumption \ref{mod:assu2} does not automatically guarantee this well-posedness. This is because if $X^{G,\x}$ is an $(\filt,\poiss^G)$-weak solution (in the sense of Remark \ref{mod:concise}) to the SDE \eqref{mod:infpart} for the initial data $(G,\x)$, then for any finite $\sO$, the corresponding random graph $G\subg{U}$ will generally fail to be $\filtm_0$-measurable. Thus $(\filt,\poiss^{G\subg{U}}) \defeq (\filt,\poiss^G_U)$ need not be a valid driving noise compatible with the IPS on $G\subg{U}$ (in the sense of Definition \ref{mod:dnoise}). 

However, this technical difficulty can be overcome by instead considering the sequence of deterministic graphs $H_{\ell} \defeq \trnc{\ell}(G)$, $\ell \in \Nmb,$ on each of which the IPS (with the corresponding initial condition) is well-posed by Assumption \ref{mod:assu2}. One can then equivalently define spatial localization for $G$ via the condition that almost surely the $\sO$-marginal dynamics of any weak solution $X^{G,\x}$ to the IPS on $G$ coincides with the $\sO$-marginal dynamics of the IPS on $H_\ell$ for some $\ell \in \Nb_+,$ or in other words, $\bigcup_{\ell\in\Nb}\{X^{G,\x}_{\sO}[T] = X^{H_{\ell},\x_{V_{H_{\ell}}}}_{\sO}[T]\}$.

We now state the precise formulation. Recall from Section \ref{nota:grph} that $$\fset{G} := \{U \subset V_G: |U|< \infty\}.$$ In what follows, we deal with {\em deterministic} initial data $(G,\x)$, recall the notion of a \fpp\ $(\filt,\poiss^G)$ from Remark \ref{mod:concise}, and recall its associated augmented filtration $\filt^{\poiss^G} = \{\filtm_t^{\poiss^G}\}_{t \geq 0}$ introduced prior to Definition \ref{mod:sol}.

\begin{definition}[Spatial localization] 
\label{WP:locunif} 
A deterministic $\sp{\emksp,\vmksp}$-graph $G$ is said to spatially localize the SDE \eqref{mod:infpart} with local jump rates ${\bf \bar{\rate}}$ if for any \fpp\ $(\filt,\poiss^G)$ defined on $\fpspace$ and $T \in \Rb_+$, there exists a (random) mapping $\Oo_T(\cdot;G,\poiss^G): \fset{G} \times \Omega \to 2^{V_G}$ such that for every $\sO, U \in \fset{G}$, 
\begin{equation} 
\label{smap-meas} 
\{\Oo_T(\sO;G,\poiss^G) \subseteq U\} := \{\omega \in \Omega: \Oo_T(\sO;G,\poiss^G)(\omega) \subseteq U \} \in \filtm_T^{\poiss^G}, 
\end{equation} 
\begin{equation} 
\label{smap-finite} 
\PP ( \Oo_T(\sO;G,\poiss^G) \in \fset{G} ) = 1, 
\end{equation} 
and the following properties hold: 
\begin{enumerate} 
\item for each $\sO \in \fset{G}$, $\Oo_T(\sO;G,\poiss^G)(\omega) \supseteq \sO$ for every $\omega \in \Omega$; 
\item given any $\ell \in \Nb,\x \in \Xc^G$, and $\sO,U\subset V$ such that $\sO \subseteq \trnc{\ell}(G)\subseteq U$, every $(\filt,\poiss^G)$-weak solution $X^{G\subg{U},\x_U}$ to the SDE \eqref{mod:infpart} (with local jump rates $\mathbf{\locrate}$) for the initial data $(G[U],\xi_U)$ satisfies 
\begin{equation} 
\label{eq-spatloc} 
X^{G\subg{U},\x_U}_{\sO}[T] = X^{\trnc{\ell}(G),\x_{\trnc{\ell}(G)}}_{\sO}[T] \quad \te{ a.s. on the event } \{\Oo_T(\sO;G,\poiss^G) \subseteq \trnc{\ell}(G)\}, 
\end{equation} 
\end{enumerate} 
where $X^{\trnc{\ell}(G),\x_{\trnc{\ell}(G)}}$ is the unique $\poiss^G_{\trnc{\ell}(G)}$-strong solution to the SDE \eqref{mod:infpart} (with local jump rates ${\bf \bar{\rate}}$) for the initial data $(\trnc{\ell}(G),\x_{\trnc{\ell}(G)})$. In this case $\Oo_T(\cdot;G,\poiss^G)$ is said to be a localizing map for the SDE \eqref{mod:infpart} with local jump rates ${\bf \bar{\rate}}$ on $(G,\poiss^G)$. 
\end{definition}

\begin{remark}[Definition \ref{WP:locunif} notation] 
\label{WP:locshort} 
In Definition \ref{WP:locunif}, we may omit one or both of the last two arguments of $\Oo_T$ when the graph and/or Poisson processes are clear from the context. We will also typically omit explicit mention of the local jump rates when clear from context. 
\end{remark}

Condition \eqref{smap-meas} ensures that the localizing set $\Oo_T(\sO;G,\poiss^G)$ is measurable (with respect to $\filtm_T^{\poiss^G}$), and condition \eqref{smap-finite} encodes the first key property that it must be almost surely finite. Property 1 simply states the obvious fact that the set of particles, equivalently driving Poisson processes, that affect the dynamics on any set $\sO$ must clearly include those associated with the set $\sO$ itself, and property 2 captures the key intuitive property that the marginal dynamics on the set $\sO$ of the IPS on the graph $G$ and the localizing set (almost surely) coincide on the interval $[0,T]$, or stated more precisely, requires the marginal dynamics on $\sO$ over the interval $[0,T]$ of the IPS on the graph $G$ and on the $\ell$-neighborhood of the root to be (almost surely) identical on $[0,T]$ whenever the $\ell$-neighborhood is large enough to include the spatially localizing set.

\begin{remark}[Spatial localization non-uniqueness] 
\label{rem-locdef} 
It should be clear from the above discussion that when $G$ spatially localizes the SDE \eqref{mod:infpart}, there may be many choices for the localizing map $\Oo_T(\sO;G,\poiss^G)$. Indeed, any almost surely finite, measurable set that contains a localizing set is also a localizing set. One natural choice is to let $\Oo_T(\sO;G,\poiss^G)$ be minimal, that is, define it to be precisely the set of particles in $G$ that influence the IPS dynamics of the particles in $\sO$ during the time interval $[0,T]$, as illustrated in Figure \ref{fig-Spatloc}. Two alternative constructions are used in the proof of Proposition \ref{WP:locality2} in Section \ref{perc:pfwp}. 
\end{remark}

\begin{figure}[h]
\centering
\begin{subfigure}{.5\textwidth}
\centering
\tikzstyle{level 1}=[sibling angle=120]
\tikzstyle{level 2}=[sibling angle=80]
\tikzstyle{level 3}=[sibling angle=60]
\tikzstyle{every node}=[fill,scale=0.5]
\tikzstyle{edge from parent}=[segment length=1mm, segment angle=10,draw]
\begin{tikzpicture}[grow cyclic,shape=circle,level distance=6mm, cap=round]
\node [label = {[label distance = 0.25 cm]-75:\huge $\sO$}] (A) at (0,0) {} 
	child [color=black] {node (B) {} 
		child [color = black] {node {} 
			child [color = black] {node {} 
				    child [color = black] {node {}}
				    child [color = black] {node {}}}
			child [color = black] {node {} 
				    child [color = black] {node {}}
				    child [color = black] {node {}}}}
		child [color = black] {node {} 
			child [color = black] {node {}}
			child [color = black] {node {}
				    child [color = black] {node {}}}}}
	child [color=black] {node (C) {} 
		child [color = black] {node {} 
			child [color = black] {node {} 
				    child [color = black] {node {}}
				    child [color = black] {node {}}}
			child [color = black] {node {}
				    child [color = black] {node {}}}}
		child [color = black] {node {} 
			child [color = black] {node {}
				    child [color = black] {node {}}}
			child [color = black] {node {}}}}
	child [color=black] {node {} 
		child [color = black] {node {} 
			child [color = black] {node {} 
				    child [color = black] {node {}}
				    child [color = black] {node {}}}
			child [color = black] {node {}}}
		child [color = black] {node {} 
			child [color = black] {node {}
				    child [color = black] {node {}}}
			child [color = black] {node {} 
				    child [color = black] {node {}}
				    child [color = black] {node {}}}}};
\node [rotate = 30][fill=none, draw,dashed,inner sep =10pt, ellipse, yscale = 0.9, fit = (A)(B)(C)] {};
\end{tikzpicture}
\caption{$\mathbb{T}_3[U]$}
\label{GW}
\end{subfigure} \\
\begin{subfigure}{.48\textwidth}
\centering
\tikzstyle{level 1}=[sibling angle=120]
\tikzstyle{level 2}=[sibling angle=80]
\tikzstyle{level 3}=[sibling angle=60]
\tikzstyle{every node}=[fill,scale=0.5]
\tikzstyle{edge from parent}=[segment length=1mm, segment angle=10,draw]
\begin{tikzpicture}[grow cyclic,shape=circle,level distance=6mm, cap=round]
\node [label = {[label distance = 0.25 cm]-75:\huge $\sO$}] (A) at (0,0) {}
	child [color=black] {node (B) {} 
		child [color = black] {node {} 
			child [color = black] {node {}}
			child [color = black] {node [color=red] {}}}
		child [color = black] {node [color=red] {} 
			child [color = black] {node [color=red] {}}
			child [color = black] {node {}}}}
	child [color=black] {node (C) {} 
		child [color = black] {node {} 
			child [color = black] {node {}}
			child [color = black] {node [color=red] {}}}
		child [color = black] {node {} 
			child [color = black] {node {}}
			child [color = black] {node {}}}}
	child [color=black] {node [color=red] {} 
		child [color = black] {node [color=red] {} 
			child [color = black] {node [color=red] {}}
			child [color = black] {node [color=red] {}}}
		child [color = black] {node [color=red] {} 
			child [color = black] {node [color=red] {}}
			child [color = black] {node [color=red] {}}}};
\node [rotate = 30][fill=none, draw,dashed,inner sep =10pt, ellipse, yscale = 0.9, fit = (A)(B)(C)] {};
\end{tikzpicture}
\caption{$\trnc{3}(\mathbb{T}_3)$}
\label{B3G}
\end{subfigure} ~
\begin{subfigure}{.48\textwidth}
\centering
\tikzstyle{level 1}=[sibling angle=120]
\tikzstyle{level 2}=[sibling angle=80]
\tikzstyle{level 3}=[sibling angle=60]
\tikzstyle{every node}=[fill,scale=0.5]
\tikzstyle{edge from parent}=[segment length=1mm, segment angle=10,draw]
\begin{tikzpicture}[grow cyclic,shape=circle,level distance=6mm, cap=round]
\node [label = {[label distance = 0.25 cm]-75:\huge $\sO$}] (A) at (0,0) {}
	child [color=black] {node (B) {} 
		child [color = black] {node {}}
		child [color = black] {node [color = red] {}}}
	child [color=black] {node [color = red] (C) {} 
		child [color = black] {node [color = red] {}}
		child [color = black] {node [color = red] {}}}
	child [color=black] {node [color = red] {} 
		child [color = black] {node [color = red] {}}
		child [color = black] {node [color = red] {}}};
\node [rotate = 30][fill=none, draw,dashed,inner sep =10pt, ellipse, yscale = 0.9, fit = (A)(B)(C)] {};
\end{tikzpicture}
\caption{$\trnc{2}(\mathbb{T}_3)$}
\label{B2G}
\end{subfigure}
\caption{$\mathbb{T}_3$ is the infinite regular tree, $\sO,U$ are deterministic subsets of its vertices, with 
  $U$ containing $S := \Oo_T(\sO;\mathbb{T}_3,\poiss^{\mathbb{T}_3})(\omega)$, a realization of  the (minimal) spatially localizing set for $\sO$, and  $\mathbb{T}_3[U]$ is the subgraph of  $\mathbb{T}_3$ induced by $U$;    
Figures (a), (b), and (c) describe the (realization $\omega$ of the) dynamics of the IPS on $\mathbb{T}_3[U]$, $B_3[\mathbb{T}_3]$, and $B_2[\mathbb{T}_3]$, respectively, with particles colored black if their dynamics on $[0,T]$ coincide with the IPS on $\mathbb{T}_3[U]$,  and colored red if their dynamics differ.   The fact that all particles in $\sO$ are colored black in (b) but not in (c) reflects the fact that 
$S \subseteq B_3[\mathbb{T}_3]$, but  
$S \not\subseteq B_2[\mathbb{T}_3]$, as is consistent with property \eqref{eq-spatloc}.}
\label{fig-Spatloc} 
\end{figure}

\sloppy Given the fixed IPS model (equivalently, a family of local jump rates $\mathbf{\bar{r}}$ satisfying Assumption \ref{mod:assu2}) we now introduce the concept of consistent spatial localization of a sequence of graphs, which will be used in the proofs of local convergence and the hydrodynamic limit in Sections \ref{LWCpf} and \ref{GEM}, respectively. 

\begin{definition}[Consistent Spatial Localization] 
\label{WP:loccons} 
A sequence of $\sp{\emksp,\vmksp}$-graphs $\{G_n\}_{n \in \Nb}$ defined on some common probability space $(\Omega, \filtm, \PP)$ is said to consistently spatially localize the SDE \eqref{mod:infpart} with local jump rates ${\bf \bar{\rate}}$ if for every $n \in \Nb$ and corresponding \fpp\ $(\filt^n,\poiss^{G_n})$ on $(\Omega, \filtm, \PP)$, there exists a mapping $\Oo_T(\cdot;G_n,\poiss^{G_n}): \fset{G_n} \times \Omega \to 2^{V_{G_n}}$ that satisfies the following properties: 
\begin{enumerate} 
\item for each $n \in \Nb$, $\Oo_T(\cdot;G_n,\poiss^{G_n})$ is a localizing map for the SDE \eqref{mod:infpart} with local jump rates $\mathbf{\bar{r}}$ on $(G_n, \poiss^{G_n})$ in the sense of Definition \ref{WP:locunif}; 
\item for every $n,n',\ell\in\Nb$ such that there exists an isomorphism $\varphi\in I(\trnc{\ell}(\nm{G_{n',*}},\trnc{\ell}(\nm{G_{n,*}}))$, the following property holds: for each pair of sets $\sO \subseteq \trnc{\ell}(G_n)$ and $\sO' \defeq \varphi^{-1}(\sO) \subseteq \trnc{\ell}(G_{n'})$, 
\[\Oo_T(\sO;G_n,\poiss^{G_n}) = \varphi(\Oo_T(\sO';G_{n'},\poiss^{G_{n'}}))\] 
a.s. on the event $\{\Oo_T(\sO;G_n,\poiss^{G_n})\subseteq \trnc{\ell}(G_n)\} \cap \Imc_{\varphi}$, where $\Imc_\varphi$ is the $\filtm_T^{\poiss^{G_{n'}}} \vee \filtm_T^{\poiss^{G_n}}$ measurable event given by 
\[\Imc_\varphi:= \{\varphi \in I(\trnc{\ell}(\nm{G_{n',*}},\poiss^{G_{n'}}),\trnc{\ell}(\nm{G_{n,*}},\poiss^{G_n}))\}.\] 
\end{enumerate} 
In this case $\{\Oo_T(\cdot;G_n,\poiss^{G_n})\}_{n \in \Nb}$ is said to be a consistent sequence of localizing maps for the SDE \eqref{mod:infpart} with local jump rates $\mathbf{\bar{r}}$ on $\{(G_n,\poiss^{G_n})\}_{n \in \Nb}$. 
\end{definition}

In Proposition \ref{WP:locality2}, we show that any sequence of graphs that is a subset of the class of so-called finitely dissociative graphs (see Definition \ref{findis:decom}) consistently spatially localizes the SDE \eqref{mod:infpart}.

It is not hard to show that (consistent) spatial localization is a property of isomorphism classes. We omit the proof as it follows easily from the definition.

\begin{definition}[Spatial localization by isomorphism classes] 
\label{pfwp:consspatloc} 
An isomorphism class $\ic{G}\in \gs\sp{\emksp,\vmksp}$ is said to spatially localize the SDE \eqref{mod:infpart} with local jump rates $\mathbf{\bar{r}}$ whenever each representative $\sp{\emksp,\vmksp}$-graph $G$ does the same. A sequence of isomorphism classes $\{\ic{G_n}\}_{n \in \Nb}$ in $\gs\sp{\emksp,\vmksp}$ is said to consistently spatially localize the SDE \eqref{mod:infpart} with local jump rates $\mathbf{\bar{r}}$ if every sequence of representative graphs $G_n \in \ic{G_n}, n \in \Nb$, does the same. 
\end{definition}

Note that while spatial localization is a class property, the notion of a localizing map given in Definition \ref{WP:locunif} applies to a given graph and not its isomorphism class.

\begin{remark}[Spatial localization for more general IPS] 
\label{WP:modeltype} 
Definitions \ref{WP:locunif} and \ref{WP:loccons} are abstract properties that easily extend to a more general class of graph-indexed jump processes $X^{G,\x}$ that satisfy a different Poisson-driven SDE from \eqref{mod:infpart}, as long as Assumption \ref{mod:assu2} still holds for that SDE. For instance, they could apply to IPS such as the exclusion process in which particles may experience simultaneous jumps, where it would be more natural to index $\poiss^{G}$ by the edges, rather than the vertices, of $G$. 
\end{remark}

\subsection{Well-posedness on spatially localizing graphs} 
\label{subs-wellposed1}
We now show that strong well-posedness holds on spatially localizing graphs under the rather mild Assumption \ref{mod:assu2}. 

\begin{proposition}[Well-posedness for IPS with spatially localizing initial data] 
\label{WP:WP2} 
Suppose the local jump rates ${\bf \bar{\rate}}$ satisfy Assumption \ref{mod:assu2} and that $\ic{(G,\x)}$ is a $\gs\sp{\emksp,\vmksp\times\Xc}$-valued random element such that $\ic{G}$ a.s. spatially localizes the SDE \eqref{mod:infpart} with local jump rates ${\bf \bar{\rate}}$ in the sense of Definition \ref{pfwp:consspatloc}. Then the SDE \eqref{mod:infpart} is strongly well-posed for the initial data $\ic{(G,\x)}$.
\end{proposition} 
\begin{proof} 
By Lemma \ref{msbl:rep}, there exists a $\sigma(\ic{(G,\x)})$-measurable $\sp{\emksp,\vmksp\times\Xc}$-random graph $(G,\x)$ such that $(G,\x) \in \ic{(G,\x)}$ a.s.. By Definition \ref{mod:sols} and Lemma \ref{mod:condwp}, it suffices to prove that \eqref{mod:infpart} is strongly well-posed for a.s. every realization of $(G,\x)$. Therefore, for the remainder of the proof, we assume without loss of generality that $(G,\x)$ is a deterministic $\sp{\emksp,\vmksp\times \Xc}$-graph where $G$ spatially localizes the SDE \eqref{mod:infpart}.

We now explicitly construct a strong solution to the SDE \eqref{mod:infpart} for $(G,\x)$. Let $\poiss^G = \{\poiss^G_v\}_{v \in V_G}$ be a collection of i.i.d. Poisson processes on $\Rb_+^2\times\jmps$ with intensity $\leb^2\otimes \Sm$. Let $\filt \defeq \{\filtm_t\} \defeq \filt^{\poiss^G}$ be the associated filtration (as defined prior to Definition \ref{mod:sol}). It is clear from Definition \ref{mod:dnoise} that $(G,\poiss^G)$ is an $\filt$-driving noise. So by Remark \ref{mod:concise}, $(\filt,\poiss^G)$ is a \fpp. Fix $T\in \Rb_+$, and let $\Oo_T(\cdot) \defeq \Oo_T(\cdot;G,\poiss^G)$ be a corresponding localizing map, which exists due to our assumption that $G$ spatially localizes the SDE. For notational conciseness, let $B_m \defeq \trnc{m}(G)$ for each $m \in \Nb$, and additionally fix $\xi$ and omit the dependence on $\xi$ from the superscript, with the understanding that the mark is always the restriction of $\xi$ to the corresponding graph in the superscript. Furthermore, for any $\ell \in \Nb$, recalling that by Assumption \ref{mod:assu2} the SDE \eqref{mod:infpart} is strongly well-posed for the finite data $(B_\ell,\x_{B_\ell})$, we let $X^{B_\ell}$ denote the corresponding $\poiss^G_{B_{\ell}}$-strong solution to \eqref{mod:infpart}. By Definition \ref{WP:locunif}, for each $m \in \Nb$, there exists a random finite set $U_{m} \defeq \Oo_T(B_m)$ such that for any $\ell, \ell' \in\Nb$, $\ell' \geq \ell\geq m$, by \eqref{smap-meas} and applying \eqref{eq-spatloc} with $U = B_{\ell'}$ and $\sO = B_m$, 
\[X^{B_{\ell}}_{B_m}[T] = X^{B_{\ell'}}_{B_m}[T] \quad \mbox{ on the event } \{U_{m}\subseteq B_{\ell}\} \in \filtm_T.\] 

For $m \in \Nb$, we define the $\cad_T^{B_m}$-valued random element $X^m$ and the random integer $M_m$ as follows: 
\begin{equation} 
\label{def-Mm} 
X^{m}[T] \defeq \lim_{\ell \to \infty} X^{B_{\ell}}_{B_m}[T] \quad \mbox{ and } \quad M_{m} \defeq \inf\{\ell \geq m: U_{m}\subseteq B_{\ell}\}. 
\end{equation}

Since $\bigcup_{\ell \in \Nb} B_{\ell} = V_G$ and $U_m$ is a.s. finite, $X^m$ is well defined on a set of full measure and $M_m$ is a.s. finite. Moreover, since $X^{B_{\ell}}$ is $\filt$-adapted for every $\ell \in \Nb$ and $\filt$ is complete, it follows that $X^{m}[T]$ is also $\filt$-adapted. The last two displays together then imply that a.s., 
\begin{equation} 
\label{WP2:2} 
X^{m}[T] = \lim_{\ell \to \infty} X^{B_{\ell}}_{B_m}[T] = X^{B_{M_m}}_{B_m}[T]. 
\end{equation} 
Furthermore, clearly the sequence $(X^m[T])_{m \in \Nb}$ is consistent: for any $m' > m$, a.s., 
\begin{equation} 
\label{WP2:3} 
X^{m'}_{B_m}[T] = \lim_{\ell \to \infty} (X^{B_{\ell}}_{B_{m'}})_{B_m}[T] = \lim_{\ell \to \infty} X^{B_{\ell}}_{B_m}[T] = X^m[T] = X^m_{B_m}[T], 
\end{equation} 
where the first and third equalities invoke \eqref{WP2:2} and the remaining equalities hold trivially. Now, for every $v \in V_G$, there exists an integer ${m_v} \in \Nb$ such that $v \in B_{m_v}$, and the last display shows that $X^{m_v}_v[T] = X^{m'}_v[T]$ a.s. when $m' \geq {m_v}$. Because $V_G$ is countable and $\filt$ is complete, we can define the $\filt$-adapted $\cad^G_T$-random element $X[T]$ by setting 
\begin{equation} 
\label{WP2:4} 
X_v[T] := \lim_{m\to\infty} X^{m}_v[T] = X^{m_v}_v[T], \quad v \in V_G, 
\end{equation} 
on the set of measure one where the latter limits exist, and setting $X_v[T] \equiv \x_v, v \in V_G,$ on the complement.

To show that the $X[T]$ thus constructed is an $\poiss^G$-strong solution to the SDE \eqref{mod:infpart} on $[0,T]$, fix $v \in V_G$ and define $\bar{m}_v := \max \{m_u: u \in \cl{v}\}$. Then $\cl{v}\subset B_{\bar{m}_v}$ and from \eqref{WP2:2}--\eqref{WP2:4}, it follows that for $\ell \in \Nb$, 
\begin{equation} 
\label{WP2:5} 
X_{\cl{v}}[T] = X^{\bar{m}_v}_{\cl{v}}[T] = X^{B_\ell}_{\cl{v}}[T] \quad \mbox{ on } \quad A_\ell(v) := \{ \ell \geq M_{\bar{m}_v}\}. 
\end{equation} 
Since $X^{B_{\ell}}$ is a $\poiss^G_{B_{\ell}}$-strong solution to the SDE \eqref{mod:infpart}, we obtain a.s., for $t \in [0,T],$ 
\begin{align*} 
X_v(t)\mathbb{I}_{A_\ell(v)} &= X^{B_{\ell}}_v(t) \mathbb{I}_{A_\ell(v)}\\ 
&= \left(\x_v + \int_{(0,t]\times \Rb_+\times\jmps} j\indic{r \leq \rate\gvpara{G\subg{B_\ell}}{v}\stpara{j}(s,X^{B_{\ell}})} \,\poiss^{G}\poissv{v}(ds,dr,dj)\right)\mathbb{I}_{A_\ell(v)}. 
\end{align*} 
By \eqref{eq:standing} and \eqref{WP2:5}, it follows that with $H := G[\cl{v}]$, 
\begin{align*} 
\rate\gvpara{G[B_\ell]}{v}\stpara{j}(s,X^{B_{\ell}}) &= \bar{r}^{H}_j\left(s,X^{B_{\ell}}_{\cl{v}}, (\ems_{e})_{e \in E_H}, (\vms_{u})_{u \in \cl{v}}\right) \\ 
&= \bar{r}^{H}_j\left(s,X_{\cl{v}},(\ems_{e})_{e \in E_H}, (\vms_{u})_{u \in \cl{v}}\right)\\ 
&= \rate\gvpara{G}{v}\stpara{j}(s,X), 
\end{align*} 
for every $j \in \jmps$ and $s \in \Rb_+$ on the event $A_\ell(v)$. Hence, we have 
\begin{align*} 
X_v(t)\mathbb{I}_{A_\ell(v)} & = \left(\x_v + \int_{(0,t]\times \Rb_+\times\jmps} j\indic{r \leq \rate\gvpara{G}{v}\stpara{j}(s,X)} \,\poiss^{G}\poissv{v}(ds,dr,dj)\right) \mathbb{I}_{A_\ell(v)}. 
\end{align*} 
Then, a.s., since $M_{\bar{m}_v}$ is finite, taking the limit as $\ell\to \infty$ shows that 
\[X_v(t) = \x_v + \int_{(0,t]\times \Rb_+\times\jmps} j\indic{r \leq \rate\gvpara{G}{v}\stpara{j}(s,X)}\,\poiss^{G}\poissv{v}(ds,dr,dj).\] 
Thus, we have proved the existence of an $\poiss^G$-strong solution to \eqref{mod:infpart} on any interval $[0,T]$.

We now turn to the proof of pathwise uniqueness. Suppose that $\alt{X}$ and $\alt{X}'$ are any two $(\filt,\poiss^G)$-weak solutions to the SDE \eqref{mod:infpart} for $(G,\x)$. Since $G$ spatially localizes the SDE, for any $\ell > m \in \Nb$ and $M_m$ defined as in \eqref{def-Mm}, invoking \eqref{smap-meas} and applying \eqref{eq-spatloc}, with $U = V$ and $\sO = B_m$, to both weak solutions $\alt{X}$ and $\alt{X}'$, we obtain a.s. on the event $\{\ell \geq M_m\} = \{\Oo_T(\sO;G,\poiss^G) \subseteq \trnc{\ell}(G)\} \in \filtm^{\poiss^G}_T$, 
\[\alt{X}_{B_m}[T] = X^{B_{\ell}(G)}_{B_m}[T] = \alt{X}'_{B_m}[T],\] 
where recall from Remark \ref{WP:locunif} that $X^{B_{\ell}(G)}$ is the unique $\poiss^G_{\trnc{\ell}(G)}$-strong solution to \eqref{mod:infpart} for the initial data $\trnc{\ell}(G,\x)$. Taking the limit as $\ell \to \infty$ and using the almost sure finiteness of $M_m$, it follows that $\widetilde{X}_{B_m}[T] = X_{B_m}[T] = \widetilde{X}'_{B_m}[T]$ a.s. for every $m \in \Nb$, which in turn shows that $X[T] = \widetilde{X}[T] = \widetilde{X}'[T]$ a.s..

Since $T$ is arbitrary for both existence and pathwise uniqueness, $X[T]$, $T > 0,$ are consistent. So there exists a.s. a unique pathwise extension $X$ of the strong solution to all of $[0,\infty)$. This concludes the proof. 
\end{proof}

\begin{remark}[Well-posedness for more general IPS] 
\label{WP:modeltype2} 
Most of the proof of Proposition \ref{WP:WP2} also extends to more general IPS. For instance, if an IPS has simultaneous jumps, then the proof will hold given Assumption \ref{mod:assu2} and spatial localization (see Remark \ref{WP:modeltype}) except for the verification that $X \defeq \lim_{m\to\infty} X^m$ solves \eqref{mod:infpart}. Instead of the latter, one would have to prove that if $X^{B_\ell}$ satisfies the SDE defining the new model on the finite graph $B_\ell(G)$, and $X^m$ is as defined in \eqref{WP2:2}, then the limit $X = \lim_{m \rightarrow \infty} X^m$ also satisfies that SDE on the infinite graph $G$. 
\end{remark}

\subsection{Finitely Dissociable Graphs} 
\label{perc:findis}
\subsubsection{Definition of Finitely Dissociable Graphs} 
\label{perc:finddef} 
We now introduce the class of finitely dissociable graphs, which are defined in terms of an inhomogeneous site percolation on the graph. Recall the intuitive description of finite dissociability given prior to Theorem \ref{WP:WP}, and also recall the definition of the measure space $\Nms(0,T]$ from Section \ref{nota:msr}.

\begin{definition}[Percolation] 
\label{findis:pdef} 
For any $0 < T < \infty$, let $(G,\ppmk)$ be a $\sp{\{1\},\Nms(0,T]}$-random graph. Fix $0 \leq t_1 < t_2 \leq T$, and set $\va_v = \va_v(t_1,t_2) \defeq \indic{\ppmk_v(t_1,t_2] > 0}$ for $v \in V_G$. Then the percolated graph $\perc{t_1,t_2}{G,\ppmk}$ is defined to be the (possibly disconnected and random) subgraph of $G$ induced by the vertex set $\{v \in V_G: R_v = 1\}$. When $t_1 = 0$, we write $\perc{t_2}{G,\ppmk} \defeq \perc{0,t_2}{G,\ppmk}$. 
\end{definition}

In the percolation, we refer to vertices $v \in V_G$ with $R_v =1$ as active and those with $R_v = 0$ as inactive. In our application of Definition \ref{findis:pdef}, the vertex marks $\ppmk$ of the graph $G$ are realizations of (Poisson) point processes. 

\begin{definition}[$\delt$-dissociation] 
\label{findis:perc} 
Given $0 < \delt \leq T < \infty$, we say a $\sp{\{1\},\Nms(0,T]}$-graph $(G,\ppmk)$ $\delt$-\emph{dissociates} if all connected components of $\perc{\delt}{G,\ppmk}$ are finite. 
\end{definition}

For any graph $G$, let $(G,\poiss^G)$ be a driving noise in the sense of Definition \ref{mod:dnoise}. Also, suppose we are given a family of constants ${\bf C} := \{C_{k,T}\}_{k \in \Nb, T \in \Rb_+}$. Then for $T \in (0,\infty)$, let $\poiss^{G,T} = (\poiss_v^{G,T})_{v \in V_G}$ be the collection of point processes on $[0,T]$ defined by 
\begin{equation} 
\label{findis:poiss} 
\poiss^{G,T}_v(t_1,t_2] \defeq \poiss^{G}\poissv{v}\left((t_1,t_2]\times (0,C_{|\cl{v}|,T}]\times\jmps\right), \quad \mbox{ for } 0 < t_1< t_2\leq T < \infty, v \in V_G, 
\end{equation} 
Note that the union of the events of $\poiss^{G,T}\poissv{v}, v \in V_G,$ almost surely contains the set of discontinuities of any weak solution $X^{G,\x}[T]$ of the SDE \eqref{mod:infpart} with jump rates satisfying Assumption \ref{mod:assu} with the family of constants {\bf $C$}. 

\begin{definition}[Finite dissociability] 
\label{findis:decom} 
A deterministic graph $G$ is said to be \emph{finitely dissociable} with respect to the family of constants ${\bf C} := \{C_{k,T}\}_{k \in \Nb, T \in \Rb_+} \subset (0,\infty)$ if for any $T \in (0,\infty)$, there exists $\delt = \delt_T > 0$ such that the associated collection of point processes $\poiss^{G,T}$ defined in \eqref{findis:poiss} is such that $(G,\poiss^{G,T})$ $\delt$-dissociates a.s.. If $G$ is a marked graph, we say $G$ is finitely dissociable (with respect to ${\bf C}$) if and only if the corresponding unmarked graph $\nm{G}$ is finitely dissociable. 
\end{definition}

\begin{remark}[Vertex removal probabilities] 
\label{rem-perc} 
In the percolated graph $\perc{\Delta}{G,\poiss^{G,T}}$ described in Definition \ref{findis:decom}, setting 
\begin{equation} 
\label{eq-barC} 
\bar{C}_{k,T} \defeq \Sm(\jmps) C_{k,T}, \quad \forall k\in \Nb, T\in \Rb_+, 
\end{equation} 
each vertex $v$ is removed from $G$ independently with a probability $\exp\left(-\bar{C}_{|\cl{v}|,T}\Delta \right)$, which is decreasing in the degree of the vertex $v$. 
\end{remark}

\begin{remark}[Finite dissociability of isomorphism classes] 
\label{findis:WD} 
Finite dissociability with respect to any family of constants $\bm{C}$ is a ``class property'' in that it depends only on the isomorphism class $\ic{G}\in \gs$ of the graph $G$, and not on the particular representative or choice of the driving noise. Indeed, if $G_1 \cong G_2$, then for any fixed $T \in (0,\infty)$ and Poisson processes $\poiss^{G_i,T}$, $i = 1, 2$, constructed as in \eqref{findis:poiss}, it is easy to construct a coupling $(\alt{\poiss}^{G_1,T},\alt{\poiss}^{G_2,T})$ of $(\poiss^{G_1,T},\poiss^{G_2,T})$ in which $\alt{\poiss}^{G_i,T} \deq \poiss^{G_i,T}$ for $i=1,2,$ such that $(G_1,\alt{\poiss}^{G_1,T}) \cong (G_2,\alt{\poiss}^{G_2,T})$ a.s.. Hence, for $\Delta \in [0, T]$, \[\perc{\Delta}{G_1,\poiss^{G_1,T}} \deq \perc{\Delta}{G_1,\alt{\poiss}^{G_1,T}} \cong \perc{\Delta}{G_2,\alt{\poiss}^{G_2,T}} \deq \perc{\Delta}{G_2,\poiss^{G_2,T}}, \] with the equivalence holding a.s.. Since the finite dissociability of $G_i$ only depends on the probability that $\perc{\Delta}{G_i,\poiss^{G,T}}$ has an infinite component for some $\Delta > 0$, and the existence of an infinite component is isomorphism invariant, this shows that the finite dissociability property is also invariant with respect to graph isomorphisms. Thus, the statement ``$\ic{G} \in \gs$ is (or is not) finitely dissociable with respect to the family of constants $\bm{C}$'' is well defined .
\end{remark}

\subsubsection{Examples of Finitely Dissociable Graphs} 
\label{perc:findex}
We now show that the class of (almost surely) finitely dissociable graphs encompasses several interesting classes of graphs of interest in applications, including lattices and regular, GW and UGW trees. 

We start by recalling that a rooted tree is a connected, rooted acyclic graph $\tree = (V,E,\root)$. Any pair of vertices $u,v \in V$ has a unique path connecting them. A vertex $v \in V$ is said to be in the $\ell$th generation of $\tree$ if $d_\tree(\root,v) = \ell$. A finite tree $\tree$ is said to be an $\ell$-generation tree if $\max_{v \in V_{\tree}} d_{\tree}(\root,v) = \ell$. For $v \in V\setminus\{\root\}$, the parent of $v$, denoted $\pare{v}{\tree}$, is the unique neighbor $u$ of $v$ such that $d_\tree(\root,u) < d_\tree(\root,v)$. Also, $\chil{v}{\tree} \defeq \gneigh{v}{\tree}\setminus\pare{v}{\tree}$ is said to be the children of $v$. For any vertex set $U\subseteq V$, let $\chil{U}{\tree} = \bigcup_{u\in U}\chil{u}{\tree}\setminus U.$ In addition, given $v, w \in V$, $w$ is said to be a descendant of $v$ if there exists $n \in \Nb$ and a path $(v=u_0, u_1, \ldots, u_n=w)$ in $\tree$ such that for every $i=1,\ldots, n$, $u_i = \chil{u_{i-1}}{\tree}$. 

Fix $\rho \in \Pc(\Nb_0)$. In view of Definition \ref{LWC:GWdef} and the notation above, a random tree $\tree$ is a GW$(\rho)$ tree if for all $\ell \in \Nb$, $\{|\chil{v}{\tree}|\}_{\{v: d_{\tree}(\root,v)=\ell\}}$ is a collection of conditionally i.i.d. $\rho$-distributed random variables given $\trnc{\ell}(\tree)$ (the conditional independence arises because the set of vertices $\{v: d_{\tree}(\root,v)=\ell\}$ is itself random and $\sigma(\trnc{\ell}(\tree))$-measurable).

\begin{proposition}[Finite dissociability of GW trees] 
\label{find:GWtree} 
If $\rho \in \Pc(\Nb_0)$ has a finite first moment, that is, $\sum_{k \in \Nb} k \rho_{k} < \infty$, then the GW tree $\tree\defeq GW(\rho)$ is a.s. finitely dissociable with respect to any family of constants ${\bf C}$. 
\end{proposition} 
\begin{proof} 
Fix $T < \infty$, $\delt \in (0,T)$, and let $\tree^{\delt} := \perc{\delt}{\tree,\poiss^{\tree,T}}$ and $\va_v = \va_v^\Delta := \indic{\poiss^{\tree,T}_v(0,\delt]>0}$, $v \in V,$ be as in Definition \ref{findis:pdef}. Recall that $\tree^\delt$ is precisely the subgraph of $\tree$ induced by active vertices in $V$. Also, recall that for $v \in V_{\tree^\Delta}$, $\cmpn{v}(\tree^{\delt})$ denotes the connected component of $\tree^\delt$ containing $v$, with $v$ as its root. With a small abuse of notation, we extend the definition of $\cmpn{v}(\tree^\delt)$ to all $v \in V$ by setting $\cmpn{v}(\tree^{\delt})$ to be the $1$-vertex graph $\{v\}$ for $v \in V \setminus V_{\tree^\Delta}$. In addition, for $v \in V$, let $\tree_v$ and $\tree^{\delt}_v$ be the restrictions of $\tree$ and $\tree^{\delt}$, respectively, to the set containing $v$ and its descendants in $\tree$. By the self-similarity of the GW tree, for each $v \in V$, $\law(\tree,(R_w)_{w \in V}) = \law(\tree_v,(R_w)_{w \in V_{\tree_v}})$, and hence, $\law(\tree^\delt_{\root}|R_{\root} = 1)= \law(\tree^\delt_v|R_v = 1)$. For $v \in V$, $R_v = 0$ implies $\cmpn{v}(\tree^\delt_v)$ consists of a single isolated vertex. Hence, $|\cmpn{\root}(\tree^{\delt})|<\infty$ a.s. if and only if $|\cmpn{v}(\tree^\delt_v)| < \infty$ a.s. for all $v\in V$, in which case all connected components of $\tree^\delt$ must be a.s. finite. Thus, it suffices to prove that $|\cmpn{\root}(\tree^\delt)|<\infty$ a.s..

Since the percolation probability at a site or vertex depends on its degree via the dependence on $C_{k,T}$ in \eqref{findis:poiss}, to bound the size of $\cmpn{\root}(\tree^\delt)$, we couple $\cmpn{\root}(\tree^\delt)$ with a larger set obtained from a simpler percolation that only removes vertices from the odd generations of $\tree$. To this end, for any rooted tree $\tree$ and $n \in \Nb$, let $L_n(\tree) \defeq \{v \in V: d_{\tree}(\root,v) = n\}$ denote the set of vertices in the $n$th generation. Define the half-percolated forest $\wh{\tree}^\delt \defeq \te{hperc}_{\delt}(\tree,\poiss^{\tree,T})$ to be the subgraph of $\tree$ induced by the vertex set $\{v \in V: R_v = 1 \te{ or } d_\tree(\root,v) \te{ is even}\}$. Then let $\wh{\tree}^\delt_{\root} := \cmpn{\root}(\wh{\tree}^\delt)$ be the subtree of $\wh{\tree}^\delt$ that contains the root (note that the root always belongs to $\wh{\tree}^\delt$). Clearly, $|\cmpn{\root}(\tree^\delt)| \leq |\wh{\tree}^\delt_{\root}|$. Thus, to prove the proposition, it suffices to show that for all sufficiently small $\delt> 0$, 
\begin{equation} 
\label{findis-toshow} 
\lim_{n\to\infty} \ex{\left|L_{2n}(\wh{\tree}^\delt_{\root})\right|} = 0, 
\end{equation} 
since then
\[\PP(|\wh{\tree}^\delt_{\root}| = \infty) = \PP\left(\left|L_{2n}(\wh{\tree}^\delt_{\root})\right| > 0\te{ for all }n \in \Nb\right) \leq \inf_{n \in \Nb}\PP\left(\left|L_{2n}(\wh{\tree}^\delt_{\root})\right| > 0\right) = 0.\]

To prove \eqref{findis-toshow}, choose $n \in \Nb_0$ with $L_{2n}(\wh{\tree}^\delt_{\root}) \neq \emptyset$, and fix $v \in L_{2n}(\wh{\tree}^\delt_{\root})$. Recall that for any tree $\tree$ and $v \in V$, $\chil{v}{\tree}$ denotes the collection of children of $v$ in $\tree$, and $\pare{v}{\tree}$ is the parent of $v$ in $\tree$ whenever $v \neq \root$. Moreover, since $\chil{v}{\wh{\tree}^\delt_{\root}} \subseteq L_{2n+1}(\tree)$, it follows that $\chil{v}{\wh{\tree}^\delt_{\root}} = \chil{v}{\wh{\tree}^\delt} = \{u \in \chil{v}{\tree}: \va_u = 1\}$. Also, observe that 
\begin{align} 
\label{findis:ccv} 
\chil{\chil{v}{\wh{\tree}^{\delt}_{\root}}}{\wh{\tree}^{\delt}_{\root}} := \bigcup_{u \in \chil{v}{\wh{\tree}^\delt_{\root}}} \chil{u}{\wh{\tree}^\delt_{\root}} = \{w \in \chil{u}{\tree}: u \in \chil{v}{\tree} \te{ and } \va_u = 1\}, 
\end{align} 
where the equality uses the fact that the half-percolation does not remove any vertices from even generations of the tree. For each $w \in L_{2n}(\tree)$, note that $w \in \wh{\tree}^{\delt}$, but it is possible that $w \notin \wh{\tree}^{\delt}_{\root}$. Define $Z_{\Delta,w}$ to be the number of grandchildren of $w$ in $\wh{\tree}^\delt$. Since $v \in L_{2n}(\wh{\tree}^\delt_{\root})$, $Z_{\Delta,v}$ is also the number of grandchildren of $v$ in $\wh{\tree}^\delt_{\root}$ and hence, by \eqref{findis:ccv}, 
\begin{align} 
\label{findis:Zdefmsbl} 
Z_{\Delta,v}\defeq \left|\chil{\chil{v}{\wh{\tree}^\delt_{\root}}}{\wh{\tree}^\delt_{\root}}\right| = \left|\chil{\chil{v}{\wh{\tree}^\delt_{\root}}}{\tree}\right| &\te{ is } \sigma\left(\{\chil{u}{\tree}\}_{u \in \{v\}\cup\chil{v}{\tree}}, \{\va_u\}_{u \in \chil{v}{\tree}}\right)\te{-measurable.} 
\end{align} 
\sloppy Furthermore, $L_{2n}(\tree)$ and $L_{2n}(\wh{\tree}^\delt_{\root})$ are both measurable with respect to $\Hmc_n \defeq \sigma(\trnc{2n}(\tree),\{R_w\}_{w \in \trnc{2n-1}(\tree)})$. Let $A \in \Hmc_n$ be an atomic event, that is, $A = \{\trnc{2n}(\tree) = \alt{\tree},\{R_w\}_{w \in \trnc{2n-1}(\tree)} = \{r_w\}_{w \in \trnc{2n-1}(\tree)}\}$ for some $\ell$-generation tree $\alt{\tree}$, with $\ell \leq 2n$ and some $(r_w)_{w \in \trnc{2n-1}(\tree)} \in \{0,1\}^{\trnc{2n-1}(\tree)}$. Then, conditioned on $A$, the collection of random variables 
\[\left(\{\chil{u}{\tree}\}_{u \in \{v\}\cup\chil{v}{\tree}},\{R_u\}_{u \in \chil{v}{\tree}}\right)_{v \in L_{2n}(\tree)}\] 
is equal in distribution to $|L_{2n}(\tree)|$ independent copies of $\left(\{\chil{u}{\tree}\}_{u \in \{\root\}\cup\chil{\root}{\tree}},\{R_u\}_{u \in \chil{\root}{\tree}}\right)$. Together with \eqref{findis:ccv}, this implies that $\gamma_{\delt} = \gamma_{\delt,v} \defeq \law(Z_{\delt,v}|A)$ does not depend on $A$ or the specific choice of $v$ in $L_{2n}(\tree)$. Moreover, conditioned on $A$, \eqref{findis:ccv} implies 
\[Z_{\delt,v} = \sum_{u \in \chil{v}{\tree}}R_u |\chil{u}{\tree}| \deq \sum_{u \in \chil{\root}{\tree}} R_u |\chil{u}{\tree}| = Z_{\delt,\root}.\] 
Thus, conditioned on $\Hmc_n$, $\{Z_{\delt,v}\}_{v \in L_{2n}(\tree)}$ is equal in distribution to $L_{2n}(\tree)$ i.i.d. copies of $Z_{\delt,\root}$. Also, for $v \in L_{2n}(\tree)$, by the assumption that $\rho$ has finite mean
\[\ex{Z_{\delt,v}|\Hmc_n} = \ex{Z_{\delt,\root}} \leq \ex{\sum_{u \in \chil{\root}{\tree}}|\chil{u}{\tree}|} = \left(\sum_{k=0}^\infty k\rho_k\right)^2 < \infty.\] 
Since $L_{2n}(\wh{\tree}^\delt_{\root})$ is $\Hmc_n$-measurable and $\gamma_{\delt,v} = \gamma_{\delt}$ for all $v \in L_{2n}(\wh{\tree}^\delt_{\root})$, a recursive calculation shows 
\begin{align*} 
\ex{\left|L_{2n+2}(\wh{\tree}^\delt_{\root})\right|} &= \ex{\sum_{v \in L_{2n}(\wh{\tree}^\delt_{\root})}\ex{ Z_{\Delta,v}|\Hmc_n}} = \ex{|L_{2n}(\wh{\tree}^\delt_{\root})|}\ex{Z_{\Delta,\root}} = \ex{Z_{\Delta,\root}}^{n+1}. 
\end{align*} 
Thus, to show \eqref{findis-toshow}, it suffices to prove that for sufficiently small $\delt > 0$, $\ex{Z_{\Delta,\root}} < 1$. However, note that for all $\Delta > 0$, $\ex{Z_{\delt,\root}} < \infty$. Furthermore, the definition of $\delt$-dissociation clearly implies $\lim_{\delt \to 0} Z_{\delt,\root} = 0$ a.s., so the Lebesgue dominated convergence theorem implies $\lim_{\delt \to 0}\ex{Z_{\delt,\root}} = 0$, which concludes the proof. 
\end{proof}

In light of Definition \ref{LWC:GWdef}, this immediately implies the corresponding result for UGW trees.

\begin{corollary}[Finite dissociability of UGW trees] 
\label{find:UGWpf} 
If $\tree \sim \te{UGW}(\rho)$, where $\rho \in {\mathcal P}(\Nb_0)$ has a finite second moment, then $\tree$ is a.s. finitely dissociable with respect to any family of constants ${\bf C}$. 
\end{corollary}
\begin{proof}
The fact that $\rho$ has a finite second moment implies that $\wh{\rho}$ in Definition \ref{LWC:GWdef} has a finite first moment. Let $\tree_v$ denote the descendant tree of $T$ rooted at $v$. Then for each $v \neq \root$, $\tree_v$ is a GW$(\wh{\rho})$-tree, so by Theorem \ref{find:GWtree}, for sufficiently small $\delt > 0$, $\cmpn{v}(\perc{\delt}{\tree_v,\poiss^{\tree_v,T}})$ is a.s. finite. Since the $\Delta$-percolated descendant tree of every non-root vertex is finite, it follows that $\tree$ a.s. $\delt$-dissociates. 
\end{proof}

For completeness, we also show that graphs with bounded maximum degree (such as infinite lattices) are also finitely dissociable.

\begin{proposition}[Finite dissociability of bounded degree graphs] 
\label{find:finite} 
Let $G=(V,E)$ be a graph with finite maximum degree: $d^* \defeq \sup_{v \in V} |\neigh{v}| <\infty$. Then $G$ is finitely dissociable with respect to any family of constants $\bf{C}$. 
\end{proposition}
\begin{proof}
Fix $0 < T <\infty$ and $\Delta \in (0,T]$. Then by Remark \ref{rem-perc}, each vertex in $G$ is inactive independently with probability 
\[p_v = \exp\left(-\delt \bar{C}_{|\cl{v}|,T}\right) \geq \exp\left(-\delt \bar{C}_{d^*+1,T} \right) := p_{\delt,T}, \]
where the inequality invokes \eqref{eq-barC}, the monotonicity of $C_{k,T}$ in both its subscripts, and the definition of $d^*$. Therefore, the probability that $(G,\poiss^{G,T})$ $\delt$-dissociates is greater than or equal to the probability that $G$ fails to percolate with respect to a standard site percolation in which each vertex is independently removed with probability $p_{\delt,T}$. For any $T < \infty$, $\lim_{\delt \to 0} p_{\delt,T} = 1$, but as is well known, the critical probability (i.e., the probability $p_c$ such that $G$ fails to percolate a.s. when vertices are independently removed with probability $p > p_c$) is strictly less than 1 (see \cite[Equation (0.3)]{GriSta98}). Thus, for all $T$, there exists a sufficiently small $\delt$ such that $(G,\poiss^{G,T})$ $\delt$-dissociates a.s. By Definition \ref{findis:decom}, this shows that $G$ is finitely dissociable. 
\end{proof}

As demonstrated in Appendix \ref{illpf}, even for Markovian IPS with very regular jump rate functions, finite dissociability, and well-posedness can fail on some graphs.

\subsection{Consistent Spatial Localization on Finitely Dissociative Graph Sequences} 
\label{perc:pfwp} 
The main result of this section is the following: 

\begin{proposition}[Consistent spatial localization by finitely dissociable graphs] 
\label{WP:locality2} 
Suppose ${\bf \bar{\rate}}$ is a family of regular local jump rates that satisfies Assumption \ref{mod:assu} (in the sense of Remark \ref{rem:modassu}) with an associated family of constants ${\bf C}$. If $G$ is a deterministic graph that is finitely dissociable with respect to ${\bf C},$ then $G$ spatially localizes the SDE \eqref{mod:infpart} with local jump rate ${\bf \bar{\rate}}$ in the sense of Definition \ref{WP:locunif}. Moreover, any sequence of deterministic finitely dissociable $\sp{\emksp,\vmksp}$-graphs $\{G_n\}_{n \in \Nb}$ that are all finitely dissociable with respect to the constants ${\bf C}$ consistently spatially localizes the SDE \eqref{mod:infpart} with local jump rates $\bf \bar{\rate}$ in the sense of Definition \ref{WP:loccons}. 
\end{proposition}

We first show why the proposition directly implies Theorem \ref{WP:WP} and subsequently present its proof. 

\begin{proof}[Proof of Theorem \ref{WP:WP} given Proposition \ref{WP:locality2}:] 
Given Assumption \ref{mod:assu}, Proposition \ref{WP:locality2}, and Definition \ref{pfwp:consspatloc} show that $\ic{G}$ a.s. spatially localizes \eqref{mod:infpart}, and Proposition \ref{WP:fin} shows that Assumption \ref{mod:assu2} holds. Therefore, the theorem follows from Proposition \ref{WP:WP2}. 
\end{proof}

A key challenge in establishing Proposition \ref{WP:locality2} is to find an explicit and consistent representation of the localizing map $\Oo_T(\cdot;G_n,\poiss^{G_n})$ on a sequence of finitely dissociative graphs $G_n$. To this end, given the family of constants ${\bf C}$, for a fixed graph $G$ and $T \in \Rb_+$, we consider the associated processes $\poiss^{G,T} = (\poiss^{G,T}_v)_{v \in V_G}$ specified in \eqref{findis:poiss}, and introduce the notion of a \emph{causal chain} associated with $(G, \poiss^{G,T})$. First, define the event set of $\poiss^{G,T}_v$ as follows: 
\begin{equation} 
\label{pfwp:EmcvT} 
\Emc_{v,T} \defeq \{t \in (0,T]: \poiss^{G,T}\poissv{v}(\{t\}) = 1\}. 
\end{equation}

\begin{definition}[Causal chains] 
\label{pfwp:causal} 
Given $T \in (0,\infty)$, an interval $I \defeq [t_1,t_2]\subseteq [0,T]$ and vertices $u, v \in V_G$, a $(G,\poiss^{G,T})$-causal chain from $v$ to $u$ during $I$ is either the singleton $(u)$ when $v = u$, or for some $n \in \Nb$, a path $\Gamma \defeq (v=u_0,u_1,\dots,u_n=u)$ in $G$ such that there exists an increasing sequence $t_1 < s_1 < \cdots < s_n \leq t_2$ for which $s_i \in \Emc_{u_i,T}$, $i = 1,\dots,n$. We write $v \rightsquigarrow_{t_1,t_2} u$ if there exists a $(G,\poiss^{G,T})$-causal chain from $v$ to $u$ during $I = [t_1,t_2]$, and for any $U \subset V_G$, we write $v \rightsquigarrow_{t_1,t_2} U$ if $v \rightsquigarrow_{t_1,t_2} u$ for some $u \in U$. 
\end{definition}

Intuitively, causal chains describe long-range interactions over the graph that can develop over an interval $I$, even though the instantaneous evolution of the state of a vertex is only influenced by the states of neighboring vertices. Specifically, given a graph $G$, processes $\poiss^{G,T}, T \in \Rb_+$, $\sO \in\fset{G}$, and $0 \leq t_1 <t_2 < \infty,$ define 
\begin{equation} 
\label{loc-AT} 
\Cc_{t_1,t_2}^G(\sO) \defeq \{v \in V_G: v \rightsquigarrow_{t_1,t_2} \sO\} \quad \mbox{ and } \quad \Cc_{t_1}^G(\sO) \defeq \Cc_{0,t_1}^{G}(\sO), 
\end{equation} 
where $\rightsquigarrow_{t_1,t_2}$ indicates the existence of a $(G,\poiss^{G,T})$-causal chain during $[t_1,t_2].$ Then $\Cc_{t_1,t_2}^G (\sO)$ represents the set of vertices in $G$ that are ``seen'' by vertices in $\sO$ through a causal chain during some time interval $[t_1,t_2]$. The proof of Proposition \ref{WP:locality2} proceeds by showing that the (random) map that takes $\sO \in \fset{G}$ to $\Cc_{T}^G(\sO)$ for finitely dissociable graphs $G$ defines a localizing map, and moreover, that given a sequence of finitely dissociable graphs $G_n$, $n \in \Nb$, the (random) maps that take $\sO \in \fset{G_n}$ to the closure $\Cc_{T}^{G_n}(\sO)$ for all $n \in \Nb$ define a consistent family of localizing maps.

\begin{proof}[Proof of Proposition \ref{WP:locality2}:] 
We start with the proof of the first statement of the proposition. Fix a deterministic $\sp{\emksp,\vmksp}$-graph $G$ and let $(\filt,\poiss^G)$ be a \fpp. For each (not necessarily finite) $W\subseteq V_G$, let $X^{G\subg{W},\x_W}$ be an arbitrary $(\filt,\poiss^G_W)$-weak solution to \eqref{mod:infpart} for $(G\subg{W},\x_W)$ assuming one exists, which is always the case when $W$ is finite by Proposition \ref{WP:fin} and the fact that Assumption \ref{mod:assu} holds. Fix the family of constants ${\bf C}$ specified in the proposition, and consider the associated processes $(\poiss^{G,T}_v)_{v \in V_G}, T \in \Rb_+$, as in \eqref{findis:poiss}. Let $\sO \in \fset{G}$, and for $T \in \Rb_+$, set $\Oo_T (\sO; G,\poiss^G) := \Cc_{T} (\sO) := \Cc_T^G(\sO)$, with the latter defined as in \eqref{loc-AT} in terms of $(G, \poiss^{G,T})$ causal chains. By Definition \ref{pfwp:causal}, for any $U' \in \fset{G}$, $\{\Cc_{T}(\sO) \subset U'\}$ lies in $\filtm_T^{\poiss^G}$, and hence $\Oo_T (\cdot; G,\poiss^G) = \Cc_T$ satisfies condition \eqref{smap-meas} of Definition \ref{WP:locunif}. Moreover, since $u \rightsquigarrow_{0,T} u$ for all $u \in \sO$, $\Cc_{T}(\sO) \supseteq \sO,$ thus verifying that $\Oo_T (\cdot; G,\poiss^G)$ satisfies property 1 of Definition \ref{WP:locunif}. To prove that any finitely dissociable graph (with respect to ${\bf C}$) spatially localizes \eqref{mod:infpart}, it only remains to verify that $\Cc_{T}(\sO)$ also satisfies condition \eqref{smap-finite} and property 2 of Definition \ref{WP:locunif}. We argue below that it suffices to establish the following claims for every $T \in (0,\infty)$: 
\begin{description} 
\item[\emph{Claim 1: }] If $(G,\poiss^{G,T})$ $T$-dissociates in the sense of Definition \ref{findis:perc}, then $|\Cc_T(\sO)| < \infty$ a.s. for every $\sO \in \fset{G}$. 
\item[\emph{Claim 2: }] If $G$ is finitely dissociable, then $|\Cc_T(\sO)|<\infty$ a.s. for every $\sO \in \fset{G}$. 
\item[\emph{Claim 3: }] Fix any $\ell \in \Nb$ and (not necessarily finite) $U\subseteq V_G$ such that $\trnc{\ell}(G)\subseteq U$ and a weak solution $X^{G\subg{U},\x_U}$ to the SDE \eqref{mod:infpart} with local jump rates ${\bf \bar{\rate}}$ exists. Then for each $\sO \in \fset{G\subg{U}}$, 
\begin{equation} 
\label{eq-claim3} 
X^{\trnc{\ell}(G),\x_{\trnc{\ell}(G)}}_{\sO}[T] = X^{G\subg{U},\x_U}_{\sO}[T] \quad \mbox{ a.s. on } \{\Cc_T(\sO) \subseteq \trnc{\ell}(G)\}. 
\end{equation} 
\end{description}

Claim 2 shows that $\Oo_T (\cdot; G,\poiss^G)$ satisfies \eqref{smap-finite} when $G$ is finitely dissociable. Suppose Claims 1-3 hold, and let $U\supseteq \trnc{\ell}(G)$ be any (not necessarily finite) vertex set for which there exists a weak solution $X^{G\subg{U},\x_{U}}$ as stated in Claim 3. Then clearly $\trnc{\ell}(G[U]) = \trnc{\ell}(G)$, so for $\sO \in \Lambda_{G[U]}$ and on the event $\{\Cc_T(\sO) \subseteq \trnc{\ell}(G)\} = \{\Cc_T(\sO) \subseteq B_{\ell}(G\subg{U})\}$, the $(G\subg{U},\poiss^{G,T}_{U})$-causal chains ending at $\sO$ at time $T$ are the same as the $(G,\poiss^{G,T})$-causal chains ending at $\sO$ at time $T$, that is, $ \Cc_T^{G[U]}(\sO) = \Cc_T^G(\sO) = \Cc_T(\sO)$. Thus, by Claim 3, \eqref{eq-claim3} holds, which verifies property 2 of Definition \ref{WP:locunif}. Since we have verified all properties of Definition \ref{WP:locunif} when $G$ is dissociable (with respect to ${\bf C}$), the first assertion of the proposition follows.

We now turn to the proofs of the claims.\\

\noindent \textbf{\emph{Proof of Claim 1.} } Fix $\sO \in \fset{G}$. Define $\wh{G}\defeq \perc{T}{G,\poiss^{G,T}}$ as in Definition \ref{findis:pdef}. If $(G,\poiss^{G,T})$ $T$-dissociates, then each of the connected components of $\wh{G}$ is a.s. finite, and $\poiss^{G,T}\poissv{w}\left(0,T\right] = 0$ for all $w \in \neigh{\wh{\sO}}$, where $\wh{\sO} := \sO \cup \bigcup_{v \in \sO} \cmpn{v}(\wh{G})$. Since $\sO$ is finite, $\wh{\sO}$ is a.s. finite. Now suppose $u \in \sO$ and $v\in \Cc_T(u)$. Let $(v=u_0,u_1,\dots,u_n=u)$ be a $(G,\poiss^{G,T})$-causal chain with respect to $[0,T]$. Then, for any $i =1,\dots,n$, $u_i$ must be active, that is, $\poiss^{G,T}\poissv{u_i}\left(0,T\right] > 0.$ Thus, for each $i = 1,\dots,n$, $u_i \in \wh{\sO}$, and so $v \in \gcl{\wh{\sO}}{G}$. Hence, $\Cc_T(\sO)= \bigcup_{u \in \sO}\Cc_T(u) \subseteq \gcl{\wh{\sO}}{G}$, which is a.s. finite since $\wh{\sO}$ is a.s. finite, and $G$ is locally finite. \\

\noindent \textbf{\emph{Proof of Claim 2. }} Fix $u \in \sO$ and $v \in \Cc_T(\sO)$. Because $G$ is finitely dissociable, there must exist $\delt > 0$ such that $(G,\poiss^{G,T})$ $\delt$-dissociates. If $\Delta \geq T$, the result follows from Claim 1. If $\Delta \in (0,T)$, then Claim 1 implies that a.s., 
\begin{equation} 
\label{eq-loc1} 
|\Cc_\Delta(\sO)| < \infty \quad \mbox{ for every } \sO \in \fset{G}. 
\end{equation} 
To complete the proof, for every $t \in [0,T-\delt]$, we show that if $|\Cc_t(\sO)| < \infty$ a.s. for every $\sO\in \fset{G}$, then we also have $|\Cc_{t+\delt}(\sO)| < \infty$ a.s. for every $\sO \in \fset{G}$.

To this end, fix $t \in [0,T-\delt]$, and suppose $|\Cc_{t}(\sO)| < \infty$ a.s. for all $\sO \in \fset{G}$. Fix $\sO \in \fset{G}$, $u \in \sO$, and $ v \in \Cc_{t+\delt}(u)$. Then there exists a $(G,\poiss^{G,T})$-causal chain $\Gamma = (v=u_0,u_1,\dots,u_n=u)$ during the interval $[0,t+\delt]$, with the corresponding sequence of times $0 = s_0 < s_1 < \cdots < s_n \leq t+\delt$. Let $i_*$ be the largest integer $i \in \{0, \ldots, n\}$ such that $s_i \leq \Delta$. Then $v \in \Cc_{\Delta}(u_{i_*})$. Furthermore, by considering the path $(u'_0 = u_{i_*}, \ldots, u'_i = u_{i+i_*}, \ldots, u'_{n-i^*} = u)$ with times $s'_0 = \delt$ and $s'_{i} = s_{i+i^*}$, $i=1, \ldots, n-i_*$, it follows that $u_{i_*} \rightsquigarrow_{\delt,t+\delt} u$. Thus, $v \in \Cc_{\delt}(\Cc_{\delt,t+\delt}(u))$, which implies $\Cc_{t+\delt}(\sO)\subseteq \Cc_{\delt}(U)$, where $U := \Cc_{\delt,t+\delt}(\sO)$. In view of \eqref{eq-loc1}, it suffices to show that $U$ is a.s. finite, but this holds because by time homogeneity of Poisson processes, $U \deq \Cc_t(\sO)$, which is a.s. finite by assumption. \\

\noindent \textbf{\emph{Proof of Claim 3. }} For notational conciseness, let $H \defeq G\subg{U}$, where we recall that $U$ is a deterministic vertex set such that $\trnc{\ell}(G)\subseteq U$ and a weak solution $X^{H, \x_U} = X^{G\subg{U},\x_{U}}$ exists. Additionally, let $H_{\ell} \defeq \trnc{\ell}(G)$, and note that $H_{\ell} = \trnc{\ell}(H)$. Fix $\sO \in \fset{H}$ with $|\Cc_T^H(\sO)|< \infty$ and note that $\Cc_T^H(\sO) = \Cc_T^G(\sO)= \Cc_T(\sO)$ on the event $\{H_{\ell}\supseteq \Cc_T^H(\sO)\}$. To prove Claim 3, we need to show that 
\begin{equation} 
\label{pfwp:TS} 
X_{\sO}^{H,\xi_H}[T] = X_{\sO}^{H_{\ell},\xi_{H_{\ell}}}[T] \te{ a.s. on }\{\Cc_T(\sO) \subseteq V_{H_{\ell}} \}. 
\end{equation}

Assume $\poiss^{G,T}_U(\{T\}) = 0$, which holds a.s. because $\poiss^{G,T}_U$ is a countable collection of homogeneous Poisson processes. We prove this claim in a recursive fashion by iterating over events in the driving noise $\poiss^G$, and relating them to the dynamics of the SDE \eqref{mod:infpart}. For each $v \in U$, let $\{t_i^v, i\in \Nb\}$ be an enumeration of the (a.s. finite) set $\Emc_{v,T}$ from \eqref{pfwp:EmcvT} of events of $\poiss^{G,T}_v$ in $[0,T]$, arranged in increasing order. Below, we use the conventions that $\max\emptyset = 0$, $\inf\emptyset = \infty$, and $[0)= \{0\}$. Define $U_0 := \sO$, $\tau_0 := T$, choose an arbitrary vertex $v_0 \in \sO$ and for $k \in \Nb$, recursively define 
\begin{align} 
\label{eq-loctk} \tau_k & \defeq 
\begin{cases} 
\max\{t_i^v: v \in U_{k-1}, t_i^v < \tau_{k-1} \} &\te{ if } \tau_{k-1} > 0,\\ 
0 &\te{ if } \tau_{k-1} = 0, 
\end{cases}\\ 
\label{eq-locvk} v_k & \defeq 
\begin{cases} 
v\in U_{k-1} \te{ s.t. }\poiss^{G,T}\poissv{v}(\{\tau_k\}) = 1 &\te{ if } \tau_k > 0,\\ 
v_{k-1} &\te{ if } \tau_k = 0, 
\end{cases}\\ 
\label{eq-locUk} U_k & \defeq 
\begin{cases} 
U_{k-1}\cup \cl{v_{k}} &\te{ if } \tau_k > 0,\\ 
U_{k-1} &\te{ if } \tau_k = 0, 
\end{cases} 
\end{align} 
where $v_k$ in \eqref{eq-locvk} is a.s. well defined because the sets $\Emc_{v,T}$, $v \in U_{k-1}$, are a.s. distinct. Also, set 
\begin{equation} 
\label{def-K} 
K := \inf \{k \in \Nb: \tau_k = 0\}. 
\end{equation} 
Note that the above construction is well defined even if $K = \infty$, and furthermore, the sequence $\{\tau_k\}_{k \in \Nb}$ is strictly decreasing and the sequence $\{U_k\}_{k \in \Nb}$ is non-decreasing but with possible repetitions (for instance, if $v_k$ lies in the interior of $U_{k-1}$). 

As shown below, the set $U_K$ is specifically constructed so that $U_K\subseteq \Cc_T(\sO)$ and so that $X^{H,\x_H}_{\sO}[T] = X^{H_{\ell},\x_{H_\ell}}_{\sO}[T]$ on the event $\{V_{H_{\ell}}\supseteq U_K\}$, which immediately implies \eqref{pfwp:TS} and hence, Claim 3. In fact, we will show that the recursive construction \eqref{eq-loctk}-\eqref{eq-locUk} is such that the following two claims are true. \\ 

\noindent \emph{Claim 3A:} For every $k \in \Nb_0$, $U_k \subseteq \Cc_T(\sO)$. 

\noindent \emph{Claim 3B:} For every $k \in \Nb_0$, 
\begin{equation} 
\label{eq-loc4b} 
X^{H,\x_{H}}_{U_k}[\tau_k) = X^{H_{\ell},\x_{H_{\ell}}}_{U_k}[\tau_k) \Rightarrow X^{H,\x_{H}}_{\sO}[T] = X^{H_{\ell},\x_{H_{\ell}}}_{\sO}[T] \mbox{ a.s. on } \{ U_k \subseteq V_{H_{\ell}}\}. 
\end{equation}

We first show how Claim 3 follows from the auxiliary claims. Since $\{\tau^v_i\}$ are events of $\poiss^{G,T}_v$, \eqref{eq-loctk}-\eqref{eq-locUk} imply that the set $\{\tau_k\}_{k \in \Nb}\setminus \{0\}$ is contained in the events of $\poiss^{G,T}_{\cup_{k \in \Nb_0} U_k}$. In view of Claim 3A, on the event $\{\Cc_T(\sO) \subseteq V_{H_{\ell}}\}$, the set of events of $\poiss^{G,T}_{H_\ell}$ contains that of $\poiss^{G,T}_{\cup_{k \in \Nb_0} U_k}$ and hence contains $\{\tau_k\}_{k \in \Nb}\setminus \{0\}$. By the definition \eqref{def-K} of $K$, since $H_{\ell}$ is finite, this implies $K\leq 1 + \sum_{v \in V_{H_{\ell}}} \poiss^{G,T}_v(0,T] <\infty$ a.s.. Also, by \eqref{def-K}, \eqref{eq-loctk}, and \eqref{eq-locUk} this implies that a.s. on the event $\{\Cc_T(\sO) \subseteq V_{H_{\ell}}\}$, $\tau_K = 0$, $U_K = \bigcup_{k \in \Nb_0} U_k$, and (applying Claim 3A with $k=K$) $U_K \subseteq\Cc_T(\sO)\subseteq V_{H_{\ell}}\subseteq U$. Together, these properties imply, $X^{H,\x_{H}}_{U_K}[\tau_K) = \x_{U_K} = X^{H_{\ell},\x_{H_{\ell}}}_{U_K}[\tau_K)$. Invoking Claim 3B with $k = K$, we see that a.s. on $\{U_K \subseteq V_{H_{\ell}}\} \supseteq \{\Cc_T(\sO) \subseteq V_{H_{\ell}}\}$, we have $X^{H,\x_{H}}_{\sO}[T] = X^{H_{\ell},\x_{H_{\ell}}}_{\sO}[T]$. This proves \eqref{pfwp:TS} and hence, Claim 3 follows.

We first provide a rough idea of the proof of the auxiliary claims. When $k = 0$, $U_k = \sO$ and it is easy to see that claims 3A and 3B hold trivially (due to the stipulation that $\sO \subseteq \Cc_T(\sO)$ and the assumption that there is no jump at $T$). If there is any jump in the driving processes $\poiss_{\sO}^{G,T}$ in $[0,T)$, then by the construction \eqref{eq-loctk}-\eqref{eq-locUk}, the most recent event before $\tau_0 = T$ that could have influenced the value of the process $X^{H,\x_{H}}$ at $\tau_0$ occurs at $\tau_1$ corresponding to a transition at the vertex $v_1$. The local nature of the dynamics implies that this transition is influenced by the particles in $\cl{v_1}$ so that the trajectory of $X^{H,\x_{H}}$ up to time $\tau_1$ is influenced by the trajectories of the particles in $\sO \cup \cl{v_1} = U_0\cup\cl{v_1} = U_1$ before time $\tau_1$. Any vertex $u$ in $U_1\setminus \sO$ belongs to $\neigh{v_1}$ and so $(u,v_1)$ forms a $(G,\poiss^{G,T})$ causal chain on the interval $[\tau_1, T]$, showing that $U_1 \subset \Cc_T(\sO)$. Claims 3A and 3B follow by proceeding inductively in this manner, tracing backward in time the $(G, \poiss^{G,T})$-causal chains that end in $\sO$. We now provide fully rigorous proofs of the auxiliary claims.\\

\noindent \emph{Proof of Claim 3A:} We prove the following assertion using induction: For any $k \in \Nb_0$ and $v \in U_k$, $v \rightsquigarrow_{\tau_{k+1},T} \sO$. This is obviously true in the case $k = 0$, in which case $U_0 = \sO$ and for any $v \in \sO$, trivially $v \rightsquigarrow_{\tau_1,T} v\in \sO$ by Definition \ref{pfwp:causal} of a causal chain. Now suppose that for some $k \in \Nb_0$, we have $v \rightsquigarrow_{\tau_{k},T} \sO$ for all $v \in U_{k-1}$. Then because $\tau_{k+1} \leq \tau_{k}$, it follows that $v \rightsquigarrow_{\tau_{k+1},T} \sO$. There are now two cases to consider. In the first case, $U_{k} = U_{k-1}$ in which case the conclusion $v \rightsquigarrow_{\tau_{k+1},T} \sO$ for all $v \in U_{k}$ holds trivially. In the second case, by \eqref{eq-locvk} and \eqref{eq-locUk}, $U_k\setminus U_{k-1} \subset \neigh{v_k}$, where $v_k \in U_{k-1}$, and furthermore, $\tau_k > 0$ and $\poiss^{G,T}_{v_k}(\{\tau_k\}) = 1$. Since $v_k \in U_{k-1}$, the induction assumption implies there exists a $(G,\poiss^{G,T})$-causal chain $\Gamma = (v_k\defeq u_0,u_1,\dots,u_n \in \sO)$ from $v_k$ to $\sO$ contained in the time interval $[\tau_k,T]$ with corresponding times $\tau_k \defeq \theta_0 < \theta _1<\cdots< \theta_n\leq T$. Then because $\poiss^{G,T}_{v_k}(\{\tau_k\}) = 1$, it follows that $(\tau_{k+1} = \theta < \theta_0 < \cdots < \theta_n)$ is a sequence of times such that for all $i = 0,\dots,n$, $\poiss^{G,T}_{u_i}(\{\theta_i\}) = 1$, and for any $w \in U_k\setminus U_{k-1}$, $(w,v_k,u_1,\dots,u_n)$ is a $(G,\poiss^{G,T})$-causal chain ending in $\sO$. Thus, $w \rightsquigarrow_{\tau_{k+1},T} \sO$ for all $w \in U_k$. By induction, we may conclude that for any $k \in \Nb_0$ and $v \in U_k$, $v\rightsquigarrow_{\tau_{k+1},T} \sO$, which proves the assertion. Now, by Definition \ref{pfwp:causal}, the assertion in turn implies $v\rightsquigarrow_{0,T} \sO$. By \eqref{loc-AT}, this implies that $U_k\subseteq \Cc_T(\sO)$ for all $k\in \Nb_0$, concluding the proof of claim 3A.\\

\noindent \emph{Proof of Claim 3B:} We will again use an argument by induction. First, note that the base case $k=0$ in \eqref{eq-loc4b} is true because $U_0 = \sO$, $\tau_0=T$, and both $X^{H,\x_H}_{\sO}$ and $X^{H_{\ell},\x_{H_{\ell}}}_{\sO}$ a.s. do not have a jump at $T$ because $\poiss^{G,T}_{\sO}(\{T\}) = 0$ a.s.. Now, suppose \eqref{eq-loc4b} holds for $k = m-1$, for some $m \in \Nb$. To show it holds for $k=m$, we first argue that it suffices to show that 
\begin{equation} 
\label{eq-loc4b2} 
X^{H,\x_H}_{U_{m}}[\tau_{m}) = X^{H_{\ell},\x_{H_{\ell}}}_{U_{m}}[\tau_{m}) \quad \Rightarrow \quad X^{H,\x_H}_{U_{m-1}}[\tau_{m-1}) = X^{H_{\ell},\x_{H_{\ell}}}_{U_{m-1}}[\tau_{m-1}) \te{ a.s. on }\{U_m \subseteq V_{H_{\ell}}\}. 
\end{equation} 
Indeed, $U_{m-1} \subseteq U_m$ implies $U_{m-1} \subseteq V_{H_{\ell}}$ on the event $\{U_{m} \subseteq V_{H_{\ell}}\}$, and so \eqref{eq-loc4b2} and \eqref{eq-loc4b}, the latter with $k=m-1$, shows that \eqref{eq-loc4b} holds for $k =m$. Claim 3B then follows by induction. 

To establish \eqref{eq-loc4b2}, assume $U_m \subseteq V_{H_{\ell}}$ and $X^{H,\x_H}_{U_{m}}[\tau_{m}) = X^{H_{\ell},\x_{H_{\ell}}}_{U_{m}}[\tau_{m}).$ Since $U_{m-1} \subset U_m$, this implies $X^{H,\x_H}_{U_{m-1}}[\tau_{m}) = X^{H_{\ell},\x_{H_{\ell}}}_{U_{m-1}}[\tau_{m})$. Moreover, from \eqref{eq-loctk} note that $\tau_m$ is the largest time prior to $\tau_{m-1}$ that there is an event for any of the Poisson processes in $U_{m-1}$ and hence, $\sum_{v \in U_{m-1}} \poiss^{G,T}\poissv{v}(\tau_m,\tau_{m-1}) = 0$. The form of the SDE \eqref{mod:infpart} then implies that both $X^{H,\x_H}_{U_{m-1}}$ and $X^{H_{\ell},\x_{H_{\ell}}}_{U_{m-1}}$ are constant on $(\tau_m,\tau_{m-1})$. Thus, to establish \eqref{eq-loc4b2}, it suffices to show that $X^{H,\x_H}_{U_{m-1}}(\tau_m) = X^{H_{\ell},\x_{H_{\ell}}}_{U_{m-1}}(\tau_m)$ a.s.. Now, by \eqref{eq-loctk}-\eqref{eq-locvk}, $v_m$ is the only vertex in $U_{m-1}$ such that $\poiss^{G,T}_{v_m}(\{\tau_m\}) = 1$. Thus, 
\begin{equation} 
\label{eq-loc4b3} 
X^{H,\x_H}_{U_{m-1}\setminus\{v_m\}}(\tau_m) = X^{H,\x_H}_{U_{m-1}\setminus\{v_m\}}(\tau_m-) = X^{H_{\ell},\x_{H_{\ell}}}_{U_{m-1}\setminus\{v_m\}}(\tau_m-) = X^{H_{\ell},\x_{H_{\ell}}}_{U_{m-1}\setminus\{v_m\}}(\tau_m),
\end{equation} 
and so it only remains to show that $X^{H,\x_H}_{v_{m}}(\tau_{m}) = X^{H_{\ell},\x_{H_{\ell}}}_{v_{m}}(\tau_{m})$ a.s.. Since for $m \in \Nb$, $\cl{v_m}(G) \subset U_m$ by \eqref{eq-locUk}, and $U_m \subseteq V_{H_{\ell}}\subseteq U$ by assumption, $\cl{v_m}(H_{\ell}) = \cl{v_m}(G) =: \cl{v_m}$. Then the assumption $X^{H,\x_H}_{U_{m}}[\tau_{m}) = X^{H_{\ell},\x_{H_{\ell}}}_{U_{m}}[\tau_{m})$ implies $X^{H,\x_H}_{\cl{v_m}}[\tau_m) = X^{H_{\ell},\x_{H_{\ell}}}_{\cl{v_m}}[\tau_m)$ a.s. The locality and predictability (see Definition \ref{def-regular}) of the jump rates stated in the standing assumption then imply that for $j \in \jmps$, denoting by $(\ems,\vms)$ the marks of $G\subg{\cl{v_m}}$ we have 
\[\rate\gvpara{H}{v_m}\stpara{j}(s,X^{H,\x_H}) = \locrate^{H\subg{\cl{v_m}}}_j(s,X^{H,\x_H}_{\cl{v_m}}, \ems, \vms) = \rate\gvpara{H_{\ell}}{v_m}\stpara{j}(s,X^{H_{\ell},\x_{H_{\ell}}}), \quad s \in [0,\tau_m].\] 
Due to the form of the SDE \eqref{mod:infpart}, this shows $X^{H,\x_H}_{v_{m}}(\tau_{m}) = X^{H_{\ell},\x_{H_{\ell}}}_{v_{m}}(\tau_{m})$, as desired. This completes the proof of Claim 3B, and therefore of the first assertion of the proposition.\\

We now turn to the proof of the second assertion of the proposition. Let $\{G_n\}$ be a sequence of deterministic finitely dissociable $\sp{\emksp,\vmksp}$-graphs, and for each $n \in \Nb$, let $\poiss^{G_n}$ be a driving noise compatible with $G_n$. Then define $\Cc^n_{t_1,t_2}$ analogously to $\Cc_{t_1,t_2}$, but with $G$ and $\poiss^{G}$ replaced by $G_n$ and $\poiss^{G_n}$, respectively, but using the same family of constants ${\bf C}$ to define $\poiss^{G,T}$ and $\poiss^{G_n,T}$, $T \in \Rb_+$. By the first assertion of the proposition just established above, for each $n$, $\Cc^n_T$ is a localizing map of the SDE \eqref{mod:infpart} on $(G_n,\poiss^{G_n})$. It follows (e.g., by Remark \ref{rem-locdef}) that $\Oo_T(\cdot;G_n,\poiss^{G_n}) \defeq \gcl{\Cc^n_T(\cdot)}{G_n}$ is likewise a localizing map of the SDE \eqref{mod:infpart} on $(G_n,\poiss^{G_n})$. Note that $\Oo_T(\cdot;G_n,\poiss^{G_n})$ satisfies \eqref{smap-meas} as a consequence of the fact that $\{\Cc^n_T(\sO) \subset U\}$ lies in $\filtm_T^{\poiss^{G_n}}$ for every $\sO, U \in \fset{G_n}$. With this choice of localizing maps $\Oo_T(\cdot;G_n, \poiss^{G_n})$, clearly property 1 of Definition \ref{WP:loccons} holds.

It only remains to show that property 2 of Definition \ref{WP:loccons} is also satisfied. Fix $n ,n',\ell \in \Nb$ such that there exists an isomorphism $\varphi\in I(\trnc{\ell}(\nm{G_{n',*}}), \trnc{\ell}(\nm{G_{n,*}}))$. Define the $\filtm_0$-measurable event 
\[\Imc_{\varphi} \defeq \{\varphi\in I(\trnc{\ell}(\nm{G_{n',*}},\poiss^{G_{n'}}), \trnc{\ell}(\nm{G_{n,*}},\poiss^{G_{n}}))\},\] 
and for notational conciseness, set $U := \trnc{\ell-1}(G_n)$ and $U' := \trnc{\ell-1}(G_{n'})$. Let $\sO_n\subset V_{G_n}$ and $\sO_{n'}\subset V_{G_{n'}}$ be any pair of sets such that $\varphi(\sO_{n'}) = \sO_n$. Then it suffices to prove that on the event $\Imc_{\varphi}\cap \{U\supseteq \Cc^n_T(\sO_n)\} = \Imc_{\varphi}\cap\{\trnc{\ell}(G_n)\supseteq \Oo_T(\sO_n;G_n,\poiss^{G_n})\}$, $\Cc^n_T(\sO_n) = \varphi(\Cc^{n'}_T(\sO_{n'}))$ a.s..

For any $m\in \Nb$, and any $(G_m,\poiss^{G_m,T})$-causal chain $\Gamma$ up to time $T$ ending in $\sO_m$, we refer to $\Gamma$ as a $G_m$-causal chain. We now claim that $\varphi$ induces a bijection between $G_n$-causal chains and $G_{n'}$-causal chains. By the definition of $\Cc^n_T(\sO_n)$, it follows that $\Gamma \subseteq U$ on the event $\Imc_{\varphi}\cap \{U\supseteq \Cc^n_T(\sO_n)\}$. Moreover, on the event $\Imc_{\varphi}$, $\varphi^{-1}|_U \in I((\nm{G_{n,*}\subg{U}},\poiss^{G_n}_U),(\nm{G_{n',*}\subg{U'}},\poiss^{G_{n'}}_{U'}))$ so $\varphi^{-1}(\Gamma)$ is a $G_{n'}$-causal chain (noting that on this event, $\poiss^{G_{n'},T}_{U'} = (\poiss^{G_n,T}_{\varphi(v)})_{v \in U'}$ as a consequence of the fact that $\poiss^{G_{n'},T}$ and $\poiss^{G_n,T}$ are constructed using the same family of constants $\mathbf{C}$). In the other direction, suppose $\Gamma' \defeq (v' = u'_0,u'_1,\dots,u'_m \in \sO_{n'})$ is a $G_{n'}$-causal chain. If $\Gamma' \subseteq U'$, then by the same argument, $\varphi(\Gamma')$ is a $G_n$-causal chain. On the other hand, suppose there exists $k \in \{0,1,\dots,m\}$ such that $u'_k \notin U'$. Note that $u'_m \in \sO_{n'} \subseteq U'$, so $k < m$. Choose the maximal $k$ such that $u'_k \notin U'$ and note that then $\{u'_{k+1},\dots,u'_m\}\subseteq U'$. It immediately follows that $\ov{\Gamma}' \defeq (u'_k,\dots,u'_m)$ is also a $G_{n'}$-causal chain. Furthermore, $u'_k \in \gneigh{U'}{G_{n'}}$, so on the event $\Imc_{\varphi}$, $\varphi(\ov{\Gamma}')$ is a $G_n$-causal chain and $\varphi(\ov{\Gamma}') \nsubseteq U$, which contradicts our assumption that $\Cc^n_T(\sO_n) \subseteq U$. Thus, for any $G_{n'}$-causal chain $\Gamma'$ ending in $\sO_{n'}$, $\varphi(\Gamma')$ is a $G_n$-causal chain ending in $\sO_n$. This proves the claim.

Because $\varphi$ a.s. induces a bijection between causal chains in $G_n$ and $G_{n'}$ that end in $\sO_n$ and $\sO_{n'}$ respectively, it follows that $\Cc_T^n(\sO_n) = \varphi(\Cc_T^{n'}(\sO_{n'}))$ on the event $\Imc_{\varphi}\cap\{\Cc_T^n(\sO_n) \subseteq U\}$. Equivalently, $\Oo_T(\sO_n;G_n,\poiss^{G_n}) = \varphi(\Oo_T(\sO_{n'};G_{n'},\poiss^{G_{n'}}))$ on the event $\{ \Oo_T(\sO_n;G_n,\poiss^{G_n})\subseteq \trnc{\ell}(G_n)\}\cap \Imc_{\varphi}$. It follows that property 2 of Definition \ref{WP:loccons} also holds for the sequence $\{G_n\}_{n \in \Nb}$. 

This proves the second assertion and, hence, concludes the proof of the proposition. 
\end{proof}

\section{Local Convergence of the IPS} 
\label{LWCpf}
The main goal of this section is to prove the local convergence result of Theorem \ref{LWC:LWC}. First, consider the case when the initial data $\ic{(G_n,\ems^n,\vms^n,\x^n)}, n \in \Nb$, and $\ic{(G,\ems,\vms,\x)}$ are finite and have deterministic and equal unmarked representatives, and thus differ only in their vertex and edge marks and initial states. Suppose also that the corresponding SDEs in \eqref{mod:infpart} are coupled so that they are driven by the same Poisson processes for different $n$. Then Theorem \ref{LWC:LWC} follows from pathwise continuity of the dynamics of the SDE \eqref{mod:infpart} with respect to the vertex and edge marks of the initial data, which is essentially a consequence of the continuity condition in Assumption \ref{mod:cont}. In the more general case when the unmarked graphs need not coincide but each IPS in the sequence and the limit IPS, all still have finite initial data, the proof entails carefully constructed couplings involving graph isomorphism (see Assumption \ref{LWCpf:finconv} and Lemma \ref{LWCpf:nograph} below). In the fully general setting, when the limit graph could be infinite, the proof proceeds by first using consistent spatial localization to reduce the analysis to the finite initial data setting and then invoking the finite initial data results. While the localizing maps associated with consistent spatial localization (see Definition \ref{WP:loccons}) only make sense on graphs rather than their equivalence classes, local convergence results are defined in terms of equivalence classes. To bridge this gap, we need to carefully select suitable representatives of equivalence classes, as well as establish correspondences between statements about the convergence of equivalence classes and statements about representative graphs. This necessitates some technicalities related to representative graphs and associated driving noises, which we introduce in Section \ref{subs-coupling}. But it is not hard to see from our arguments that if one specializes to the case of convergent graphs whose vertex sets are subsets of a fixed countable set (see the space $\wh{\gs}\sp{\ov{\Pol},\Pol}$ defined in \eqref{LWC:whgs} below), then we can do away with many of these technicalities. 

Section \ref{subs-coupling} defines canonical representatives with a view to constructing suitable couplings. These are then used in Section \ref{subs-pflocconv} to establish a more general almost sure local convergence result in Proposition \ref{LWCpf:LWC}. The latter result is used in Section \ref{subs-finalpfLWC} to prove  the local convergence result (Theorem \ref{LWC:LWC}), and also in Section \ref{GEM} in the proof of the hydrodynamic limit (Theorems \ref{GEM:GEMconv} and \ref{GEM:empconv}). 

\subsection{Canonical Representative Graphs and Consistent Extensions} 
\label{subs-coupling}
Let $\ov{\Pol},\Pol$ be Polish spaces. We begin by defining a space of $\sp{\ov{\Pol},\Pol}$-graphs: 
\begin{equation} 
\label{LWC:whgs} 
\wh{\gs}\sp{\ov{\Pol},\Pol} = \{\sp{\ov{\Pol},\Pol}\te{-graphs }G \defeq (V,E,\root,\ov{\dvms},\dvms)\te{ s.t. } V \subseteq \Nb\}. 
\end{equation} 
Properties of this space are elucidated in Appendix \ref{msbl}. In particular, it follows from Lemmas \ref{msbl:Pol}, \ref{msbl:alphcont}, and \ref{msbl:rep} that $\wh{\gs}\sp{\ov{\Pol},\Pol}$ can be equipped with a Polish topology that is compatible with the topology of $\gs\sp{\ov{\Pol},\Pol}$, and that for any $\gs\sp{\ov{\Pol},\Pol}$-random element $\ic{G}$, there exists a $\sigma(\ic{G})$-measurable $\wh{\gs}\sp{\ov{\Pol},\Pol}$-random element $G$ such that $G \in \ic{G}$ almost surely. In the latter case, $G$ is referred to as a $\wh{\gs}\sp{\emksp,\vmksp}$-random representative of $\ic{G}$, and thus, $\wh{\gs}\sp{\ov{\Pol},\Pol}$ can be viewed as a canonical space of measurable representative graphs compatible with the local topology. We begin with the definition of a representative convergent sequence. 

\begin{definition}[Representative convergent sequences] 
\label{MG:rcs} 
\sloppy Let $\ic{(G_n,\x^n)}$ be a random sequence converging a.s. to $\ic{(G,\x)}$ in $\gs\sp{\emksp,\vmksp\times \Xc}$ on some complete probability space. Then a \emph{representative convergent sequence} (henceforth abbreviated to \emph{rep-con sequence}) of $(\{\ic{(G_n,\x^n)}\}_{n \in \Nb}, \ic{(G,\x)})$ is a $\sigma(\{\ic{(G_n,\x^n)}\}_{n\in\Nb},\ic{(G,\x)})$-measurable tuple $\left(\{(G_n,\x^n),M_n\}_{n\in \Nb}, (G,\x),\{\varphi_{n,m}\}_{n \in \Nb,m\leq M_n}\right)$, defined on the same probability space, that satisfies the following properties: 
\begin{enumerate}[1.] 
\item for each $n \in \Nb$, $(G_n,\x^n) = (V_n,E_n,\root_n,\ems^n,\vms^n, \x^n)$ is a $\wh{\gs}\sp{\emksp,\vmksp\times \Xc}$-random representative of $\ic{(G_n,\x^n)}$; 
\item $(G,\x) = (V,E,\root,\ems,\vms,\x)$ is a $\wh{\gs}\sp{\emksp,\vmksp\times \Xc}$-random representative of $\ic{(G,\x)}$; 
\item $\lim_{n\to\infty} M_n = \infty$ a.s. and $\trnc{m}(\nm{G_*},\x) \cong \trnc{m}(\nm{G_{n,*}},\x^n)$ on the event $\{m \leq M_n\}$; 
\item for each $n\in\Nb, m\leq m' \leq M_n$, $\varphi_{n,m}$ lies in $I(\trnc{m}(\nm{G_{*}},\x),\trnc{m}(\nm{G_{n,*}},\x^n))$, and $\varphi_{n,m'}|_{\trnc{m}(G)} = \varphi_{n,m}$; 
\item for every $n', m \in \Nb$ such that $m \leq M_{n'}$ and $e \in E_{\trnc{m}(G)}$, $\lim_{n \to \infty,n > n'} \ems^n_{\varphi_{n,m}(e)} = \ems_e$; 
\item for every $n', m \in \Nb$ such that $m \leq M_{n'}$ and $v \in \trnc{m}(G)$, $\lim_{n \to \infty,n > n'} \vms^n_{\varphi_{n,m}(v)} = \vms_v$. 
\end{enumerate} 
\end{definition}

In Definition \ref{MG:rcs}, $M_n$ represents the size of a random neighborhood of the root for which the unmarked versions of $G_n$ and $G$ are isomorphic (in the notation of Definition \ref{def-locconvnoiso} of local convergence, $M_n$ is any random integer bounded from above by $\sup\{m: n \geq n_m\}$). Likewise $\varphi_{n,m}$ is a random isomorphism of the $m$-neighborhoods of the roots of $G_n$ and $G$ (analogous to the deterministic isomorphisms $\varphi_{n,m}$ introduced in Definition \ref{def-locconvnoiso}), and property 4 imposes a consistency property on the sequence of random isomorphisms. 

The next lemma shows that there is a canonical way to identify rep-con sequences. A constructive proof of the lemma, which leverages the existence of suitably measurable representative graph sequences, isomorphisms, and driving maps, is deferred to Appendix \ref{subs-Cmain}.

\begin{lemma}[Existence of rep-con sequences] 
\label{MG:iso} 
Fix a complete probability space $\pzspace$ that supports the random sequence $\ic{(G_n,\x^n)}$ converging a.s. to $\ic{(G,\x)}$ in $\gs\sp{\emksp,\vmksp\times \Xc}$. Then there exists an $\wh{\filtm}$-measurable sequence of $\wh{\gs}\sp{\emksp,\vmksp\times\Xc}$-random elements $\{(\ov{G}_n,\ov{\x}^n)\}_{n \in\Nb},(\ov{G},\ov{\x})$ satisfying Properties 1 and 2 of Definition \ref{MG:rcs}. In addition, given any sequence $\{(G_n,\x^n)\}_{n \in\Nb},(G,\x)$ satisfying Properties 1 and 2 of Definition \ref{MG:rcs}, the probability space $\pzspace$ also supports a rep-con sequence $\left(\{(G_n,\x^n),M_n\}, (G,\x),\{\varphi_{n,m}\}\right)$. 
\end{lemma}

To extend the notion of rep-con sequences from initial data consisting of $\wh{\gs}\sp{\emksp,\vmksp\times \Xc}$-random elements to the corresponding IPS characterized by $\wh{\gs}\sp{\emksp,\vmksp\times \cad}$-random elements, we will need a common probability space on which we can define both the driving noise and a rep-con sequence. As mentioned in Remark \ref{mod:coupling}, this driving noise is a key element of the coupling we later construct to prove local convergence of the IPS.

\begin{definition}[Consistent representative convergent extensions] 
\label{MG:extdef} 
Given a complete probability space $\pzspace$ that supports a sequence $\ic{(G_n,\x^n)}$ converging a.s. to $\ic{(G,\x)}$ in $\gs\sp{\emksp,\vmksp\times \Xc}$, a \emph{consistent representative convergent extension} (henceforth abbreviated to consistent rep-con extension) of $(\pzspace,\{\ic{(G_n,\x^n)}\}_{n\in\Nb},\ic{(G,\x)})$ is a 4-tuple $(\fpspace,\{(G_n,\x^n,\poiss^{G_n}), M_n\}_{n \in \Nb}$, $(G,\x,\poiss^G),\{\varphi_{n,m}\}_{n \in \Nb,m\leq M_n})$ such that 
\begin{enumerate}[1.] 
\item $(\{(G_n,\x^n),M_n\},(G,\x),\{\varphi_{n,m}\})$ is a rep-con sequence of $(\{\ic{(G_n,\x^n)}\}_{n \in \Nb},\ic{(G,\x)})$; 
\item $\fpspace$ is a complete extension of the probability space $\pzspace$ such that $\filt$ satisfies the usual conditions and $(G_n,\poiss^{G_n})$, $n \in \Nb,$ and $(G,\poiss^G)$, respectively, are $\filt$-driving noises that are compatible with $G_n$, $n \in \Nb,$ and $G$, as defined in Definition \ref{mod:dnoise}; 
\item for each $n \in \Nb$ and $v \in \trnc{M_n}(G)$, $\poiss^{G_n}_{\varphi_{n,M_n}(v)} = \poiss^G_v$ a.s.. 
\end{enumerate} 
\end{definition}

The construction of consistent rep-con extensions is facilitated by the use of so-called \emph{driving maps} defined below. Let $\mps$ be the space of maps from subspaces $W\subseteq \Nb$ to $\Nb$, which can be equipped with a Polish topology by Remark \ref{msbl:mpspace}.

\begin{definition}[Driving maps] 
\label{MG:driving} 
\sloppy Given a $\wh{\gs}\sp{\emksp,\vmksp\times \Xc}$-random element $(G,\x) \defeq (V,E,\root,\ems,\vms,\x)$, a \emph{driving map} is a random injective map $\psi: V \to \Nb$, that is, $\psi$ is a $\sigma(G,\x)$-measurable random element taking values in $\mps$. Suppose $\fpspace$ is a filtered probability space supporting a collection of i.i.d. $\filt$-Poisson processes $\{\poiss_k\}_{k \in \Nb}$ on $\Rb_+^2\times\jmps$ with intensity $\leb^2\otimes\Sm$. Suppose also that $\filt$ satisfies the usual conditions and $(G,\x)$ is $\filtm_0$-measurable. Then the $\sp{\emksp,\vmksp\times\Nms(\Rb_+^2\times\jmps)}$-random graph $(G,\poiss^G)$ defined by $\poiss^G_v = \poiss_{\psi(v)}$ is said to be an $\filt$-driving noise generated by $\psi$. 
\end{definition}

\begin{remark}[Driving maps generate driving noise] 
\label{MG:dpoiss}
We now justify our reference to $(G,\poiss^G)$ as a $\filt$-driving noise in Definition \ref{MG:driving}. It is easy to see that on any complete, filtered probability space $\fpspace$ supporting the $\filtm_0$-measurable initial data $(G,\x)$ and $\filt$-Poisson processes $\{\poiss_k\}_{k \in \Nb}$, the $\sigma(G,\x)$-measurability (and thus $\filtm_0$-measurability) of driving maps ensures that for any $\filtm_0$-measurable $v \in V_G$, $\poiss^G_v$ is $\filt$-adapted and therefore supported on $\fpspace$, and condition 3 of Definition \ref{mod:dnoise} is satisfied. Furthermore, conditioned on $\filtm_0$, $\poiss^G$ is a collection of i.i.d. Poisson processes indexed by $V_G$. Thus, condition 1 of Definition \ref{mod:dnoise} is satisfied. Lastly, because $\{\poiss_k\}_{k \in \Nb}$ are $\filt$-Poisson processes, for any $t > 0$ and $A \in \borel((t,\infty)\times\Rb_+\times\jmps)$, $(G,\poiss^G(A))$ is conditionally independent of $\filtm_t$ given $\filtm_0$ so condition 2 of Definition \ref{mod:dnoise} is satisfied. Assuming $\filt$ satisfies the usual conditions, all three conditions of Definition \ref{mod:dnoise} are satisfied.
\end{remark}

The following lemma shows that there always exists a consistent rep-con extension of any convergent sequence of random elements in $\gs\sp{\emksp,\vmksp\times\Xc}$. We prove this at the end of Appendix \ref{msbl}.

\begin{lemma}[Existence of consistent rep-con extensions] 
\label{MG:repextexist} 
\sloppy Let $\{\ic{(G_n,\x^n)}\}_{n \in \Nb}$ be a random sequence on $\pzspace$ that converges a.s. to $\ic{(G,\x)}$ in $\gs\sp{\emksp,\vmksp}$. Then there exists a consistent rep-con extension $(\fpspace,\{(G_n,\x^n,\poiss^{G_n}),M_n\}_{n\in\Nb}, (G,\x,\poiss^G),\{\varphi_{n,m}\}_{n\in\Nb,m\leq M_n})$ of $(\pzspace,\{\ic{(G_n,\x^n)}\}_{n \in \Nb},\ic{(G,\x)})$ such that $\fpspace$ supports a collection of i.i.d. $\filt$-Poisson processes $\poiss \defeq \{\poiss_k\}_{k \in \Nb}$ and driving maps $\psi_n,$ $n \in \Nb$, and $\psi$ that generate the respective $\filt$-driving maps $\poiss^{G_n}$, $n \in \Nb$, and $\poiss^G$. 
\end{lemma}

Consistent rep-con extensions are useful because, as shown in Proposition \ref{LWCpf:LWC} of the next section, under the conditions of Theorem \ref{LWC:LWC}, a sequence of IPS will converge almost surely if it is generated by a consistent rep-con extension in the following sense. 

\begin{definition}[IPS sequences generated by consistent rep-con sequences] 
\label{MG:genX} 
A sequence of IPS $\{(G_n,X^{G_n,\x^n})\}_{n \in \Nb},(G,X^{G,\x})$ is said to be generated by a consistent rep-con extension $(\fpspace,\{(G_n,\x^n,\poiss^{G_n}),M_n\}_{n \in \Nb},(G,\x,\poiss^G),\{\varphi_{n,m}\}_{n \in \Nb,m\leq M_n})$ if for each $n \in \Nb$, $(G,X^{G_n, \xi^n})$ is an $\poiss^{G_n}$-strong solution to \eqref{mod:infpart} for $(G_n, \x^n)$ and, likewise, $(G,X^{G,\xi})$ is an $\poiss^G$-strong solution to \eqref{mod:infpart} for $(G,\xi)$. 
\end{definition}

\begin{remark}[Reduction from random to deterministic initial data] 
\label{MG:cond} 
If $\{(G_n,X^{G_n,\x^n})\}_{n \in \Nb}$, $(G,X^{G,\x})$, are IPS generated by a consistent rep-con extension of $(\pzspace,\{\ic{(G_n,\x^n)}\}_{n \in \Nb}, \ic{(G,\x)})$, then for almost every $\omega \in \Omega$, $\law(\{(G_n,X^{G_n,\x^n})\}_{n \in \Nb}, $ $(G,X^{G,\x})|\filtm_0)(\omega)$ describes the law of a sequence of IPS generated by a consistent rep-con extension of a tuple that contains the terms $\{\ic{(G_n,\x^n)}(\omega)\}_{n \in \Nb}$ and $\ic{(G,\x)}(\omega)$. A rigorous justification of this would follow along the same lines as Lemma \ref{mod:condwp} and is thus omitted. 
\end{remark}

\subsection{Proof of Local Convergence in the almost sure setting} 
\label{subs-pflocconv} 
As outlined at the beginning of Section \ref{LWCpf}, we start by proving convergence for sequences and limits with finite initial data (or, to be more precise, initial data whose underlying graphs have a uniformly bounded radius). 

\begin{lemma}[The convergence of IPS on graphs of uniformly bounded size] 
\label{LWCpf:nograph} 
Suppose the family of local jump rates $\mathbf{\locrate}$ satisfies Assumption \ref{mod:assu}, and $\ic{(G,\x)}$ satisfies Assumption \ref{mod:cont}. If $\ic{(G_n,\x^n)} \to \ic{(G,\x)}$ a.s., then the tuple $(\pzspace,\{\ic{(G_n, \x^n)}\}_{n\in\Nb},\ic{(G, \x)})$ satisfies the following:

\noindent\textbf{Finite Convergence Property:} given any consistent rep-con extension $(\fpspace, \{(G_n,\x^n,\poiss^{G_n}),M_n\}_{n \in \Nb}, (G,\x,\poiss^G),\{\varphi_{n,m}\}_{n \in \Nb,m\leq M_n})$ of $(\pzspace, \{\ic{(G_n, \x^n)}\}_{n\in\Nb},\ic{(G, \x)})$, $T \in \Rb_+$, and $m \in \Nb$, there exists an a.s. finite, $\filtm_T$-measurable random variable $\ov{N}_m \defeq \ov{N}_{m,T}$ such that for every $n \in \Nb$, 
\begin{equation} 
\label{LWCpf:Gcong} 
\left(\trnc{m}(\nm{G_{n,*}}),X^{m,n}[T]\right) \cong \left(\trnc{m}(\nm{G_{*}}),X^{m,\infty}[T]\right) \te{ a.s. on the event }\{n \geq \ov{N}_m\}\cap \{m \leq M_n\}, 
\end{equation} 
where $X^{m,n}$ and $X^{m,\infty}$ are the respective $\poiss^{G_n}_{\trnc{m}(G_n)}$- and $\poiss^G_{\trnc{m}(G)}$-strong solutions to \eqref{mod:infpart} for the initial data $\trnc{m}(G_n,\x^n)$ and $\trnc{m}(G,\x)$. 
\end{lemma}

We defer the proof of Lemma \ref{LWCpf:nograph} to the end of the section, and instead first establish local convergence for IPS on graphs with general (possibly infinite) initial data given the conclusion of Lemma \ref{LWCpf:nograph}. To do so, the only consequence of Assumption \ref{mod:cont} that we need is the finite convergence property. With that in mind, we encode the latter property in the following weaker assumption. 

\begin{assumptionp}{\ref{mod:cont}$'$}[Generalization of Assumption \ref{mod:cont}] 
\label{LWCpf:finconv} 
The tuple $(\pzspace,\{\ic{(G_n, \x^n)}\}_{n\in\Nb},$ $\ic{(G, \x)})$ is such that the finite convergence property of Lemma \ref{LWCpf:nograph} holds.
\end{assumptionp}

\begin{remark}[Reduction from random to deterministic initial data] 
\label{LWCpf:cond} 
\sloppy Given Assumption \ref{mod:assu2}, as a consequence of Remark \ref{MG:cond}, if Assumption \ref{LWCpf:finconv} holds for a tuple $(\pzspace,\{\ic{(G_n,\x^n)}\},\ic{(G,\x)})$ then it holds for a.s. every realization of the isomorphism classes $\{\ic{(G_n,\x^n)}\},\ic{(G,\x)}$ when they are random. 
\end{remark}

We now state our general almost sure local convergence result, which holds under the weaker conditions of Assumptions \ref{mod:assu2} and \ref{LWCpf:finconv} (which are implied by Assumptions \ref{mod:assu} and \ref{mod:cont}, respectively), so as to make it applicable in situations where the latter may fail but the former still hold. 

\begin{proposition}[Almost sure local convergence] 
\label{LWCpf:LWC} 
Suppose that the local jump rates ${\bf \bar{\rate}}$ satisfy Assumption \ref{mod:assu2}, the tuple $(\pzspace, \{\ic{(G_n, \x^n)}\}_{n\in\Nb}$, $\ic{(G, \x)})$ satisfies Assumption \ref{LWCpf:finconv} and $\{\ic{G_n}\}_{n\in\Nb},\ic{G}$ a.s. consistently spatially localizes the SDE \eqref{mod:infpart} with local jump rates ${\bf \bar{\rate}}$. Also, let $\{(G_n,X^{G_n,\x^n})\}_{n\in\Nb}, (G,X^{G,\x})$ be a collection of IPS generated by a consistent rep-con extension of the tuple in the sense of Definiton \ref{MG:genX}. If $\ic{(G_n,\x^n)} \to \ic{(G,\x)}$ a.s. in $\gs\sp{\emksp,\vmksp\times \Xc}$, then $\ic{(G_n,X^{G_n,\x^n})} \to \ic{(G,X^{G,\x})}$ a.s. in $\gs\sp{\emksp,\vmksp\times \cad}$. 
\end{proposition} 
\begin{proof} 
\sloppy Let $\Nb_{\infty} \defeq \Nb\cup \{\infty\}$ and let $\ic{(G_{\infty},\x^{\infty})} \defeq \ic{(G,\x)}$, and $(G_{\infty},X^{G_{\infty},\x^{\infty}}) \defeq (G,X^{G,\x})$ respectively. We first prove the proposition under the additional assumption that $\{(G_n,\x^n)\}_{n \in \Nb_{\infty}}$ are deterministic. In this case, let $(\fpspace,\{(G_n,\x^n,\poiss^{G_n}),M_n\}_{n \in \Nb},(G,\x,\poiss^G),\{\varphi_{n,m}\}_{n \in \Nb, m \leq M_n})$ be a consistent rep-con extension of $(\pzspace,\{\ic{(G_n, \x^n)}\}_{n \in \Nb_{\infty}})$, where $(G_{\infty},\x^{\infty}) \defeq (G,\x)$. In addition, let $\poiss^{G_{\infty}} \defeq \poiss^G$. Due to Assumption \ref{mod:assu2} and the assumption of consistent spatial localization, there exists a consistent sequence of localizing maps $\{\Oo_T(\cdot;G_n,\poiss^{G_n})\}_{n \in \Nb_{\infty}}$ for the SDE \eqref{mod:infpart} on $\{(G_n, \poiss^{G_n})\}_{n \in \Nb_{\infty}}$, and by Proposition \ref{WP:WP2}, for each $n \in \Nb_{\infty}$, the SDE \eqref{mod:infpart} is strongly well-posed for $(G_n,\x^n)$. Hence, the $\poiss^{G_n}$-strong solution $X^{G_n,\x^n}$ to \eqref{mod:infpart} for the initial data $(G_n,\x^n)$ is well defined. 

Fix $n', m, M \in \Nb$ such that $m \leq M \leq M_{n'}$. By property 3 of Definition \ref{MG:extdef}, for every $n \in \Nb$, $m \leq M_n$, and $v\in \trnc{m}(G)$, $\poiss^{G_n}\poissv{\varphi_{n,m}(v)} = \poiss^{G}\poissv{v}$. We now apply property 2 of the consistent spatial localization property in Definition \ref{WP:loccons}, with $\sO = \trnc{m}(G)$, $\ell = M$, $n = \infty$, $n' = n'$, $\varphi = \varphi_{n',M}^{-1}$. Then, noting that $\Imc_{\varphi} = \{M \leq M_{n'}\} = \Omega$, we see that
\[\Oo_T(\sO;G,\poiss^G) = \varphi(\Oo_T(\varphi^{-1}(\sO);G_{n'},\poiss^{G_{n'}})),\] 
which implies that 
\[\Oo_T(\varphi_{n',M}(\sO);G_{n'},\poiss^{G_{n'}}) = \varphi_{n',M}(\Oo_T(\sO;G,\poiss^{G}))\] 
a.s. on the event $\{\Oo_T(\sO;G,\poiss^{G}) \subseteq \trnc{M}(G)\}$. Because $\Oo_T(\sO;G,\poiss^{G})$ is a.s. finite, there must exist an $\filtm_T$-measurable a.s. finite random variable $\ov{M}_m$ such that $\Oo_T(\sO;G,\poiss^{G})\subseteq \trnc{\ov{M}_m}(G)$ a.s.. Furthermore, by \eqref{eq-spatloc} this implies that 
\begin{align*} 
X^{G,\x}_{\sO}[T] &= X^{\trnc{M}(G),\x_{\trnc{M}(G)}}_{\sO}[T] \quad \te{and}\quad X^{G_{n'},\x^{n'}}_{\varphi_{n',M}(\sO)}[T] = X^{\trnc{M}(G_{n'}),\x^{n'}_{\trnc{M}(G_{n'})}}_{\varphi_{n',M}(\sO)}[T] 
\end{align*} 
a.s. on the event $\{M \geq \ov{M}_m\}$. By \eqref{LWCpf:Gcong} of Assumption \ref{LWCpf:finconv}, there exists an a.s. finite, $\filtm_T$-measurable random variable $N_M$ such that 
\begin{equation*} 
X^{\trnc{M}(G),\x_{\trnc{M}(G)}}_{\sO}[T] = X^{\trnc{M}(G_{n'}),\x^{n'}_{\trnc{M}(G_{n'})}}_{\varphi_{n',M}(\sO)}[T] \te{ a.s. on the event } \{n' \geq N_M\}. 
\end{equation*} 
The last two displays together show that a.s. on the event $\{M\geq \ov{M}_{m}\} \cap \{n' \geq N_M\}$, 
\begin{align*} 
X^{G,\x}_{\sO}[T] &= X^{\trnc{M}(G),\x_{\trnc{M}(G)}}_{\sO}[T] = X^{\trnc{M}(G_{n'}),\x^{n'}_{\trnc{M}(G_{n'})}}_{\varphi_{n',M}(\sO)}[T] = X^{G_{n'},\x^{n'}}_{\varphi_{n',M}(\sO)}[T]. 
\end{align*} 
Applying the last display for each $M$ satisfying $m \leq M \leq M_{n'}$ and noting by property 1 of Definition \ref{WP:locunif} that $\ov{M}_m \geq m$, it follows that \[X^{G,\x}_{\sO}[T] = X^{G_{n'},\x^{n'}}_{\varphi_{n',\ov{M}_m}(\sO)}[T] \te{ a.s. on the event } \{\ov{M}_m\leq M_{n'}\}\cap \{n' \geq N_{\ov{M}_m}\}.\] 
Sending $n' \to \infty$ and noting that then $M_{n'} \rightarrow \infty$ by property 3 of Definition \ref{MG:rcs}, and $\ov{M}_m$ is a.s. finite, it follows that 
\[\lim_{n'\to\infty} X^{G_{n'},\x^{n'}}_{\varphi_{n',\ov{M}_m}(\sO)}[T] = X^{G,\x}_{\sO}[T] \te{ a.s.. }\] 
This concludes the proof for deterministic sequences. 

The random case can be obtained by conditioning on the initial data. More precisely, suppose $\{\ic{(G_n, \x^n)}\}_{n\in\Nb}$, $\ic{(G, \x)}$ is random and let $\{(G_n,X^{G_n,\x^n})\}_{n \in \Nb}$, $(G,X^{G,\x})$ be the sequence of IPS generated by a consistent rep-con extension of the tuple $(\pzspace, \{\ic{(G_n, \x^n)}\}_{n\in\Nb},\ic{(G, \x)})$. Then by Remark \ref{MG:cond}, for almost every $\omega \in \Omega$, $\law(\{(G_n,X^{G_n,\x^n})\},(G,X^{G,\x})|\filtm_0)(\omega)$ describes the law of a sequence of IPS generated by a consistent rep-con extension of a tuple that contains the terms $\{\ic{(G_n,\x^n)}(\omega)\}_{n \in \Nb}, \ic{(G,\x)}(\omega)$. Moreover, by Remark \ref{LWCpf:cond}, Assumption \ref{LWCpf:finconv} holds a.s. conditioned on $\filtm_0$. Therefore, by the proof of the proposition for deterministic initial data, $\PP(\ic{(G_n,X^{G_n,\x^n})} \to \ic{(G,X^{G,\x})}|\filtm_0) = 1$ a.s.. Therefore, $\ic{(G_n,X^{G_n,\x^n})} \to \ic{(G,X^{G,\x})}$ a.s., as desired. 
\end{proof}

We finish the section with a proof of Lemma \ref{LWCpf:nograph}. 

\begin{proof}[Proof of Lemma \ref{LWCpf:nograph}:] 
Fix a consistent rep-con extension of the given tuple, write $G = (V,E,\root,\ems,\vms)$, $G_n = (V_n,E_n,\root_n,\ems^n,\vms^n)$, $n \in \Nb$, and set $M := \max_{v \in V}d_G(v,\root)$. We first consider the case when $\{G_n\}_{n \in \Nb}$ and $G$ are a.s. finite graphs and additionally assume that $\max_{v \in V_n} d_{G_n}(v,\root_n) \leq M$ for all $n \in \Nb$, and show that for any $T \in \Rb_+$, there exists an a.s. finite, $\filtm_T$-measurable random variable $\ov{N} \defeq \ov{N}_{M,T}$ such that, 
\begin{equation} 
\label{LWCpf:fGcong} 
\left(\nm{G_{n,*}},X^{G_n,\x^n}[T]\right) \cong \left(\nm{G_*},X^{G,\x}[T]\right) \te{ a.s. on the event } \{n \geq \ov{N}\}\cap \{M \leq M_n\}. 
\end{equation}

Note that on the event $A_n \defeq \{M\leq M_n\}$, $(\nm{G_{n,*}},\x^n) \cong (\nm{G_*},\x^n)$. To show \eqref{LWCpf:fGcong}, fix $T \in \Rb_+$ and for each $n \in \Nb$, let $\varphi_n: V \to V_n$ be the $\filtm_0$-measurable map given by $\varphi_n \defeq \varphi_{n,M}$ on the event $A_n$ (on $A_n^c$, we may define $\varphi_n$ to be any measurable function with the appropriate domain and range, for instance, the function that maps all vertices of $G$ to the root of $G_n$). For each $n \in \Nb$ and $v \in V_G$, recall from property 3 of Definition \ref{MG:extdef} that $\poiss^{G_n}\poissv{\varphi_n(v)} = \poiss^G\poissv{v}$ on the event $A_n$. In terms of the family of constants $\{C_{k,T}\}_{k \in \Nb, T > 0}$ from Assumption \ref{mod:assu}, define 
\[\ov{\Emc} = \ov{\Emc}_T \defeq \left\{(v,t,r,j) \in V \times [0,T] \times (0,C_{|\gcl{v}{G}|,T}] \times \jmps: \poiss^G\poissv{v}\left(\{(t,r,j)\}\right) = 1\right\}.\]

Since $|\ov{\Emc}| < \infty$ a.s., we can a.s. order the elements $\{(v_k,\tau_k,r_k,j_k)\}_{k=1}^{|\ov{\Emc}|}$ of $\ov{\Emc}$ such that $\{\tau_k\}_{k \in \Nb}$ is strictly increasing. Note that $\{\tau_k\}_{k \in \Nb}$ is the sequence of points in a time-homogeneous Poisson process and, therefore, a sequence of absolutely continuous random variables. Let $R \defeq \inf\left\{|r_k - \rate\gvpara{G}{v_k}\stpara{j_k}(\tau_k,X^{G,\x})|:k=1,\dots,|\ov{\Emc}|\right\}$. By the predictability of the jump rates $\rate^{G,v_k}_{j_k}$, the random variables $r_k$ and $\rate^{G,v_k}_{j_k}(\tau_k,X^{G,\x})$ are independent and therefore $r_k\neq \rate^{G,v_k}_{j_k}(\tau_k,X^{G,\x})$ a.s. which implies that $R> 0$ a.s.. On the event $A_n$, let 
\[R_n \defeq \sup\left\{\left|\rate\gvpara{G}{v_k}\stpara{j_k}(\tau_k,X^{G,\x}) - \rate\gvpara{G_n}{\varphi_n(v_k)}\stpara{j_k}(\tau_k,X^{G,\x}_{\varphi^{-1}_n(V_n)})\right|:k = 1,\dots,|\ov{\Emc}|\right\},\] 
\sloppy where $X^{G,\x}_{\varphi^{-1}_n(V_n)} = (X^{G,\x}_{\varphi^{-1}_n(v)})_{v \in V_n}$. By property 3 of Definition \ref{MG:rcs}, $M_n$ a.s. diverges to infinity, so $\{A_n\}_{n\in\Nb}$ is a sequence of events such that $\PP(\cap_{n' \geq n}A_{n'}) \to 1$ so that $\lim_{n\to\infty}\indic{A_n} = 1$ a.s.. By Properties 5 and 6 of Definition \ref{MG:rcs}, it follows that for each $v \in V$ and $e \in E$, 
\begin{equation} 
\label{LWCpf:markconv} 
\lim_{n\to\infty}\indi{A_n}(\emsn{n}_{\varphi_n(e)},\vmsn{n}_{\varphi_n(v)}) = (\ems_e,\vms_v) \quad \te{a.s..} 
\end{equation} 

Let $\emsn{n}_{\varphi_n(E)} = (\emsn{n}_{\varphi_n(e)})_{e \in E}$ and $\vmsn{n}_{\varphi_n(V)} = (\vmsn{n}_{\varphi_n(v)})_{v \in V}$. Then using the fact that the local jump rates are class functions (specifically, applying Remark \ref{rem-classfn} with $G_1 = (V,E,\root,\ems^n_{\varphi_n(E)},\vms^n_{\varphi_n(V)})$, $G_2 = G_n$, and $\varphi = \varphi_n^{-1}$) in the first equality below, and combining \eqref{LWCpf:markconv} with the absolute continuity of $\tau_k$ and the fact that $\ic{(G,\x)}$ and ${\bf \bar{\rate}}$ satisfy Assumption \ref{mod:cont}, and the jump rates satisfy \eqref{eq:standing} we have for every $k \in \Nb$, 
\begin{align*} 
\lim_{n\to\infty}\rate\gvpara{G_n}{\varphi_n(v_k)}\stpara{j_k}\left(\tau_k,X^{G,\x}_{\varphi_n^{-1}(V_n)}\right)\indi{A_n} &= \lim_{n\to\infty} \rate^{(V,E,\root,\ems^n_{\varphi_n(E)},\vms^n_{\varphi_n(V)}),v_k}_{j_k}\left(\tau_k,X^{G,\x}\right)\indi{A_n} \\ 
&= \rate^{G,v_k}_{j_k}\left(\tau_k,X^{G,\x}\right) 
\end{align*} 
a.s. on the event $\{k\leq |\ov{\Emc}|\}$. Because $|\ov{\Emc}|$ is a.s. finite, it follows that $\lim_{n\to\infty} R_n = 0$ a.s.. Moreover, the fact that $R > 0$ a.s. implies the existence of an a.s. finite random variable $\ov{N}$ such that, 
\begin{equation} 
\label{LWCpf:Rov2} 
R_n < \frac{R}{2} \te{ on the event }\{n \geq \ov{N}\}\cap A_n. 
\end{equation} 

We now argue that $X^{G,\x}[T] = X^{G_n,\x^n}[T]$ on the event $\{n \geq \ov{N}\}\cap A_n$ by making use of the fact that for each $n \in \Nb$ and $v \in V$, $X^{G,\x}_v$ and $X^{G_n,\x^n}_{\varphi_n(v)}$ are driven by the same Poisson processes on the event $A_n$. Fix $n \in \Nb$ and note that by the SDE \eqref{mod:infpart}, $X^{G,\x}$ and $X^{G_n,\x^n}$ are both a.s. continuous on the random set $\{t \in[0,\infty)\setminus \{\tau_k\}_{k\in \Nb}\}\cap A_n$. Furthermore, at time $\tau_k$, the processes $X^{G,\x}$ and $X^{G_n,\x^n}$ may either remain constant or experience a jump of size $j_k$ at the respective vertices $v_k$ and $\varphi_n(v_k)$. It follows from the SDE \eqref{mod:infpart} that the processes will either simultaneously jump or fail to jump if and only if 
\[\te{sgn}\left(r_k - \rate\gvpara{G}{v_k}\stpara{j_k}(\tau_k,X^{G,\x})\right) = \te{sgn}\left(r_k - \rate\gvpara{G_n}{\varphi_n(v_k)}\stpara{j_k}(\tau_k,X^{G_n,\x^n})\right),\] 
where $\te{sgn}: \Rb \to \Rb$ is the c\`adl\`ag function given by $\te{sgn}(a) = \indic{a \geq 0} - \indic{a < 0}$. Suppose that $X^{G,\x}[\tau_k) = X^{G_n,\x^n}_{\varphi_n(V)}[\tau_k)$. Then, using the predictability of the jump rates (by the standing assumption and Definition \ref{def-regular}) in the first equality below, we have 
\begin{align*} 
\te{sgn}&\left(r_k - \rate\gvpara{G_n}{\varphi_n(v_k)}\stpara{j_k}(\tau_k,X^{G_n,\x^n})\right) \\ 
& \ind = \te{sgn}\left(r_k - \rate\gvpara{G_n}{\varphi_n(v_k)}\stpara{j_k}(\tau_k,X^{G,\x}_{\varphi^{-1}_n(V_n)})\right)\\ 
& \ind= \te{sgn}\left(r_k - \rate\gvpara{G}{v_k}\stpara{j_k}(\tau_k,X^{G,\x}) + \rate\gvpara{G}{v_k}\stpara{j_k}(\tau_k,X^{G,\x})-\rate\gvpara{G_n}{\varphi_n(v_k)}\stpara{j_k}(\tau_k,X^{G,\x}_{\varphi_n^{-1}(V_n)})\right). 
\end{align*} 
However, the last line of the above display is a.s. equal to $\te{sgn}\left(r_k - \rate\gvpara{G}{v_k}\stpara{j_k}(\tau_k,X^{G,\x})\right)$ on the event $\{n \geq\ov{N}\}$ by \eqref{LWCpf:Rov2}. Thus, $X^{G,\x}[\tau_k] = X^{G_n,\x^n}_{\varphi_n(V)}[\tau_k]$ a.s. on the event $\{n \geq \ov{N}\}\cap A_n$. It follows that $X^{G,\x}[\tau_{k+1}) = X^{G_n,\x^n}_{\varphi_n(V)}[t_{k+1})$ a.s. if $|\ov{\Emc}| > k$ and $X^{G,\x}_V[T] = X^{G_n,\x^n}_{\varphi_n(V)}[T]$ a.s. if $|\ov{\Emc}| = k$. Applying induction, we see that $X^{G,\x}[T] = X^{G_n,\x^n}_{\varphi_n(V)}[T]$ a.s. on the event $\{n \geq \ov{N}\}\cap A_n$, and so \eqref{LWCpf:fGcong} follows.

To see why this implies Lemma \ref{LWCpf:nograph} in the general case of possibly infinite graphs, note that \eqref{LWCpf:fGcong} implies that for any $m \in\Nb$ and any sequence $\{\ic{G_n}\}_{n \in \Nb}, \ic{G}$ we may replace $(G_n,\x^n),n \in \Nb$ by $\trnc{m}(G_n,\x^n),n \in \Nb$, $(G,\x)$ by $\trnc{m}(G,\x)$, $\ov{N}_{m}$ by $\ov{N}$, and $M$ by $m$ in \eqref{LWCpf:fGcong} to get \eqref{LWCpf:Gcong}, $(\pzspace, \{\ic{(G_n, \x^n)}\}_{n\in\Nb},\ic{(G, \x)})$ satisfies the finite convergence property. This concludes the proof of the lemma. 
\end{proof}

\subsection{Proof of Theorem \ref{LWC:LWC}} 
\label{subs-finalpfLWC}
We now show how Theorem \ref{LWC:LWC} follows from Proposition \ref{LWCpf:LWC}. The proof uses a simple argument involving the Skorokhod representation theorem and Proposition \ref{WP:locality2}.

\begin{proof}[Proof of Theorem \ref{LWC:LWC}] 
Set $\ic{(G_\infty,\x^\infty)} := \ic{(G,\x)}$ and $\Nb_\infty := \Nb \cup \{\infty\}$. Since Assumption \ref{mod:assu} holds with associated constants ${\bf C}$, and each $\ic{G_n}, n \in \Nb_\infty,$ is a.s. finitely dissociable with respect to ${\bf C}$, the conditions of Proposition \ref{WP:locality2} are satisfied a.s.. Hence, the collection $\{\ic{G_n}\}_{n \in \Nb_{\infty}}$ a.s. consistently spatially localizes the SDE \eqref{mod:infpart}. Assumption \ref{mod:assu}, Proposition \ref{WP:fin}, and Proposition \ref{WP:WP2} then imply that the SDE \eqref{mod:infpart} is strongly well-posed for all initial data in $\{\ic{(G_n,\x^n)}\}_{n \in \Nb_\infty}$. Moreover, since $\ic{(G_n,\x^n)}\Rightarrow\ic{(G,\x)}$ in $\gs\sp{\emksp,\vmksp\times \Xc}$, by the Skorokhod representation theorem there exists a (complete) probability space $(\alt{\Omega}, \alt{\filtm}, \alt{\mathbb{P}})$ that supports random elements $\ic{(\alt{G}_n,\alt{\x}^n)} \deq \ic{(G_n,\x^n)}$, $n \in \Nb_\infty,$ such that $\ic{(\alt{G}_n,\alt{\x}^n)} \to \ic{(\alt{G},\alt{\x})}\defeq \ic{(\alt{G}_{\infty},\alt{\x}^{\infty})}$ $\alt{\mathbb{P}}$-a.s..

By Lemma \ref{MG:repextexist} there exists a consistent rep-con extension of $(\pzspace,\{\ic{(\alt{G}_n,\alt{\x}^n)}\}_{n \in \Nb_{\infty}})$ given by $(\fpspace,\{(\alt{G}_n,\alt{\x}^n,\poiss^{\alt{G}_n}),\alt{M}_n\}_{n \in \Nb},(\alt{G},\alt{\x},\poiss^{\alt{G}}),$ $\{\alt{\varphi}_{n,m}\}_{n \in \Nb, m \leq M_n})$. For each $n \in \Nb_{\infty}$, let $(G_n,X^{\alt{G_n},\alt{\x}^n})$ be the resulting $\poiss^{\alt{G}_n}$-strong solution to \eqref{mod:infpart} for $(\alt{G}_n,\alt{\x}^n)$, where $(\alt{G}_{\infty},\alt{\x}^{\infty}) = (\alt{G},\alt{\x})$ and let $\ic{(\alt{G}_n,X^{\alt{G}_n,\alt{\x}^n})}$ denote its isomorphism class. Such strong solutions are well defined by Lemma \ref{mod:stunique} (which establishes the existence of a pathwise unique strong solution for the initial data $\ic{(\alt{G}_n,\alt{\x}^n)}$, $n \in \Nb_{\infty}$). Because $\{\alt{G}_n\}_{n \in \Nb_{\infty}}$ spatially localizes the SDE \eqref{mod:infpart}, by Assumption \ref{mod:assu} and the fact that $\ic{(G,\x)}$ (and therefore $\ic{(\alt{G},\alt{\x})}$) satisfies Assumption \ref{mod:cont}, Lemma \ref{LWCpf:nograph} implies that the tuple $(\pzspace,\{\ic{(G_n, \x^n)}\}_{n\in\Nb},\ic{(G, \x)})$ satisfies Assumption \ref{LWCpf:finconv}. Proposition \ref{LWCpf:LWC} then implies 
\[\lim_{n\to\infty} \ic{(\alt{G}_n,X^{\alt{G}_n,\alt{\x}^n})} = \ic{(\alt{G}_{\infty},X^{\alt{G}_{\infty},\alt{\x}^{\infty}})} \quad \te{a.s..}\] 
By well-posedness of \eqref{mod:infpart}, $\ic{(\alt{G}_n,X^{\alt{G}_n,\alt{\x}^n})} \deq \ic{(G_n,X^{G_n,\x^n})}$ for every $n \in \Nb_{\infty}$. Thus, $\ic{(G_n,X^{G_n,\x^n})}$ $\deq \ic{(\alt{G}_n,X^{\alt{G}_n,\alt{\x}^n})}\Rightarrow \ic{(\alt{G},X^{\alt{G},\alt{\x}})} \deq \ic{(G,X^{G,\x})}$ in $\gs\sp{\emksp,\vmksp\times \cad}$. 
\end{proof}

\section{Proof of Asymptotic Correlation Decay} 
\label{GEM}
This section is devoted to the proof of Theorem \ref{GEM:deccor}. Recall that for an unrooted $\sp{\emksp,\vmksp}$-graph $G$ and a vertex $v \in V_G$, $\cmpn{v}(G)$ is the connected component of $G$ equipped with $v$ as its root. For the remainder of the section, we fix a sequence of finite, (possibly disconnected) unrooted $\sp{\emksp,\vmksp\times \Xc}$-random graphs $(G_n,\x^n)$, $n \in \Nb,$ and a $\gs\sp{\emksp,\vmksp\times \Xc}$-random element $\ic{\cmpn{\root}(G,\x)}$ (henceforth, denoted just $\ic{(G,\x)}$), all defined on a common complete probability space $\pzspace$. We additionally assume that, by extending the probability space if necessary, $\pzspace$ also supports an i.i.d. pair of vertices $(o^1_n, o^2_n)$, with each vertex uniformly distributed on $G_n$, for all $n \in \Nb$. The proof of Theorem \ref{GEM:deccor} is comprised of two steps. The first step is to establish an asymptotic independence property stated in Lemma \ref{EMP:dLWC} below.

\begin{lemma}[Asymptotic independence of initial data] 
\label{EMP:dLWC} 
Suppose $(G_n,\x^n)$ converges locally in probability to $\ic{(G,\x)}$. Then 
\begin{equation} 
\label{eq:dLWC} 
(\ic{\cmpn{o_n^1}(G_n,\x^n)},\ic{\cmpn{o_n^2}(G_n,\x^n)}) \Rightarrow (\ic{(G^{(1)},\x^{(1)})},\ic{(G^{(2)},\x^{(2)})}),
\end{equation} 
where $ \Rightarrow$ represents convergence in distribution in $(\gs\sp{\emksp,\vmksp\times \Xc})^2$ and $\ic{(G^{(i)},\x^{(i)})}, i=1,2,$ are two independent copies of $\ic{(G,\x)}$. 
\end{lemma}
\begin{proof}
Let $f_1,f_2: \gs\sp{\emksp,\vmksp\times\Xc} \to \Rb$ be bounded, continuous functions. Then by \cite[Lemma 2.8]{LacRamWu23},
\begin{align}
\lim_{n\to\infty}\ex{f_1(\ic{\cmpn{o_n^1}(G_n,\x^n)})f_2(\ic{\cmpn{o_n^2}(G_n,\x^n)})} &= \ex{f_1(\ic{(G,\x)})}\ex{f_2(\ic{(G,\x)})}\nonumber\\
&= \ex{f_1(\ic{(G^{(1)},\x^{(1)})})f_2(\ic{(G^{(2)},\x^{(2)})})}.\label{EMP:initsep}
\end{align}
Since $\gs\sp{\emksp,\vmksp\times\Xc}$ is a metric space, the algebra of separable bounded functions on $(\gs\sp{\emksp,\vmksp\times\Xc})^2$ strongly separates points and hence, is convergence determining \cite[Theorem 3.4.5(b)]{EthKur86}. Thus, \eqref{EMP:initsep} implies \eqref{eq:dLWC}. 
\end{proof}

Invoking the Skorokhod representation theorem, the second and main step of the proof assumes joint local convergence of the initial data $\{\ic{\cmpn{o_n^1}(G_n,\x^n)},\ic{\cmpn{o_n^2}(G_n,\x^n)}\}_{n \in \Nb}$ to the i.i.d. pair $(\ic{(G^{(1)},\x^{(1)}},\ic{(G^{(2)},\x^{(2)})})$ and proves convergence \emph{in probability} of the corresponding pairs of strong solutions $\{\ic{\cmpn{o_n^1}(G_n,X^{G_n,\x^n})},\ic{\cmpn{o_n^2}(G_n,X^{G_n,\x^n})}\}_{n \in \Nb}$ to $(\ic{(G^{(1)},X^{G^{(1)},\x^{(1)}})},\ic{(G^{(2)},X^{G^{(2)},\x^{(2)}})})$. The coupling proof proceeds as follows. For each $i = 1, 2$, we first construct a sequence of driving noises $\{(G,\poiss^{n,i})\}_{n \in \Nb}$ such that the isomorphism class of the corresponding $\{\poiss^{n,i}\}$-strong solution $\{\ic{\cmpn{o_n^i}(G_n,X^{n,i})}\}$ for the initial data $\ic{\cmpn{o_n^i}(G_n,\x^n)}$ converges a.s. as $n \rightarrow \infty$ to $\ic{(G^{(i)},X^{G^{(i)},\x^{(i)}})}$. From this, we construct a single sequence of \emph{common} driving noises $\{(G_n,\poiss^{G_n})\}_{n \in \Nb}$ such that for each $i=1,2,$ the corresponding $\{\poiss^{G_n}\}_{n \in \Nb}$-strong solutions $\cmpn{o_n^i}(G_n,X^{G_n,\x^n}),$ $n \in \Nb,$ and $(G^{(i)},X^{G^{(i)},\x^{(i)}})$ satisfy 
\[\ic{\cmpn{o^i_n}(G_n,X^{G_n,\x^n})} \to \ic{(G^{(i)},X^{G^{(i)},\x^{(i)}})}\te{ in probability.}\] 
By the independence of $\{\ic{(G^{(i)},X^{G^{(i)},\x^{(i)}})}\}_{i=1,2}$, this would imply the desired correlation decay result. 

\begin{lemma}[Asymptotic independence of IPS on independently rooted components] 
\label{EMP:inspce} 
\sloppy Given $\{(G_n,\x^n)\}_{n \in \Nb}$, suppose there exists a countable set $\Smc$ such that for $\PP$-a.s. every $\omega \in \Omega$, $V_{G_n(\omega)} \subset \Smc$ for all $n \in \Nb$. Suppose also that Assumption \ref{mod:assu2} holds and that the collection of isomorphism classes $\{\ic{\cmpn{o^i_n}(G_n)},\ic{G^{(i)}}\}_{n \in \Nb,i=1,2}$ a.s. consistently spatially localizes the SDE \eqref{mod:infpart}. In addition, assume that for each $i = 1, 2,$, 
\begin{equation} 
\label{EMP:asconvassu} 
\ic{\cmpn{o_n^i}(G_n,\x^n)} \to \ic{(G^{(i)},\x^{(i)})} 
\end{equation} 
and the tuple $(\{\ic{\cmpn{o^i_n}(G_n,\x^n)}\}_{n\in\Nb},\ic{(G^{(i)},\x^{(i)})})$ satisfies Assumption \ref{LWCpf:finconv}. Then it is possible to define a filtered probability space $\fpspace$ supporting solutions $(G_n,X^{G_n,\x^n}), n \in \Nb,$ and $\ic{(G^{(i)},X^{G^{(i)},\x^{(i)}})}$, $i=1,2,$ to the SDE \eqref{mod:infpart} for the respective initial data $(G_n,\x^n), n \in \Nb,$ and $\ic{(G^{(i)},\x^{(i)})}$, $i=1,2,$ such that as $n\to\infty$, 
\begin{equation} 
\label{EMP:ipsolconv} 
\ic{\cmpn{o_n^i}(G_n,X^{G_n,\x^n})} \to \ic{(G^{(i)},X^{G^{(i)},\x^{(i)}})},\te{ }i=1,2, \te{ in probability.} 
\end{equation} 
\end{lemma} 
\begin{proof} 
Fix any deterministic injection $\wh{\psi}: \Smc \to \Nb$. For each $n\in\Nb$, let $(\alt{G}_n,\alt{\x}^n)$ be the unique $\wh{\gs}\sp{\emksp,\vmksp\times\Xc}$-random element such that $\wh{\psi}|_{V_{G_n}} \in I((G_n,\x^n),(\alt{G}_n,\alt{\x}^n))$. To be precise, if for each $n \in \Nb$, $(G_n,\x^n) = (V_n,E_n,\root_n,\ems^n,\vms^n,\x^n)$, then 
\[(\alt{G}_n,\alt{\x}^n) \defeq (\alt{V}_n,\alt{E}_n,\alt{\root}_n,\alt{\ems}^n,\alt{\vms}^n,\alt{\x}^n) = (\wh{\psi}(V_n),\wh{\psi}(E_n),\wh{\psi}(\root_n),(\ems^n_{\wh{\psi}(e)})_{e \in \alt{E}_n},(\vms^n_{\wh{\psi}(v)},\x^n_{\wh{\psi}(v)})_{v \in \alt{V}_n}).\] 
Noting that all the remaining statements of the lemma depend only on $\ic{\cmpn{o^i_n}(G_n,\x^n)} = \ic{\cmpn{\wh{\psi}(o^i_n)}(\alt{G}_n,\alt{\x}^n)}$, we may assume without loss of generality that $\Smc \subseteq \Nb$ and therefore that $\cmpn{o^i_n}(G_n,\x^n)$ is a $\wh{\gs}\sp{\emksp,\vmksp\times \Xc}$-random element for each $n \in\Nb$ and $i=1,2$. 

Now, by Assumption \ref{mod:assu2} and the spatial localization assumption, Proposition \ref{WP:WP2} implies that the SDE \eqref{mod:infpart} is strongly well-posed for all initial data in the collection of marked graph representatives $\{\cmpn{o^i_n}(G_n,\x^n),(G^{(i)},\x^{(i)})\}_{n\in\Nb,i=1,2}$, where for $i = 1, 2$ $(G^{(i)},\x^{(i)})$ is a random representative of $\ic{(G^{(i)},\x^{(i)})}$ (whose existence is guaranteed by Lemma \ref{msbl:rep}). Also, by assumption, the collection of $\wh{\gs}\sp{\emksp,\vmksp\times \Xc}$-random elements $\{\cmpn{o^i_n}(G_n,\x^n)\}_{n \in \Nb,i=1,2}$ is $\wh{\filtm}$-measurable. For each $i= 1,2,$ Lemma \ref{MG:repextexist} implies the existence of an associated consistent rep-con extension $((\Omega^i,\filtm^i,\filt^i,\PP^i),\{\cmpn{o^i_n}(G_n,\x^n,\poiss^{n,i}),M_n^{(i)}\}_{n \in \Nb},(G^{(i)},\x^{(i)},\poiss^{G^{(i)}}),\{\varphi^{(i)}_{n,m}\}_{n \in \Nb, m \leq M_n})$ where $(\Omega^i,\filtm^i,\filt^i,\PP^i)$ supports a collection of i.i.d. $\filt^i$-Poisson processes $\ov{\poiss}^i \defeq \{\ov{\poiss}^i_k\}_{k \in \Nb}$ which together with the driving maps $\ov{\psi}^{n,i}$, $n \in \Nb,$ and $\ov{\psi}^{(i)}$ generate the respective $\filt^i$-driving noises $\poiss^{n,i}$, $n \in \Nb,$ and $\poiss^{G^{(i)}}$. We may also identify $\fpspace \defeq (\Omega^1,\filtm^1,\filt^1,\PP^1)=(\Omega^2,\filtm^2,\filt^2,\PP^2)$ and assume that $\ov{\poiss}^1$ is independent of $\ov{\poiss}^2$ so that both collections consist of $\filt$-Poisson processes. On this space, define for each $k \in \Nb$ and $i = 1,2,$ $\poiss_{2k-2+i} \defeq \ov{\poiss}^i_k$. Then for each $n\in\Nb$ and $i = 1,2,$ $\poiss^{n,i}$ and $\poiss^{G^{(i)}}$ are generated from $\{\poiss_k\}_{k \in \Nb}$ by the driving maps $\psi^{n,i}(\cdot) \defeq 2\ov{\psi}^{n,i}(\cdot)-2+i$ and $\psi^{(i)}(\cdot) \defeq 2\ov{\psi}^{(i)}(\cdot) - 2 + i$, respectively.

For each $n$, let $L_n \defeq d_{G_n}(o_n^1,o_n^2)$ and let $\ov{M}_n$ be the maximal $\filtm_0$-measurable random variable such that $\ov{M}_n < \frac{L_n}{2}$ and $\ov{M}_n \leq \min_{i=1,2}M_n^{(i)}$ a.s.. Define the mapping $\psi_n:V_n \to \Nb$ by 
\[\psi^n(v) \defeq 
\begin{cases} 
\psi^{n,1}(v) &\te{ on the event }\{d_{G_n}(v,o^1_n)\leq \ov{M}_n\},\\ 
\psi^{n,2}(v) &\te{ otherwise.} 
\end{cases}\] 
Since for $i = 1, 2$, the driving maps $\psi^{n,i}$ have disjoint images and $\ov{M}_n < \frac{1}{2}d_{G_n}(o_n^1,o_n^2)$, it follows that $\psi^n$ is $\filtm_0$-measurable and injective and is therefore also a driving map. For each $n \in \Nb$, let $\poiss^{G_n}$ be the $\filt$-driving noise generated by the driving map $\psi^n$. Then $\poiss^{G_n}$ is compatible with $\cmpn{o^i_n}(G_n),$ $i = 1, 2$. Because $\ov{M}_n \leq M_n^{(i)}$, $i = 1,2$, it follows that for any $i=1,2,$ $n\in\Nb,$ and $v \in \trnc{m}(\cmpn{o^i_n}(G_n))$, 
\begin{equation} 
\label{EMP:poisseq} 
\poiss^{G_n}_v = \poiss^{n,i}_v = \poiss^{G^{(i)}}_{(\varphi_{n,m}^{(i)})^{-1}(v)} \te{ on the event } \{m \leq \ov{M}_n\}. 
\end{equation}

Let $\{n_k\}_{k \in \Nb}$ be any deterministic sequence such that $L_{n_k}$ (and therefore $\ov{M}_{n_k}$) diverges to infinity a.s. as $k \rightarrow \infty$. Such a sequence exists because $L_n \to \infty$ in probability. Fix $i \in \{1, 2\}$ and consider the tuple $\mathbf{T}_i \defeq (\fpspace, \{\cmpn{o^i_{n_k}}(G_{n_k},\x^{n_k},\poiss^{G_{n_k}}),\ov{M}_{n_k}\}_{k\in\Nb},$ $(G^{(i)},\x^{(i)},\poiss^{G^{(i)}}),\{\varphi^{(i)}_{n_k,m}\}_{k \in \Nb, m \leq \ov{M}_{n_k}})$. Using \eqref{EMP:poisseq}, the definition of $\ov{M}_n$ and directly checking the three properties of Definition \ref{MG:extdef}, it is easy to see that this tuple is a consistent rep-con extension of $\mathbf{R}_i := (\pzspace,\{\ic{\cmpn{o^i_{n_k}}(G_{n_k},\x^{n_k})}\},\ic{(G^{(i)},\x^{(i)})})$. The verification is provided below for completeness.

\begin{description}
\item[Property 1:] By construction, $(\{\cmpn{o^i_n}(G_n,\x^n), M^{(i)}_n\}_{n \in \Nb}, (G^{(i)},\x^{(i)})$,$\{\varphi^{(i)}_{n,m}\}_{n \in \Nb, m \leq M_n})$ is a rep-con sequence of $(\{\ic{\cmpn{o^i_n}(G_n,\x^n)}\}_{n \in \Nb},\ic{(G^{(i)},\x^{(i)})})$. Since $\ov{M}_n \leq M^{(i)}_n$ for every $n \in\Nb$, the tuple $\mathbf{\ov{R}}_i \defeq (\{\cmpn{o^i_{n_k}}(G_{n_k},\x^{n_k}), \ov{M}_{n_k}\}, (G^{(i)},\x^{(i)}),\{\varphi^{(i)}_{n_k,m}\}_{k \in \Nb, m \leq \ov{M}_{n_k}})$ satisfies Properties 1,2 and 4-6 of Definition \ref{MG:rcs}. Because $\ov{M}_{n_k} \to \infty$ a.s. and $\ov{M}_{n_k} \leq M^{(i)}_{n_k}$ for each $k \in \Nb$, it follows that Property 3 of Definition \ref{MG:rcs} holds. Thus, $\mathbf{\ov{R}}_i$ is a rep-con sequence of $(\{\ic{\cmpn{o^i_{n_k}}(G_n,\x^n)}\}_{n \in \Nb},\ic{(G^{(i)},\x^{(i)})})$. 
\item[Property 2:] This holds because $\fpspace$ is part of the consistent rep-con extension of $\mathbf{R}_i$ and $\poiss^{G_n}$ is an $\filt$-driving noise. 
\item[Property 3:] This holds by \eqref{EMP:poisseq}.
\end{description}
Given that $\mathbf{T}_i$ is a consistent rep-con extension of $\mathbf{R}_i$, by Proposition \ref{LWCpf:LWC}, 
\[\ic{\cmpn{o^i_{n_k}}(G_{n_k},X^{G_{n_k},\x^{n_k}})} \to \ic{(G^{(i)},\x^{(i)})} \te{ a.s..}\] 
Lastly, note that for any deterministic subsequence $\{n_k\}_{k\in\Nb}\subseteq \Nb$, there exists a further deterministic subsequence $\{n_{k_\ell}\}_{\ell\in\Nb}$ such that $\lim_{\ell\to\infty}L_{n_{k_\ell}} = \infty$ a.s.. By the above result, this implies that 
\[\lim_{\ell\to\infty}\ic{\cmpn{o^i_{n_{k_\ell}}}(G_{n_{k_\ell}},X^{G_{n_{k_\ell}},\x^{n_{k_\ell}}})} = \ic{(G^{(i)},\x^{(i)})} \te{ a.s.,}\] 
which immediately implies \eqref{EMP:ipsolconv} and completes the proof. 
\end{proof}

\begin{remark}[Asymptotic independence for more general IPS] 
\label{EMP:modeltype} 
Note that the proof of Lemma \ref{EMP:inspce} does not make direct use of the specific form of the SDE \eqref{mod:infpart} but only relies on spatial localization of the SDE, Assumptions \ref{mod:assu2} and \ref{LWCpf:finconv}, and Proposition \ref{LWCpf:LWC}, all of which could potentially be verified for solutions $X^{G,\x}$ of more general Poisson-driven SDEs. 
\end{remark}

Applying Lemmas \ref{EMP:dLWC} and \ref{EMP:inspce}, we now prove Theorem \ref{GEM:deccor}. 

\begin{proof}[Proof of Theorem \ref{GEM:deccor}:] 
Lemma \ref{EMP:dLWC} implies that 
\[(\ic{\cmpn{o_n^1}(G_n,\x^n)},\ic{\cmpn{o_n^2}(G_n,\x^n)}) \Rightarrow (\ic{(G^{(1)},\x^{(1)})},\ic{(G^{(2)},\x^{(2)})}),\] 
\sloppy where $\ic{(G^{(i)},\x^{(i)})},i=1,2,$ are i.i.d. copies of $\ic{(G,\x)}$. By the Skorokhod representation theorem and Lemma \ref{msbl:rep}, there exist $\{\ic{(\alt{G}^{(i)},\alt{\x}^{(i)})}\}_{i=1,2}\deq \{\ic{(G^{(i)},\x^{(i)})}\}_{i=1,2}$, and finite, $\wh{\gs}\sp{\emksp,\vmksp\times\Xc}$-random elements $\{(\alt{G}_n,\alt{o}^i_n,\alt{\x}^n)\}_{n\in\Nb,i=1,2}\deq \{(G_n,o^i_n,\x^n)\}_{n\in\Nb,i=1,2}$ such that 
\[\ic{\cmpn{\alt{o}^i_n}(\alt{G}_n,\alt{\x}^n)} \to \ic{(\alt{G}^{(i)},\alt{\x}^{(i)})}\te{, }i=1,2,\te{ a.s..}\] 
\sloppy By assumption, $\alt{G}_n \deq G_n$ and $\alt{G}^{(i)} \deq G$ are a.s. finitely dissociable for $i=1,2,$ and $n \in \Nb$. So by Assumption \ref{mod:assu} and Proposition \ref{WP:locality2}, $\{(\alt{G}_{n}, \alt{o}^i_n,\alt{\x}^n),(\alt{G}^{(i)},\alt{\x}^{(i)})\}_{n\in\Nb,i=1,2}$ consistently spatially localizes the SDE \eqref{mod:infpart}. Assumptions \ref{mod:assu} and \ref{mod:cont} imply by Lemma \ref{LWCpf:nograph} that the tuple $(\pzspace,\{\ic{\cmpn{o^i_n}(\alt{G}_n,\alt{\x}^n)}\}_{n\in\Nb},\ic{(\alt{G}^{(i)},\alt{\x}^{(i)})})$ satisfies Assumption \ref{LWCpf:finconv} for each $i$. Thus, by Lemma \ref{EMP:inspce}, it is possible to construct a collection of driving noises $\{\poiss^{\alt{G}_n},\poiss^{\alt{G}^{(i)}}\}_{n \in \Nb,i=1,2}$ such that for each $n \in \Nb$ and $i=1,2,$ the respective $\poiss^{\alt{G}_n}$ and $\poiss^{\alt{G}^{(i)}}$-solutions $X^{\alt{G}_n,\alt{\x}_n}$ and $X^{\alt{G}^{(i)},\alt{\x}^{(i)}}$ satisfy 
\[\ic{\cmpn{\alt{o}^n_i}(\alt{G}_n,X^{\alt{G}_n,\alt{\x}^n})} \to \ic{(\alt{G}^{(i)},X^{\alt{G}^{(i)},\alt{\x}^{(i)}})}\te{ in probability as } n\to \infty.\] 
By the bounded convergence theorem, for any bounded, continuous function $f: (\gs\sp{\emksp,\vmksp\times\cad})^2\to \Rb$, 
\begin{equation} 
\label{EMP:fconv} 
\lim_{n\to\infty}\ex{f\left(\left(\ic{\cmpn{\alt{o}^n_i}(\alt{G}_n,X^{\alt{G}_n,\alt{\x}^n})}\right)_{i=1,2}\right) } =\ex{f\left(\left(\ic{(\alt{G}^{(i)},X^{\alt{G}^{(i)},\alt{\x}^{(i)}})}\right)_{i=1,2}\right) }. 
\end{equation} 
By well-posedness of the SDE \eqref{mod:infpart} on each of the graphs $(\alt{G}_n,\alt{\x}^n), n \in \Nb,$ and $(\alt{G},\alt{\x})$, which holds by Theorem \ref{WP:WP}, it follows that for every $n \in \Nb$, $\left(\ic{\cmpn{\alt{o}^n_i}(\alt{G}_n,X^{\alt{G}_n,\alt{\x}^n})}\right)_{i=1,2} \deq \left(\ic{\cmpn{o_n^i}(G_n,X^{G_n,\x^n})}\right)_{i=1,2}$ and $(\alt{G}^{(1)},X^{G^{(1)},\x^{(1)}}) \deq (\alt{G}^{(2)},X^{G^{(2)},\x^{(2)}}) \deq (G,X^{G,\x})$. Together with \eqref{EMP:fconv} and the independence of $\ic{(\alt{G}^{(1)},X^{\alt{G}^{(1)},\alt{\x}^{(1)}})}$ and $\ic{(\alt{G}^{(2)},X^{\alt{G}^{(2)},\alt{\x}^{(2)}})}$, for any bounded, continuous functions $f_1,f_2: \gs\sp{\emksp,\vmksp\times\cad}\to \Rb$, it follows that 
\[
\begin{array}{l} 
\lim_{n\to\infty}\ex{f_1(\ic{\cmpn{o_n^1}(G_n,X^{G_n,\x^n})})f_2(\ic{\cmpn{o_n^2}(G_n,X^{G_n,\x^n})})} \\ 
\qquad \qquad =\lim_{n\to\infty}\ex{f_1(\ic{\cmpn{\alt{o}^n_1}(\alt{G}_n,X^{\alt{G}_n,\alt{\x}^n})})f_2(\ic{\cmpn{\alt{o}^n_2}(\alt{G}_n,X^{\alt{G}_n,\alt{\x}^n})})}\\ 
\qquad \qquad =\ex{f_1(\ic{(\alt{G}^{(1)},X^{\alt{G}^{(1)},\alt{\x}^{(1)}})})f_2(\ic{(\alt{G}^{(2)},X^{\alt{G}^{(2)},\alt{\x}^{(2)}})})}\\ 
\qquad\qquad =\ex{f_1(\ic{(\alt{G}^{(1)},X^{\alt{G}^{(1)},\alt{\x}^{(1)}})})}\ex{f_2(\ic{(\alt{G}^{(2)},X^{\alt{G}^{(2)},\alt{\x}^{(2)}})})}\\ 
\qquad\qquad =\ex{f_1(\ic{(G,X^{G,\x})})}\ex{f_2(\ic{(G,X^{G,\x})})}. 
\end{array} 
\] 
This implies the desired asymptotic correlation decay in \eqref{GEM:corlim}. 
\end{proof}

\begin{remark}[Correlation decay for more general IPS] 
\label{EMP:unbdded} 
From the proof of Theorem \ref{GEM:deccor}, it is not hard to see that, in fact, the conclusion of Theorem \ref{GEM:deccor} holds under Assumptions \ref{mod:assu2} and \ref{LWCpf:finconv} and the spatial localization property in place of Assumptions \ref{mod:assu} and \ref{mod:cont} and the finite dissociability property respectively. 
\end{remark}

\appendix
\section{Counterexample: When the SDE is not Well-Posed} 
\label{illpf}
There are two ways in which the SDE \eqref{mod:infpart} can fail to be well-posed. The IPS may either admit no solution, or it may admit multiple solutions. It is easy to construct a time-inhomogeneous IPS that does not admit any solution by simply specifying unbounded jump rates that increase sufficiently fast with time such that any stochastic process $X \defeq X^{G,\x}$ for which \eqref{mod:infpart} holds almost surely, must with positive probability explode (i.e., have infinitely many discontinuities on some compact interval in $\Rb_+$). However, in this section, we show that well-posedness may fail even if we restrict ourselves to time-homogeneous Markov processes with uniformly bounded jump rates (which is even stronger than Assumption \ref{mod:assu}). More precisely, we identify a graph $G$ and an associated family of jump rates ${\bf r^G}$ for which the SDE \eqref{mod:infpart} has multiple strong solutions with different laws.

\begin{proposition}[An IPS that is not well-posed] 
\label{WP:ill} 
Let $\tree$ be a rooted tree such that for every $k \in \Nb$, all vertices $v$ in the $k$th generation of $\tree$ (i.e., those that are distance $k$ from the root) have $4^k$ children. Consider the Markovian IPS with state space $\Xc = \{0,1\}$, allowable jump set $\jmps = \{1\}$, and for $v \in \tree$, $t \in \Rb_+$, and $x \in \cad^\tree$, let 
\begin{equation} 
\label{rate-prop} 
\rate\gvpara{\tree}{v}\stpara{1} (t,x) = 
\begin{cases} 
1 & \te{ if } x_v (t-) = 0 \te{ and } \sum_{u \in \neigh{v}} x_{u} (t-) > 0, \\ 
0 & \te{ otherwise.} 
\end{cases} 
\end{equation} 
Then the corresponding SDE \eqref{mod:infpart} with $G = \tree$ and initial state $\x_v = 0, v \in \tree$, has multiple strong solutions with different laws. 
\end{proposition}

Note that for the SDE \eqref{mod:infpart} with $G = \tree$, the driving noises $\poiss^\tree$ are simply the collection of i.i.d. Poisson processes $(\poiss^{\tree}_v)_{v \in \tree}$ on $[0,\infty)^2 \times \{1\}$ with intensity measure ${\rm Leb}^2 \otimes \delta_1$. Also observe that the family of jump rates ${\bf r^\tree}$ in Proposition \ref{WP:ill} satisfy Assumption \ref{mod:assu} with constants $C_{k,T} = 1$ for all $k \in \Nb$ and $T \in \Rb_+$.

To prove Proposition \ref{WP:ill}, we first establish a certain property of the tree $\tree$. We start with a definition that generalizes the notion of causal chains given in Definition \ref{pfwp:causal}, for the above example. 

\begin{definition}[Infinite causal chains] 
\label{def:infcchain} 
Fix $T \in [0,\infty)$, and let $\{\poiss_{v}\}_{v \in \tree}$ be a collection of i.i.d. rate $1$ Poisson processes on $[0,T]$. Then an infinite path $\Gamma = \{u_0, u_1, \ldots \}$ with $u_i \in \tree$, $i \in \Nb_0,$ is said to be a $(\tree, {\bf \poiss})$-infinite causal chain ending at $u_0$ during $[0,T]$ if there exists a decreasing sequence $T > s_0 > s_1 > \cdots > 0$ such that $s_i \in \{t \in [0,T]: \poiss_{u_i}(\{t\})=1\}$ for all $i \in \Nb_0.$ Moreover, we say a $(\tree, {\bf \poiss})$-infinite causal chain exists if there exists some $T \in \Rb_+$, $u_0 \in \tree$, and a $(\tree, {\bf \poiss})$-infinite causal chain ending at $u_0$ during $[0,T]$. 
\end{definition}

\begin{lemma}[Existence of infinite causal chains] 
\label{illpf:tchn} 
There exists a $(\tree,\poiss)$-infinite causal chain. 
\end{lemma} 
\begin{proof} 
Fix $T \in [0,\infty)$. Define the infinite path $\Gamma' \defeq \{u'_0, u'_1,\dots\}$ using the following recursive construction. Set $u'_0 = \root$ and for $n \in \Nb_0$, first check if there exists $v \in \chil{u'_n}{\tree}$ such that $\poiss_v(2^{-n-2}T,2^{-n-1}T) > 0$. If this condition is satisfied, then set $u'_{n+1} = v$ (if multiple children satisfy this then we choose one arbitrarily from amongst them). If not, then we set $u'_{n+1}$ to be any child of $u'_n$.

Next, set $\mathcal{I} := \{ n \in \mathbb{N}_0: \poiss_{u'_n} (2^{-n-1}T,2^{-n}T) > 0\}$, and define $U := \{u'_n, n \in \mathcal{I}\} \subseteq \Gamma'$. We now claim that $U$ contains an infinite path a.s.. To see why, we prove the equivalent claim that $U$ a.s. contains all but finitely many vertices in $\Gamma'$. Let $\alpha_n \defeq \PP(u'_{n+1} \notin U)$. Note that for each $n \in \Nb_0$ and $v \in \chil{u'_n}{\tree}$, 
\[\PP\left(\poiss^{\tree,T}_v(2^{-n-2}T,2^{-n-1}T) = 0\right) = e^{-2^{-n-2}T}.\] 
The independence of the Poisson processes $(\poiss_v)_{v \in \tree}$ then implies that 
\begin{align*} 
\alpha_n &= \prod_{v \in \chil{u'_n}{\tree}}\PP\left(\poiss_{v}(2^{-n-2}T,2^{-n-1}T) = 0\right) \leq \left(e^{-2^{-n-2}T}\right)^{4^n}= e^{-2^{n-2}T}. 
\end{align*} 
Because $\alpha_n$ decreases super-exponentially fast, $\sum_{n=0}^\infty \alpha_n < \infty$. By the Borel-Cantelli lemma, it follows that with probability 1, there exists $m = m(\omega)$ such that $\{u'_n\}_{n \in \Nb_0, n \geq m} \subseteq U$, thus proving the claim. In turn, setting $u_k := u'_{m+k}$, $k = \Nb_0$, the claim implies that the infinite path $\Gamma := (u_0, u_1, \ldots) \subset U$ and thus, a.s. there exists a decreasing sequence of times $\{t_k \in (2^{-(m+k+1)}T,2^{-(m+k)}T)\}, k \in\Nb_0,$ such that $\poiss_{u_{k}}(\{t_k\}) = \poiss_{u'_{m+k}}(\{t_k\}) = 1$. The lemma follows on setting $s_k = t_{k}$.
\end{proof}

\begin{remark}[Infinite causal chains] 
\label{rem-tchn} 
The notion of an infinite causal chain can be extended in a natural way to the general class of IPS considered in this article with rates that satisfy Assumption \ref{mod:assu} with a family of constants ${\bf C}$, by simply replacing $\tree$ by a general graph $G$ in Definition \ref{def:infcchain} and, for each $T \in \Rb_+$, the point process family $\poiss$ by the family $\poiss^{G,T}$ defined via \eqref{findis:poiss}. Note that this is consistent with the choice of $\poiss$ in Definition \ref{def:infcchain} because the IPS from Proposition \ref{WP:ill} is such that $\jmps = \{1\}$ and Assumption \ref{mod:assu} holds with $C_{k,T} := 1$ for all $k \in \Nb, T \in \Rb_+$. From the proof of Proposition \ref{WP:locality2}, it is easy to see that for any IPS, the non-existence of infinite causal chains implies spatial localization. By Proposition \ref{WP:WP2}, spatial localization and Assumption \ref{mod:assu2} imply strong well-posedness of the SDE \eqref{mod:infpart}. Therefore, if Assumption \ref{mod:assu2} holds, then to construct a graph $G$ and jump rates $\mathbf{r}^G$ for which \eqref{mod:infpart} is ill-posed, it is necessary to ensure that an infinite causal chain exists. 
\end{remark}

For the particular IPS from Proposition \ref{WP:ill}, we now show that the existence of an infinite causal chain is also sufficient to conclude that the IPS is not well-posed.

\begin{proof}[Proof of Proposition \ref{WP:ill}:] 
Although there are infinitely many strong solutions to the SDE \eqref{mod:infpart} for the initial data $(\tree,\x)$ with $\x \equiv 0$ and family of jump rates $\{\rate^{\tree,v}_1\}_{v \in \tree}$ as given in \eqref{rate-prop}, we explicitly construct two of them. First, note that the process $X^{\tree,\x}_v(t) = 0$ for all $v \in V$ and $t \in \Rb_+$ is clearly one strong solution.

Next, note that the driving noises $\poiss^{\tree}_v, v \in \tree,$ of the SDE are i.i.d. Poisson process on $[0,\infty)^2 \otimes \{1\}$, and $\poiss_v(A) := \poiss^{\tree}_v(A \times (0,1]\times \{1\})$, $v \in \tree,$ are i.i.d. unit rate Poisson processes. Then, for $(v,t) \in \tree \times (0,T)$, we say $(v,t)$ has an infinite causal chain if there exists an infinite $(\tree,\poiss)$-causal chain ending at $v$ on the interval $[0,t]$, in the sense of Remark \ref{rem-tchn}. Then for $v \in V$ and $t \in [0,T]$, define 
\begin{equation} 
\label{illpf:sol} 
\alt{X}^{\tree,\x}_v(t)\defeq 
\begin{cases} 
1 &\te{ if } (v,t)\te{ has an infinite causal chain,}\\ 
0 &\te{ otherwise.} 
\end{cases} 
\end{equation}

We claim $\alt{X}^{\tree,\x}$ is also a solution to the same SDE. To see why the claim is true, define 
\begin{equation} 
\label{illpf:whinfpart} 
\wh{X}_v(t) \defeq \int_{(0,t]\times \Rb_+\times \jmps} \indic{r \leq \rate_1^{\tree,v}(s,(\wh{X}_v,\alt{X}^{\tree,\x}_{\tree \setminus \{v\}}))}\,\poiss^{\tree}_v(ds,dr,dj),\quad t \in \Rb_+. 
\end{equation} 
If $(v,t) \in \tree \times \Rb_+$ has no infinite causal chain, then by definition, for any $s \in [0, t]$, either $\poiss_v([s,t]) = \poiss^{\tree}_v([s,t]\times (0,1]\times \{1\}) = 0$, or there does not exist any $u \in \gneigh{v}{\tree}$ for which $(u,s)$ has an infinite causal chain. Since, from the form given in \eqref{rate-prop}, $\rate_1^{\tree,v}(s,(\wh{X}_v,\alt{X}^{\tree,\x}_{\tree \setminus \{v\}}))$ depends on $\alt{X}^{\tree,\x}_{\tree \setminus \{v\}}$ only via $\alt{X}^{\tree,\x}_{\gneigh{v}{\tree}}(s-)$, it immediately follows that for every event $\sigma$ in $\poiss_v$ that lies in the interval $(0,t]$, $\rate^{\tree,v}_1(\sigma,\wh{X}_v,\alt{X}^{\tree,\x}_{\tree \setminus \{v\}}) = 0$, and hence, $\wh{X}_v(t) = \alt{X}^{\tree,\x}_v(t) = 0$. 

If $(v,t)$ has an infinite causal chain then $\alt{X}_v(t) = 1$ by \eqref{illpf:sol} and Definition \ref{def:infcchain}, there exists an infinite path $\Gamma = \{v=u_0,u_1,\dots\}$ and a sequence $t \geq s_0 > s_1 > \dots > 0$ for which $\poiss_{u_i}(\{s_i\}) = 1$ for all $i \in \Nb_0$. This naturally implies that for any $s_0 > s \geq s_1$, $(u_1,s)$ also has an infinite causal chain, so $\alt{X}^{\tree,\x}_{u_1}(s) = 1$ by \eqref{illpf:sol}. On the other hand, invoking \eqref{rate-prop} and the fact that $u_1 \sim v$, it follows that for all $s > s_1$ such that $\wh{X}_v(s) = 0$, we have 
\[\rate^{\tree,v}_1\left(s,(\wh{X}_v,\alt{X}^{\tree,\x}_{\tree\setminus\{v\}})\right) = 1.\] 
Also, note that by \eqref{rate-prop} and \eqref{illpf:whinfpart}, for any $s > 0$, 
\[\wh{X}_v(s-) = 1 \quad \Rightarrow \quad \rate^{\tree,v}_1\left(s,(\wh{X}_v,\alt{X}^{\tree,\x}_{\tree\setminus\{v\}})\right) = 0 \quad \Rightarrow \quad \wh{X}_v(s) = 1,\]
which shows that $1$ is an absorbing state of $\wh{X}_v$. Together, the last two displays imply that $\wh{X}_v(s_0) = 1$. Indeed, because either $\wh{X}_v(s_0-) = 1$, in which case $\wh{X}_v(s_0) = 1$ by the last display, or $\wh{X}_v(s_0-)=0$ in which case by the form of the SDE \eqref{illpf:whinfpart}, $\wh{X}_v(s_0) = \wh{X}_v(s_0-)+\poiss^{\tree}_v(\{s_0\}\times [0,1]\times \{1\}) = \poiss_v(\{s_0\}) = 1$, with the last equality holding by the definition of the infinite causal chain and the fact that $u_0 = v$. Because $s_0 \leq t$ and $1$ is an absorbing state of $\wh{X}_v$, it follows that $\wh{X}_v(t) = 1 = \alt{X}_v(t)$.

Thus, we have shown that $\wh{X}_v = \alt{X}^{\tree,\x}_v$. Because $v$ and $t$ were chosen arbitrarily, it follows that $\alt{X}^{\tree,\x}$ satisfies \eqref{mod:infpart}. Finally, because $(\tree,\poiss)$-causal chains are $\filt^{\tree,\x,\poiss^\tree}$-adapted, $\alt{X}^{\tree,\x}$ must also be $\filt^{\tree,\x,\poiss^\tree}$-adapted by \eqref{illpf:sol}. Thus, $\alt{X}^{\tree,\x}$ is also a strong solution to \eqref{mod:infpart}, and by Lemma \ref{illpf:tchn}, $X^{\tree,\x} \neq \alt{X}^{\tree,\x}$ a.s.. 
\end{proof} 

\section{Measurable Representatives of Graph Isomorphism Classes} 
\label{msbl}
In this section, we establish the measurability of marked graph representatives of random isomorphism classes, culminating in the proofs of Lemmas \ref{MG:iso} and \ref{MG:repextexist}. Along the way, we introduce a Polish space of canonical representative graphs that is compatible with local convergence, which may be of independent interest. We start with preliminaries in Appendix \ref{subs-Cprelim} and then establish the main measurable selection results in Appendix \ref{subs-Cmain}.

Throughout the section, we let $\ov{\Pol}$ and $\Pol$ be Polish spaces denoting the respective spaces in which edge and vertex marks lie, and let $d_{\bar{\Pol}}$ and $d_{\Pol}$ be associated metrics that induce the respective topologies. Let $\etra$ be an arbitrary point not lying in $\ov{\Pol}\cup \Pol$, define the spaces $\et{\ov{\Pol}}\defeq \ov{\Pol}\cup \{\etra\}$ and $\et{\Pol} \defeq \Pol\cup\{\etra\}$, and endow them with the corresponding Polish topologies from $\ov{\Pol}$ and $\Pol$, respectively, with $\etra$ being an isolated point. Throughout the section, we often implicitly denote (possibly random) $\sp{\ov{\Pol},\Pol}$-graphs by $G \defeq (V,E,\root,\ov{\dvms},\dvms)$ and $G_n \defeq (V_n,E_n,\root_n,\ov{\dvms}^n,\dvms^n)$. 
\begin{definition}[Random closed subsets] 
\label{def:mst} 
Given a Polish space $\Pol'$, let $\clos (\Pol')$ denote the set of closed subsets of $\Pol'$. Given in addition a measurable space $(\Omega,\filtm)$, a mapping $F: \Omega \mapsto \clos (\Pol')$ is said to be an $\filtm$-random closed subset of $\Pol'$ if it is weakly measurable in the following sense: for every open set $U\subseteq \Pol'$, the set $\{\omega \in \Omega: U \cap F(\omega) \neq \emptyset\}$ lies in $\filtm$. Moreover, $F$ is said to be non-empty if $F(\omega) \neq \emptyset$ for every $\omega \in \Omega$. 
\end{definition}

Several results in this appendix make use of the following measurable selection theorem.

\begin{theorem}[Kuratowski \& Ryll-Nardzewski Measurable Selection Theorem] 
\label{msbl:mst} 
Suppose $\Pol'$ is a Polish space, $(\Omega,\filtm)$ is a measurable space and $F: \Omega \mapsto \clos (\Pol')$ is a non-empty $\filtm$-random closed subset of $\Pol'$. Then there exists an $\filtm$-measurable function $Z:\Omega \to \Pol'$ such that $Z(\omega) \in F(\omega)$ for all $\omega \in \Omega$. 
\end{theorem} 
\begin{proof} 
This is simply a restatement of \cite[Theorem 6.9.3]{Bog07} in our notation, in particular with $X, \borel$, and $\Psi$ in \cite{Bog07} replaced by $\Pol'$, $\filtm$, and $F$, respectively. 
\end{proof}

\subsection{A Canonical Subspace of Rooted Graphs and its Properties} 
\label{subs-Cprelim}
In this section, we construct a canonical subspace of the space of $\sp{\ov{\Pol},\Pol}$-graphs introduced in Section \ref{nota:LWC} and equip it with a topology that is compatible with the topology of $\gs\sp{\ov{\Pol},\Pol}$. In the ensuing definition, we use the following standard notion of convergence of subsets of $\Nb$. Given $S_n \subseteq \Nb,n \in \Nb,$ and $S \subseteq \Nb$, we write $S_n \to S$ if and only if
\[S = \bigcup_{n \in \Nb}\bigcap_{n' > n} S_{n'} = \bigcap_{n \in \Nb}\bigcup_{n' > n} S_{n'}.\] 
Equivalently, $S_n \to S$ if and only if for every $k \in S$, there exist only finitely many $n\in\Nb$ such that $k \notin S_n$ and for every $k' \notin S$ there exist only finitely $n'\in\Nb$ such that $k' \in S_{n'}$. In addition, we equip $\Nb$ with the discrete topology. Lastly, recall the definition of $\gcl{v}{G}$ from Section \ref{nota:grph}.

\begin{definition}[Canonical space of representative graphs] 
\label{msbl:whgspce} 
We equip the canonical space $\wh{\gs}\sp{\ov{\Pol},\Pol}$ of rooted $\sp{\ov{\Pol},\Pol}$-graphs 
\begin{equation} 
\wh{\gs}\sp{\ov{\Pol},\Pol} \defeq \{\sp{\ov{\Pol},\Pol}\te{-graphs } G = (V,E,\root, \ov{\dvms}, \dvms) \te{ s.t. } V \subseteq \Nb\}, 
\end{equation} 
with the following notion of convergence: $G_n \to G$ in $\wh{\gs}\sp{\ov{\Pol},\Pol}$ if and only if 
\begin{enumerate} 
\item $\lim_{n\to\infty} V_n = V$; 
\item $\lim_{n\to\infty} E_n = E$; 
\item $\lim_{n\to\infty} \root_n = \root$; 
\item $\lim_{n\to\infty} \ov{\dvms}^n_e \indic{e \in E_n} = \ov{\dvms_e}$ for all $e \in E$; 
\item $\lim_{n\to\infty} \dvms^n_v \indic{v \in V_n} = \dvms_v$ for all $v \in V$; 
\item $\lim_{n\to\infty} \left[\max\{u \in \gcl{v}{G_n}\} \indic{v \in V_n}\right] = \max\{u \in \gcl{v}{G}\}$ for each $v \in V$. 
\end{enumerate}
For $G,G' \in\wh{\gs}\sp{\ov{\Pol},\Pol}$, notions such as graph distance $d_G(\cdot,\cdot)$, truncations $\trnc{m}(G)$, sets of isomorphisms $I(G,G')$ and isomorphism classes $\ic{G}$ are all defined as in Sections \ref{nota:grph} and \ref{nota:LWC}. 
\end{definition}

\begin{remark}[Definition \ref{msbl:whgspce}.6] 
\label{msbl:defcounter}
The least intuitive condition is perhaps condition 6 of Definition \ref{msbl:whgspce}, but it is necessary for the topology on $\wh{\gs}\sp{\ov{\Pol},\Pol}$ to be compatible with the topology of $\gs\sp{\ov{\Pol,\Pol}}$ (in the sense made precise in Lemma \ref{msbl:alphcont}). This is best illustrated via an example of how compatibility could fail without condition 6. Let $\ov{\Pol} = \Pol = \{1\}$ be trivial and suppose that for each $n \in \Nb$, $\root_n = \root = 1$, $V_n = \{1,n+1\}$, $V = \{1\}$, $E_n = \{\{1,n+1\}\}$, and $E = \emptyset$. Then conditions 1-5 of Definition \ref{msbl:whgspce} are all satisfied, and only condition 6 fails. However, note that $G_n \cong G_{n'} \ncong G$ for all $n,n'\in\Nb$, so $\ic{G_n} \to \ic{G_1} \neq \ic{G}$ in $\gs\sp{\ov{\Pol},\Pol}$.
\end{remark}

In order to apply Theorem \ref{msbl:mst} to find measurable representatives of isomorphism classes, it is necessary to prove that the space $\wh{\gs}\sp{\ov{\Pol},\Pol}$ is Polish. This can be done by direct verification using Definition \ref{msbl:whgspce}. We fill in the details for completeness.

\begin{lemma} 
\label{msbl:Pol} 
The space $\wh{\gs}\sp{\ov{\Pol},\Pol}$ is Polish. 
\end{lemma}
\begin{proof}
Define the Polish space $\Space \defeq \left(\Nb\times \et{\Pol}\times (\et{\ov{\Pol}})^\Nb\right)^{\Nb} \times \Nb$, equipped with the product topology, and  consider the map $\psi: \wh{\gs}\sp{\ov{\Pol},\Pol} \to \Space$  defined by $\psi(G) := ((\psi_k(G))_{k \in \Nb},\root)$, where for each $k \in \Nb$, $\psi_k(G) = (c'_k,\dvms'_k,(\ov{\dvms}'_{\{k,k'\}})_{k'\in\Nb}),$ with 
\[c'_k \defeq \begin{cases}
\max\{\gcl{k}{G}\} &\te{ if }k\in V,\\
k &\te{otherwise,}
\end{cases}\;
\quad \dvms'_k := \begin{cases}
\dvms_k&\te{ if } k \in V,\\
\etra &\te{ otherwise,}
\end{cases}\;
\quad 
\ov{\dvms}'_{\{k,k'\}} := 
\begin{cases}
\ov{\dvms}_{\{k,k'\}}&\te{ if } \{k,k'\}\in E,\\
\etra&\te{ otherwise.}
\end{cases}
\]
We then have the following observations:
\begin{itemize}
\item[(i)] \sloppy  The map $\psi$ is a bijection from $\wh{\gs}\sp{\ov{\Pol},\Pol}$ to $\Range$, where  $\Range$ is the subset of elements  $\zeta = ((c'_k,\dvms'_k,\{\ov{\dvms}'_{\{k,k'\}}\}_{k'\in\Nb})_{k \in \Nb},\root) $ in $\Space$ that satisfy the following constraints: 
\begin{enumerate}[(a)]
\item $\root\in V_\zeta := \{k \in \Nb: \dvms'_k \neq \etra\}$;
\item for every $k \notin V_{\zeta}$, $c'_k = k$;
\item One has $\{k,k'\}\subseteq V_{\zeta}$ 
for every $\{k,k'\}\in E_{\zeta} \defeq \cup_{k\in \Nb} E_\zeta(k)$, where  $E_\zeta(k) := \{\{k,k'\} \subset \Nb: k \neq k', \ov{\dvms}'_{\{k,k'\}} \neq \etra\}$;
\item for every $k \in \Nb$, $c'_k = \max[\{k\}\cup \{k'\in \Nb:\{k,k'\}\in E_\zeta(k)\}]$. 
\end{enumerate}
It is trivial to check  that for every $G\in\wh{\gs}\sp{\ov{\Pol},\Pol}$, $\psi(G)$ must satisfy conditions (a)--(d) above. Thus, the image of $\wh{\gs}\sp{\ov{\Pol},\Pol}$ under $\psi$ is contained in $\Range$. On the other hand, note that  any $\zeta = ((c'_k,\dvms'_k,\{\ov{\dvms}'_{\{k,k'\}}\}_{k'\in\Nb})_{k \in \Nb},\root) \in \Range$ has a unique inverse under $\psi$, described by $\psi^{-1}(\zeta) = G_{\zeta} \defeq (V_\zeta, E_\zeta, \root_\zeta, \ov{\dvms}_{k}(\zeta),\ov{\dvms}'_{\{k,k'\}}(\zeta))$, where $V_\zeta$ and $E_{\zeta}$ are respectively defined as in (a) and (c) above,  $\root_\zeta = \root$, $\dvms_{k}(\zeta) = \dvms'_{k} \in \Pol$ for $k \in V_\zeta$, and $\ov{\dvms}_{\{k,k'\}} (\zeta)= \ov{\dvms}'_{\{k,k'\}} \in \bar{\Pol}$ for $\{k,k'\} \in E_\zeta$. Condition (c) also ensures that $\zeta \in \Range$ implies $|E_\zeta(k)| < \infty$ for every $k \in V_\zeta$, and so the resulting graph $G_\zeta$ is locally finite. Moreover, conditions (a) and (c) together ensure that the edge and vertex marks lie in $\ov{\Pol}$ and $\Pol$, respectively, thus showing that $G_\zeta$ is a $\sp{\ov{\Pol},\Pol}$-graph. Lastly, it is easy to see that for any $G \in \wh{\gs}\sp{\ov{\Pol},\Pol}$ and $\zeta \in \Range$, $\psi\circ\psi^{-1}(\zeta) = \zeta$ and $\psi^{-1}\circ \psi(G) = G$, thus proving $\psi$ is a bijection between $\wh{\gs}\sp{\ov{\Pol},\Pol}$ and $\Range$.
\item[(ii)] $\Range$ is a closed subset of $\Space$ under componentwise convergence:  suppose the sequence $\zeta^n = ((c'^{,n}_k,\dvms'^{,n}_k,\{\ov{\dvms}'^{,n}_{\{k,k'\}}\}_{k'\in\Nb})_{k \in \Nb},\root^n) \in \Range, n \in \Nb,$ converges to $\zeta = ((c'_k,\dvms'_k,\{\ov{\dvms}'_{\{k,k'\}}\}_{k'\in\Nb})_{k \in \Nb},\root)$ pointwise. Then  $\zeta \in \Space$ because $\Space$ is Polish. To show $\zeta \in \Range$, it suffices to show that $\zeta$ satisfies the constraints (a)--(d) in (i) above. To prove condition (a), we argue by contradiction. Suppose $\dvms'_{\root} = \etra$. Then  since $\root^n \rightarrow \root$ it follows that there exists $N < \infty$ such  that $\root^n = \root$ for all $n \geq N$. Since $\dvms'^{,n}_{\root} \rightarrow \dvms'_{\root}$ and $\etra$ is isolated, this implies that for all sufficiently large $n$, $\dvms'^{,n}_{\root_n} = \dvms'^{,n}_{\root} = \etra$, which implies  $\root_n \not \in V_{\zeta^n}$ and thus contradicts the assumption that $\zeta^n \in \Range$. Thus, this proves that $\zeta$ satisfies condition (a). Condition (c) can be established in an exactly analogous fashion. Next, suppose $k \notin V_\zeta$.  Then since $\zeta$ satisfies (a) as shown above, $\dvms_k = \etra$ and since $\etra$ is isolated and  $\dvms'^{,n}_k \rightarrow \dvms_k$, it follows that  $\dvms'^{,n}_k = \etra$ for all sufficiently large $n$. In turn, since $\zeta^n \in \Range$, this implies $k \not \in V_{\zeta^n}$ and hence that $c_k'^{,n} = k$. Since $c_k'^{,n} \rightarrow c_k'$, this implies $c_k' = k$ and condition (b) follows for $\zeta$. Finally, condition (d) for $\zeta$, which implies each $E_{\zeta}(k)$, $k\in\Nb,$ is a finite set, can be deduced by similarly observing that $c_k'^{,n} = c_k$ and therefore $E_{\zeta^n}(k) = E_{\zeta}(k)$ for all $k\in\Nb$ and all sufficiently large $n$, and the fact that each $\zeta^n$ satisfies condition (d). 
\item[(iii)] The map $\psi$ is a homeomorphism from $\wh{\gs}\sp{\ov{\Pol},\Pol}$ to $\Range$: For each $n \in \Nb$, let $G_n = (V_n,E_n,\root_n,\ov{\dvms}_n,\dvms_n)\in \wh{\gs}\sp{\ov{\Pol},\Pol}$ and $\psi(G_n) = ((c'^{,n}_{k},\dvms'^{,n}_{k},\{\ov{\dvms}'^{,n}_{\{k,k'\}}\}_{k'\in\Nb})_{k\in \Nb},\root^n) \in \Range$. Then we wish to show that $G_n \to G = (V,E,\root,\dvms,\ov{\dvms})$ in $\wh{\gs}\sp{\ov{\Pol},\Pol}$ if and only if $\psi(G_n) \to \psi(G) \defeq ((c'_k,\dvms'_k,\{\ov{\dvms}'_{\{k,k'\}}\}_{k'\in\Nb})_{k\in\Nb},\root)$ in $\Range$.
\begin{description}
\item[(a)] If $G_n \to G$, then $\lim_{n\to\infty} \psi(G_n) = \psi(G)$: for any $k,k' \in \Nb$, conditions 1 and 2 of Definition \ref{msbl:whgspce} imply that if $k \in V$ and $\{k,k'\} \in E$, then $k \in V_n$ and $\{k,k'\}\in E_n$ for $n$ sufficiently large. Then for any $k \in V$, condition 6 implies that $\lim_{n\to\infty} c'^{,n}_k \to \lim_{\substack{n\to\infty\\k \in V_n}}\max \{\cl{k}(G_n)\} = \max\{\cl{k}(G)\} = c'_k$ and for $k \notin V$, $k \notin V_n$ for $n$ sufficiently large so $\lim_{n\to\infty}c'^{,n}_k = k = c'_k$. Conditions 4 and 5 imply that for $k \in V$ and $\{k,k'\}\in E$, $\lim_{n\to\infty} (\dvms'^{,n}_k,\ov{\dvms}'^{,n}_{\{k,k'\}}) = \lim_{\substack{n\to\infty,(\dvms'^{,n}_k,\ov{\dvms}'^{,n}_{\{k,k'\}})\in V_n\times E_n}} (\dvms^{n}_k,\ov{\dvms}^{n}_{\{k,k'\}}) = (\dvms_k,\ov{\dvms}_{\{k,k'\}}) = (\dvms'_k,\ov{\dvms}'_{\{k,k'\}})$. If $k \notin V$, then $k \notin V_n$ for $n$ sufficiently large and because $\etra$ is an isolated point, this implies that $\lim_{n\to\infty} \dvms'^{,n}_k = \etra = \dvms'_k$. By the same argument, if $\{k,k'\} \notin E$, then $\lim_{n\to\infty} \ov{\dvms}'^{,n}_{\{k,k'\}} = \etra = \ov{\dvms}^{'}_{\{k,k'\}}$. Lastly, $\root^n \to \root$ by condition 3 of Definition \ref{msbl:whgspce}. Thus, $\psi(G_n) \to \psi(G)$, which establishes the continuity of $\psi$.

\skipLine

\item[(b)] If $\psi(G_n) \to \psi(G)$, then $G_n \to G$: conditions 1 and 2 of Definition \ref{msbl:whgspce} follow from the convergence of $\dvms'^{,n}_k$ to $\dvms'_k$ and the convergence of $\ov{\dvms}'^{,n}_{\{k,k'\}}$ to $\ov{\dvms}^{'}_{\{k,k'\}}$ and the fact that $\etra$ is isolated in $\et{\Pol}$ and $\et{\ov{\Pol}}$. This directly implies that every vertex $k\in V$ and edge $\{k,k'\} \in E$ is in $V_n$ and $E_n$ respectively for all sufficiently large $n$, and likewise every non-vertex $k \in \Nb\setminus V$ and non-edge $\{k,k'\}\in \Nb\setminus E$ is not in $V_n$ or $E_n$ respectively for all sufficiently large $n$. Condition 3 follows from the convergence of $\root^n$ to $\root$. To prove condition 4, recall that we have already shown that any $k \in V$ is in $V_n$ for $n$ sufficiently large. Thus, $\lim_{\substack{n \to \infty\\ k \in V_n}} \dvms^n_k = \lim_{\substack{n \to \infty\\ k \in V_n}} \dvms'^{,n}_k = \lim_{n\to\infty}\dvms'^{,n}_k = \dvms'_k = \dvms_k$. The proof of condition 5 is exactly analogous except we apply the fact that $\{k,k'\} \in E$ implies $\{k,k'\}\in E_n$ for $n$ sufficiently large and then apply the convergence of $\ov{\dvms}'^{,n}_{\{k,k'\}}$ to $\ov{\dvms}'_{\{k,k'\}}$. Lastly, for each $k \in V$, because $c'^{,n}_k = c'_k$ for sufficiently large $n$, this implies that $E_{\zeta}(k)\cap \{c'_k+1,\dots\} = \emptyset$ for $n$ sufficiently large so $c'_k \geq \max \cl{k}(G)$. However, $\{k,c'_k\} \in E$ if and only if $\{k,c'_k\} \in E_n$ for all $n$ sufficiently large, and because $\{k,c'^{,n}_k\} \in E_n$ and $c'^{,n}_k = c'_k$ for $n$ sufficiently large, this implies that $\{k,c'_k\}\in E$ so $c'_k = \max\cl{k}(G)$. Thus, condition 6 holds.

\end{description}
Because $\psi$ is a one-to-one map onto $\Range$, this proves the claim. 
\end{itemize}
\sloppy Since $\wh{\gs}\sp{\ov{\Pol},\Pol}$ is homeomorphic to the closed subset $\Range$ of the Polish space $\Space$, it is also Polish. 
\end{proof}

Next, in Lemma \ref{msbl:alphcont} below, we show that the map $G\mapsto \ic{G}$ is continuous and its set-valued inverse map $\ic{G} \mapsto \{G\in \wh{\gs}\sp{\ov{\Pol},\Pol}: G \in \ic{G}\}$ is lower semicontinuous in the sense of \cite[Definition 1.4.2]{AubFra09}.

\begin{lemma}[Equivalence with the local topology] 
\label{msbl:alphcont} 
If $G_n \to G$ in $\wh{\gs}\sp{\ov{\Pol},\Pol}$, then the isomorphism classes also converge locally, that is, $\ic{G_n} \to \ic{G}$ in $\gs\sp{\ov{\Pol},\Pol}$. Moreover, given the limit $\ic{G_n} \to \ic{G}$ in $\gs\sp{\ov{\Pol},\Pol}$ and any representative $G \in \ic{G}$ such that $G \in \wh{\gs}\sp{\ov{\Pol},\Pol}$, there exists a sequence of representatives $G_n \in \ic{G_n}$, $n \in \Nb,$ such that $G_n \to G$ in $\wh{\gs}\sp{\ov{\Pol},\Pol}$. In other words, the correspondence $\gs\sp{\ov{\Pol},\Pol} \ni \ic{G} \mapsto \{G\in \wh{\gs}\sp{\ov{\Pol},\Pol}: G \in \ic{G}\}$ is lower semicontinuous. 
\end{lemma} 
\begin{proof} 
For convenience of notation, let $[n]\defeq \{1,\dots,n\}$. We start by proving the first statement. Fix any $m \in \Nb$, and let $M_m\defeq \max\{v \in \trnc{m}(G)\}$. Then conditions 1 and 2 of Definition \ref{msbl:whgspce} imply that for sufficiently large $n$, $V\cap [M_m] = V_n\cap [M_m]$ and $E\cap [M_m]^2 = E_n \cap [M_m]^2$. Condition 3 implies that $\root_n = \root$ for $n$ sufficiently large. Thus, the subgraphs of $\trnc{m}(\nm{G_*})$ and $\trnc{m}(\nm{G_{n,*}})$ induced by the set $[M_m]$ exactly match for sufficiently large $n$. Then by condition 6, $\max\{\gcl{v}{G_n}\} = \max\{\gcl{v}{G_n}\}$ for all $v \in V\cap [M_m]$ and $n$ sufficiently large. Because $\trnc{m}(G) \subseteq [M_m]$, these statements imply that all of the vertices in $\trnc{m}(G_n)$ also fall inside $[M_m]$, and hence, $\trnc{m}(\nm{G_*}) = \trnc{m}(\nm{G_{n,*}})$ for all sufficiently large $n$, that is $N_m \defeq \min \{n \in \Nb: \trnc{m}(\nm{G_{n',*}}) = \trnc{m}(\nm{G_*})\te{ for all }n' \geq n\}$ is finite (where the minimum of an empty set is taken to be $\infty$). For $n > N_m$, let $\varphi_{n,m}: \trnc{m}(G) \to \trnc{m}(G_n)$ be the identity isomorphism. Then for each $v \in \trnc{m}(G)$ and $e \in E_{\trnc{m}(G)}$, conditions 4 and 5 of Definition \ref{msbl:whgspce} imply $\lim_{\substack{n\to\infty\\n > N_m}} \dvms^n_{\varphi_{n,m}(v)} = \lim_{\substack{n\to\infty\\v \in V_n}} \dvms^n_v = \dvms_v,$ and $\lim_{\substack{n\to\infty\\n > N_m}} \ov{\dvms}^n_{\varphi_{n,m}(e)} = \lim_{\substack{n\to\infty\\ e \in E_n}} \ov{\dvms}^n_e = \ov{\dvms}_e$. By Definition \ref{def-locconvnoiso}, this proves that $\ic{G_n} \to \ic{G}$.

To prove the second statement, first consider the case when the representative $G \in \ic{G}$ has a vertex set that is canonical in the sense that $V = [|V|] = \{1,\dots,|V|\}$ where we interpret $[\infty]$ as $\Nb$. It is easy to see that one can always choose representatives $G'_n = (V'_n,E'_n,\root'_n,\ov{\dvms}'^{,n},\dvms'^{,n})$ of $\ic{G_n},$ $n\in\Nb,$ such that 
\begin{align} 
V'_n &= [|V'_n|] = \{1,\dots,|V'_n|\} \te{ for }n\in\Nb. \label{msbl:vnnogaps} 
\end{align} 
Then by Definition \ref{def-locconvnoiso}, for each $m \in \Nb_0$, there exist $n_m < \infty$ and a collection of isomorphisms $\varphi_{n,m}\in I(\trnc{m}(\nm{G_*}),\trnc{m}(\nm{G'_{n,*}}))$, $m \in \Nb_0, n > n_m,$ such that for each $m\in \Nb,$ the inclusions $v \in \trnc{m}(G)$ and $e \in E_{\trnc{m}(G)}$ imply 
\begin{equation} 
\label{msbl:notaconv} 
\lim_{\substack{n \to \infty\\n > n_m}} d_{\ov{\Pol}}(\ov{\dvms}'^{,n}_{\varphi_{n,m}(e)},\ov{\dvms}_e) = 0 \quad \te{ and } \quad \lim_{\substack{n \to \infty\\n > n_m}} d_{\Pol}(\dvms'^{,n}_{\varphi_{n,m}(v)},\dvms_v)= 0. 
\end{equation} 
Hence, there exists a sequence of non-decreasing integers $M_n, n \in \Nb,$ converging to infinity such that for each $n \in \Nb$, the inclusions $v \in \trnc{M_n}(G)$ and $e \in E_{\trnc{M_n}(G)}$ imply $d_{\Pol}(\dvms'^{,n}_{\varphi_{n,M_n}(v)},\dvms_v) < 2^{-M_n}$ and $d_{\ov{\Pol}}(\ov{\dvms}'^{,n}_{\varphi_{n,M_n}(e)},\ov{\dvms}_e) < 2^{-M_n}$ when $M_n > 0$. Set $\varphi_n \defeq \varphi_{n,M_n}$, and for each $n\in\Nb,$ define $\ov{\varphi}_n: \Nb \to \Nb$ by 
\[\ov{\varphi}_n(v) = 
\begin{cases} 
\varphi_n(v) &\te{ if } v \in \trnc{M_n}(G),\\ 
w^n_v &\te{ otherwise, } 
\end{cases}\] 
where if $v$ is the $k$th smallest element of $\Nb\setminus \trnc{M_n}(G)$, then $w^n_v$ is the $k$th smallest element of $\Nb\setminus \trnc{M_n}(G_n')$. Now for each $n\in\Nb$, define 
\[G_n \defeq (V_n,E_n,\root_n,\ov{\dvms}^n,\dvms^n)\defeq \left(\ov{\varphi}^{-1}_n(V'_n),\ov{\varphi}^{-1}_n(E'_n),\ov{\varphi}_n^{-1}(\root'_n),(\ov{\dvms}'^{,n}_{\ov{\varphi}_n(e)})_{e \in E_n},(\dvms'^{,n}_{\ov{\varphi}_n(v)})_{v \in V_n}\right).\] 
By construction, for each $n \in \Nb$, $V_n = V'_n = \{1,\dots,|V_n|\}$, just as in \eqref{msbl:vnnogaps}. Then by definition, for each $m \in \Nb$ and noting that (i) when $m \leq M_n$, $\ov{\varphi}_n|_{\trnc{m}(G)} = \varphi_{n,M_n}|_{\trnc{m}(G)}$ (which implies $\ov{\varphi}^{-1}_n|_{\trnc{m}(G_n)} = \varphi_{n,M_n}^{-1}|_{\trnc{m}(G_n)}$) and (ii) $M_n$ increases to infinity, one has 
\begin{equation}
\label{msbl:conv}
\trnc{m}(\nm{G_{n,*}}) = \varphi^{-1}_{n,M_n}(\trnc{m}(\nm{G'_{n,*}})) = \trnc{m}([G_*]) \te{ for }n\te{ sufficiently large.}
\end{equation}
For every $v \in V$ and $e \in E$, there exists $m \in \Nb$ such that $v \in \trnc{m}(G)$, $e \in E_{\trnc{m}(G)}$. Then by \eqref{msbl:conv}, $v \in V_n$ and $e \in E_n$ for $n$ sufficiently large. On the other hand, suppose $v \notin V$. Then, because $V$ is in canonical form, $|V|< v$ is finite. Thus, $G$ has radius $M < \infty$. Setting $m = M+1$ in \eqref{msbl:conv}, $\trnc{M+1}(\nm{G_{n,*}}) = \trnc{M+1}(\nm{G_*}) = \nm{G_*}$ for $n$ sufficiently large. Furthermore, because $\trnc{M+1}(\nm{G_{n,*}})$ has radius $M$, it follows that $\nm{G_{n,*}} = \trnc{M+1}(\nm{G_{n,*}}) = \nm{G_*}$. Thus, $v \notin V_n$ for $n$ sufficiently large, which implies that $V_n \to V$. Now, suppose $e = \{u,v\} \notin E$. Then either (without loss of generality) $u \notin V$, in which case $u \notin V_n$ for sufficiently large $n$ and $\{u,v\}\notin E_n$ for such $n$, or $u, v \in V$. In the latter case, we can fix $m$ so that $u,v \in \trnc{m}(G)$. Then for $n$ sufficiently large, $\trnc{m}(\nm{G_{n,*}}) = \trnc{m}(\nm{G_{*}})$ by \eqref{msbl:conv}, so $\{u,v\} \notin E_n$. Thus, $E_n \to E$. Setting $m = 0$, $\root_n = \ov{\varphi}_0^{-1}(\root'_n) = \root$ for all $n \in \Nb$. Lastly, for each $v \in V$, fix $m$ such that $v \in \trnc{m-1}(G)$. Then by \eqref{msbl:conv}, $\gcl{v}{G_n} = \gcl{v}{G}$ for $n$ sufficiently large. Thus, conditions 1-3 and 6 of Definition \ref{msbl:whgspce} hold. Now for each $v \in V$ and $e \in E$, there must exist $m \in \Nb$ such that $v \in \trnc{m}(G)$ and $e \in E_{\trnc{m}(G)}$. Then, by the definition of $M_n$, it follows that 
\begin{align*}
\lim_{n\to\infty} d_{\Pol}(\dvms^n_v,\dvms_v) &= \lim_{n\to\infty} d_{\Pol}(\dvms^{'n}_{\varphi_{n,M_n}(v)},\dvms_v) = \lim_{n\to\infty} 2^{-M_n} = 0\\
\lim_{n\to\infty} d_{\Pol}(\ov{\dvms}^n_e,\ov{\dvms}_e) &= \lim_{n\to\infty} d_{\Pol}(\ov{\dvms}^{'n}_{\varphi_{n,M_n}(e)},\ov{\dvms}_e) = \lim_{n\to\infty} 2^{-M_n} = 0.
\end{align*}
Thus, conditions 4 and 5 of Definition \ref{msbl:whgspce} also hold, proving that $G_n \to G$ in $\wh{\gs}\sp{\ov{\Pol},\Pol}$.

We finish the proof of the second assertion by considering the general case in which $G$ is an arbitrary representative of $\ic{G}$ with no restriction on the vertex set. In this case, there exists $G'\defeq (V',E',\root',\ov{\dvms}',\dvms') \cong G$ whose vertex set $V' \defeq V_{G'}$ is in canonical form. Using the argument above, construct the sequence $\{G'_n\}$ whose vertex sets $V'_n \defeq V_{G'_n}$ are in canonical form and such that $G'_n \to G'$ in $\wh{\gs}\sp{\ov{\Pol},\Pol}$. Given any isomorphism $\varphi \in I(G',G)$, define 
\[\ov{\varphi}(v) \defeq 
\begin{cases} 
\varphi(v) &\te{ if } v \in V_{G'},\\ 
w_v &\te{ otherwise,} 
\end{cases}\] 
where if $v$ is the $k$th smallest element of $\Nb\setminus V'$, then $w_v$ is the $k$th smallest element of $\Nb\setminus V$. Then $\ov{\varphi}: \Nb\to\Nb$ is a bijection, so for each $n \in \Nb$, 
\[G_n \defeq \ov{\varphi}(G_n) \defeq (\ov{\varphi}(V'_n),\ov{\varphi}(E'_n),\ov{\varphi}(\root'_n),(\ov{\dvms}^{'n}_{\ov{\varphi}^{-1}(e)})_{e \in \ov{\varphi}(E'_n)},(\dvms^{'n}_{\ov{\varphi}^{-1}(v)})_{v \in \ov{\varphi}(V'_n)}),\] 
is isomorphic to $G'_n$ and $G_n \to G$. Thus, the correspondence $\ic{G}\mapsto \{G'\in \wh{\gs}\sp{\ov{\Pol},\Pol}: G' \in \ic{G}\}$ is lower semicontinuous by \cite[Definition 1.4.2]{AubFra09}. 
\end{proof}

\subsection{Existence of Measurable Selections} 
\label{subs-Cmain}
The goal of this section is to establish Lemmas \ref{MG:iso} and \ref{MG:repextexist}. We begin this section by showing that every random isomorphism class of rooted graphs has a measurable random representative. 

\begin{lemma}[Measurable selection of representative graphs] 
\label{msbl:rep} 
Given any $\gs\sp{\ov{\Pol},\Pol}$-random element $\ic{G}$, there exists a $\sigma(\ic{G})$-measurable representative $\sp{\ov{\Pol},\Pol}$-random graph $G$ in $\wh{\gs}\sp{\ov{\Pol},\Pol}$. 
\end{lemma} 
\begin{proof} 
For every $\ic{H} \in \gs\sp{\ov{\Pol},\Pol}$, define 
\begin{equation} 
\label{def-psistar} 
\Psi_*(\ic{H}):= \{H' \in \wh{\gs}\sp{\ov{\Pol},\Pol}: H' \in \ic{H}\}. 
\end{equation} 
Then the set $\Psi_*(\ic{H})$ is non-empty since it contains $H$ and is also closed due to the continuity of the map $H \mapsto \ic{H}$ established in Lemma \ref{msbl:alphcont}. Thus $\Psi_*$ maps any isomorphism class in $\gs\sp{\ov{\Pol},\Pol}$ to the closed set in $\clos(\wh{\gs}\sp{\ov{\Pol},\Pol})$ that contains all graphs in $\wh{\gs}\sp{\ov{\Pol},\Pol}$ that lie in that isomorphism class. Given the $\gs\sp{\ov{\Pol},\Pol}$-random element $\ic{G}$, we first argue that to prove the lemma it suffices to prove the claim that for every open set $U\subseteq \wh{\gs}\sp{\ov{\Pol},\Pol}$, the following set is open in $\gs\sp{\ov{\Pol},\Pol}$: 
\[\Lambda_*^U := \{\ic{H}\in \gs\sp{\ov{\Pol},\Pol}: \Psi_*(\ic{H})\cap U \neq \emptyset\} = \bigcup_{H \in U} \{\ic{H}\}.\] 
Indeed, since $\ic{G}$ is a $\gs\sp{\ov{\Pol},\Pol}$-random element, the claim implies that $\psi_*(\ic{G})$ is a non-empty $\sigma(\ic{G})$-random closed subset of $\wh{\gs}\sp{\ov{\Pol},\Pol}$ in the sense of Definition \ref{def:mst}, and the lemma follows on applying Theorem \ref{msbl:mst} with $\Pol' = \wh{\gs}\sp{\ov{\Pol},\Pol}$, $F = \psi_*(\ic{G})$, and ${\mathcal F} = \sigma(\ic{G})$. 

We now turn to the proof of the claim. If $U$ is empty then so is $\Lambda_*^U$. Now, fix $U\subseteq \wh{\gs}\sp{\ov{\Pol},\Pol}$ non-empty. Then $\Lambda_*^U$ is trivially non-empty as well. For $\ic{H} \in \Lambda_*^U$, suppose there exists a sequence $\{\ic{H_n}\}_{n \in\Nb}$ of (deterministic) elements of $\gs\sp{\ov{\Pol},\Pol}$ converging to $\ic{H}$. Select $H'\in \Psi_*(\ic{H})\cap U$ (which is non-empty since $\Lambda_*^U \neq \emptyset$). By Lemma \ref{msbl:alphcont}, there exists a sequence of representative graphs $H'_n\in \ic{H_n}$, $n\in\Nb,$ that converges to $H'$ in $\Pol'$. Because $H'\in U$ and $U$ is open, $H'_n \in U$ for all but finitely many $n$. By the definition of $\Lambda_*^U$, this implies $\ic{H_n} \in \Lambda_*^U$ for all but finitely many $n$. Because this is true for any sequence converging to an element in $\Lambda_*^U$, and because $\gs\sp{\ov{\Pol},\Pol}$ is a Polish space (in which convergence is equivalent to sequential convergence) $\Lambda_*^U$ is open. This concludes the proof. 
\end{proof}

\begin{remark}[A Polish space of maps] 
\label{msbl:mpspace} 
For $W_1,W_2 \subseteq \Nb$, let $\mps(W_1,W_2)$ be the space of mappings from $W_1$ to $W_2$, and define $\mps \defeq \bigcup_{W_1,W_2\subseteq \Nb} \mps(W_1,W_2)$. Then any map $f \in \mps(W_1,W_2)\subset \mps$ can be embedded in $\Nb_0^{\Nb}$ (equipped with the discrete product topology) via the bijective map 
\begin{equation} 
\label{map-beta} 
\mpsmp(f) \defeq (\mpsmp_n(f))_{n\in\Nb}\te{ where } 
\begin{cases}
f(n) & \te{ if } n\in W_1, \\ 
0 & \te{ otherwise.}
\end{cases}
\quad \mbox{ for } f \in \mps(W_1,W_2). 
\end{equation} 
We equip $\mps$ with the topology induced by $\mpsmp$; that is, we define a subset $U\subseteq \mps$ to be open if and only if $\mpsmp(U)\subseteq \Nb_0^{\Nb}$ is open. With this definition, $\mpsmp$ is automatically a homeomorphism, and $\mps$ is a Polish space. We now apply Theorem \ref{msbl:mst} to select the isomorphisms in Lemma \ref{MG:iso} in a measurable manner. 
\end{remark}

Applying Theorem \ref{msbl:mst}, it is possible to establish Lemma \ref{msbl:iso}, which shows that the isomorphisms in Lemma \ref{MG:iso} can be chosen in a measurable manner, and Lemma \ref{msbl:dmap}, which shows that driving maps on different random graphs can be constructed so as to be consistent with random isomorphisms between those graphs. The latter property is used to construct the driving noise in Lemma \ref{MG:repextexist} in a measurable manner.

\begin{lemma}[Measurable selection of isomorphisms] 
\label{msbl:iso} 
Fix a probability space $\pspace$ that supports finite, $\filtm$-measurable $\wh{\gs}\sp{\ov{\Pol},\Pol}$-random elements $G_i$, $i=1,2,$ that satisfy $\PP(G_1\cong G_2) > 0$. Then given any $\filtm$-random closed subset $A$ of $\mps$ such that $A$ is a non-empty subset of $I(G_1,G_2)$ on the event $\{G_1\cong G_2\}$, there exists an $\filtm$-measurable map $\varphi \in \mps(V_1,V_2)$ such that $\varphi\in A$ on the event $\{G_1\cong G_2\}$. 
\end{lemma} 
\begin{proof} 
Given any $\filtm$-random closed subset $A$ of $\Mmc$ such that $A\subseteq I(G_1,G_2)$, consider the set-valued mapping $B(\omega) \defeq \mps(V_1,V_2)$ if $G_1\ncong G_2$ and $B(\omega) \defeq A$ otherwise. Note that by finiteness of $G_1$ and $G_2$, $B$ is also finite because $B\subseteq \mps(V_1,V_2)$, and thus $B(\omega)$ is closed for all $\omega\in \Omega$. Let $\ov{B} = \{\mpsmp(f):f \in B\}$, with $\mpsmp$ as in \eqref{map-beta}, and note that $\ov{B}(\omega)$ is similarly finite and therefore closed and non-empty for all $\omega \in \Omega$. We claim that $\ov{B}(\omega)$ is weakly $\filtm$-measurable. That is, for every open $U \subseteq \Nb_0^{\Nb}$, the set $\{\ov{B}\cap U \neq \emptyset\}$ is $\filtm$-measurable. If the claim holds, then $\ov{B}$ is a non-empty $\filtm$-random closed subset, so by Theorem \ref{msbl:mst} with $\Pol' = \Nb_0^{\Nb}$ and $F = \ov{B}$, there exists an $\filtm$-measurable random variable $b$ such that $b(\omega) \in \ov{B}(\omega)$ for every $\omega\in\Omega$. The lemma follows on setting $\varphi = \mpsmp^{-1}(b)$.

To prove the claim, first note that the set $\{\omega: G_1 \cong G_2\}$ is $\filtm$-measurable by assumption. Moreover, fixing any open set $U \subseteq \Nb_0^{\Nb}$, note that $\{G_1 \cong G_2\}\cap \{\ov{B}\cap U \neq \emptyset\} = \{G_1 \cong G_2\}\cap \{A\cap \mpsmp^{-1}(U) \neq \emptyset\}$, and also that $U' := \mpsmp^{-1}(U) \subset \mps$ is open since $\mpsmp$ is continuous. Because $A$ is an $\filtm$-random closed subset, this implies $\{G_1 \cong G_2\}\cap \{\ov{B}\cap U \neq \emptyset\}$ is $\filtm$-measurable. Next, observe that $\{G_1 \ncong G_2\}\cap \{\ov{B}\cap U \neq \emptyset\} = \{G_1 \ncong G_2\}\cap \{\mps(V_1,V_2)\cap \mpsmp^{-1}(U) \neq \emptyset\} = \{G_1 \ncong G_2\}\cap \{\mps(V_1,V_2)\cap U' \neq \emptyset\}$. However, note that 
\[\{\mps(V_1,V_2)\cap U' \neq \emptyset\} = \bigcup_{\te{finite } W_1,W_2\subset \Nb} \{V_i(\omega) =W_i,i=1,2\} \cap \{ U' \cap \mps(W_1,W_2)\neq \emptyset\},\] 
which is a countable union of $\filtm$-measurable sets since $U'$ is open. Thus, $\{\ov{B}\cap U \neq \emptyset\}$ is $\filtm$-measurable for every open $U$, and the claim follows. 
\end{proof}

The following lemma ensures driving maps on different graphs can be constructed so as to be consistent with isomorphisms between those graphs, which is required in Sections \ref{LWCpf} and \ref{GEM}. 

\begin{lemma}[Existence of consistent driving maps] 
\label{msbl:dmap} 
Fix a probability space $\pspace$ that supports the $\filtm$-measurable $\wh{\gs}\sp{\ov{\Pol},\Pol}$-random elements $G_i$, $i=1,2$. Let $M$ be an $\filtm$-measurable random variable such that $M (\omega) \in \{m \in \Nb_0: \trnc{m}(G_{1}(\omega))\cong \trnc{m}(G_{2}(\omega))\}$ for every $\omega \in \Omega$. Then for any $\filtm$-measurable isomorphism $\varphi\in I(\trnc{M}(G_1),\trnc{M}(G_2))$ and $\filtm$-measurable driving map $\psi_1:V_1 \to \Nb$ such that $\Nb\setminus \psi_1(V_1)$ is infinite, there exists an $\filtm$-measurable driving map $\psi_2: V_2 \to \Nb$ such that $\psi_1(v) = \psi_2(\varphi(v))$ for every $v \in \trnc{M}(G_1)$ and $\psi_1(G_1)\cap \psi_2(G_2\setminus\trnc{M}(G_2)) = \emptyset$. 
\end{lemma} 
\begin{proof} 
Fix an $\filtm$-measurable random variable $M$ and $\filtm$-measurable isomorphism $\varphi\in I(\trnc{M}(G_1),\trnc{M}(G_2))$ as in the statement of the lemma. Let $\widetilde{\Pol} \defeq (\Nb_0^{\Nb})^2$, equipped with the product topology, and consider the $\widetilde{\Pol}$-random element $b = (b_1, b_2)$ with $b_i = (b_{i,k})_{k\in \Nb}$ given, for $i = 1, 2$ and $k \in \Nb$, by 
\[b_{i,k} = 
\begin{cases} 
0 &\te{ if } k \notin G_i,\\ 
\psi_1(k) & \te{ if } k \in V_{G_1},\\ 
\psi_1(\varphi^{-1}(k)) & \te{ if } k \in \trnc{M}(G_2)\te{ and } i = 2,\\ 
\alpha(k) & \te{ if } k \in V_2\setminus \trnc{M}(G_2) \te{ and } i = 2, 
\end{cases}\] 
where $\alpha:\Nb\to\Nb$ is an injection mapping each $k$ to the $k$th smallest element of $\Nb\setminus \psi_1(V_1)$. Observe that $\alpha$ is well defined because $\Nb\setminus\psi_1(V_1)$ is infinite and is also an $\filtm$-measurable element of $\mps$ because $\psi_1$ is $\filtm$-measurable. Thus, $b$ is clearly also $\filtm$-measurable. Define $\psi_i \defeq \mpsmp^{-1}(b_i)$, $i=1,2$, with $\mpsmp$ as in \eqref{map-beta}, and note that since $\beta^{-1}$ is Borel measurable, each $\psi_i$ is also $\filtm$-measurable. Because $\psi_1(V_1)$ and $\alpha(\Nb)$ are disjoint, $\psi_2$ is also injective and therefore a driving map. Furthermore, for each $k \in \trnc{M}(G_1)$, $\psi_1(k) = \psi_1(\varphi^{-1}(\varphi(k))) = \psi_2(\varphi(k))$, and $\psi_1(G_1) \subseteq \psi_1(V_1)$ and $\psi_2(G_2\setminus\trnc{M}(G_2)) \subseteq \alpha(\Nb)$, so $\psi_1(G_1) $ and $\psi_2(G_2\setminus \trnc{M}(G_2))$ are disjoint. This completes the proof. 
\end{proof}

We finish the appendix with proofs of Lemmas \ref{MG:iso} and \ref{MG:repextexist}.

\begin{proof}[Proof of Lemma \ref{MG:iso}] 
The first assertion follows from Lemma \ref{msbl:rep} and the fact that $\sigma(\ic{(\ov{G}_n,\ov{\x}^n)})\subseteq \wh{\filtm}$. For the second assertion, let $\{(G_n,\x^n)\}_{n \in \Nb},(G,\x)$ be any sequence of measurable representatives satisfying Properties 1 and 2 of Definition \ref{MG:rcs}. By Definition \ref{def-locconvnoiso}, the convergence of $\ic{(G_n,\x^n)}$ to $\ic{(G,\x)}$ in $\gs\sp{\emksp,\vmksp\times \Xc}$ implies that for every $m \in \Nb$, $\trnc{m}(\nm{G_{*}})\cong \trnc{m}(\nm{G_{n,*}})$ for sufficiently large $n$. Since $\Xc$ is discrete, this further implies that there exists an a.s. finite random variable $N_m$ such that $\trnc{m}(\nm{G_{*}},\x) \cong \trnc{m}(\nm{G_{n,*}},\x^n)$ for all $n \geq N_m$ and $N_m$ is $\wh{\filtm}$-measurable. Then for any $\varphi \in \mps$, define the following random variable: 
\[\Psi_{n,m}(\varphi) \defeq 
\begin{cases} 
\sum_{v \in \trnc{m}(G)} d_{\vmksp}(\vms^n_{\varphi(v)},\vms_v) + \sum_{e \in E_{\trnc{m}(G)}} d_{\emksp}(\ems^n_{\varphi(e)},\ems_e)&\te{ if } \varphi\in I_{n,m},\\ 
\infty &\te{ otherwise,} 
\end{cases}\] 
where $I_{n,m} \defeq I(\trnc{m}(\nm{G_*},\x),\trnc{m}(\nm{G_{n,*}},\x^n))$. Let $\minset_{n,m} = \te{argmin}_{\varphi \in \mps} \Psi_{n,m}(\varphi)$, where we define $\minset_{n,m}$ to be empty if $\Psi_{n,m}(\varphi) = \infty$ for all $\varphi \in\mps$. On the $\wh{\filtm}$-measurable set $\{n \geq N_m\}$, $I_{n,m}$ is non-empty and finite, and so $\minset_{n,m}$ is also non-empty and finite. We first show that to prove the following claim:\\

\noindent \textbf{Key claim: } $\minset_{n,m}$ is a $\wh{\filtm}$-random closed subset of $\mps$ that is non-empty on the event $\{I_{n,m} \neq \emptyset\}$.\\

Deferring the proof of the claim, first note that given the claim, applying Lemma \ref{msbl:iso}, with $G_1 = \trnc{m}(\nm{G_*},\x)$, $G_2 = \trnc{m}(\nm{G_{n,*}},\x^n)$, and $A = \minset_{n,m}$, there exists a sequence of $\wh{\filtm}$-measurable maps $\{\ov{\varphi}_{n,m}\}_{n,m \in \Nb}$ such that $\ov{\varphi}_{n,m} \in \minset_{n,m}$ for every $n,m \in \Nb$ on the event $\{n \geq N_m\}$. Hence, by the definition of $\minset_{n,m}$, the fact that $\ov{\varphi}_{n,m} \in I_{n.m}$ and by Definition \ref{def-locconvnoiso}, for any $m \in \Nb$, 
\[\lim_{n \to \infty,n > \ov{N}_m} \min_{\varphi\in I_{n,m}} \Psi_{n,m}(\varphi) = \lim_{n \to \infty,n > N_m}\Psi_{n,m}(\ov{\varphi}_{n,m}) = 0.\] 
This implies that for every $m \in \Nb$ there exists a $\wh{\filtm}$-measurable integer $\wh{N}_m$ such that $\Psi_{n,m}(\ov{\varphi}_{n,m}) < 2^{-m}$ for all $n \geq \wh{N}_m$. Moreover, $\wh{N}_m$ is non-decreasing and a.s. finite and so $M_n \defeq \max\{m \in \Nb: n \geq \wh{N}_m\}$ is $\wh{\filtm}$-measurable and increases to infinity. Therefore, property 3 of Definition \ref{MG:rcs} is satisfied. For each $n\in \Nb$ and $m \leq M_n$, define $\varphi_{n,m} \defeq \ov{\varphi}_{n,M_n}|_{\trnc{m}(G)}$. It follows that $\{\varphi_{n,m}\}_{n \in \Nb,m\leq M_n}$ is also an $\wh{\filtm}$-measurable sequence and satisfies property 4 of Definition \ref{MG:rcs}. Furthermore, Properties 5 and 6 follow directly from the fact that for $v \in \trnc{m}(G)$ and $e \in E_{\trnc{m}(G)}$, 
\[\max\{d_{\vmksp}(\vms^n_{\varphi_{n,m}(v)},\vms_v),d_{\emksp}(\ems^n_{\varphi_{n,m}(e)},\ems_e)) \leq \Psi_{n,M_n}(\ov{\varphi}_{n,M_n}) < 2^{-M_n} \to 0\quad \te{as}\quad n\to \infty.\] 

We now turn to the proof of the key claim. Fix $n,m \in \Nb$. We first prove the following:\\

\noindent \textbf{Sub-Claim 1:} For each $\varphi \in \mps$, $\Psi_{n,m}(\varphi)$ is $\wh{\filtm}$-measurable.\\ \noindent \emph{Proof of Sub-Claim 1:} Fix $\varphi \in \mps$. Define ${\mathcal A}^\varphi_{m} := {\mathcal A}_{m,1} \cap {\mathcal A}^\varphi_{m,2}$, where 
\begin{eqnarray*} 
{\mathcal A}_{m,1} & := & \left\{(H_1,H_2) \in \wh{\gs}\sp{\emksp,\vmksp \times \Xc}^2: \max_{i=1,2,v \in V_{H_i}}d_{H_i}(\root,v) \leq m\right\}, \\ 
{\mathcal A}^\varphi_{m,2} & := & \left\{(H_1,H_2) \in \wh{\gs}\sp{\emksp,\vmksp \times \Xc}^2: \varphi \in I(\theta(H_1),\theta(H_2))\right\}, 
\end{eqnarray*} 
with $\theta: \wh{\gs}\sp{\emksp,\vmksp\times\Xc} \to \wh{\gs}\sp{1,\Xc}$ being the mapping that takes $H = (V_{H},E_{H},\root_{H},\emsn{H},\vmsn{H},\x^{H}) \in \wh{\gs}\sp{\emksp,\vmksp \times \Xc}$ to $(V_{H},E_{H},\root_{H}, \x^{H}) \in \wh{\gs}\sp{1,\Xc}$, which is the rooted representative graph $H$ with only the $\Xc$-valued vertex marks retained. Then consider the mapping $\Map_m^\varphi:(\wh{\gs}\sp{\emksp,\vmksp\times \Xc})^2\to\Rb_+\cup\{\infty\}$ given by 
\begin{align*} 
&\Map_{m}^\varphi (H_1,H_2) \defeq 
\begin{cases} 
\sum_{v \in V_{H_1}} d_{\vmksp}(\vmsn{H_2}_{\varphi(v)},\vmsn{H_1}_v) + \sum_{e \in E_{H_1}} d_{\emksp}(\emsn{H_2}_{\varphi(e)},\emsn{H_1}_e) & \te{ if } (H_1, H_2) \in {\mathcal A}^\varphi_{m}, \\ 
\infty &\te{ otherwise. } 
\end{cases} 
\end{align*} 
Since $\Psi_{n,m}(\varphi) = \Map_m^{\varphi}(\trnc{m}(G,\x),\trnc{m}(G_n,\x^n))$, to prove Sub-Claim 1, it suffices to show the following:\\

\noindent \textbf{Sub-Claim 2:} The map $\Map_m^{\varphi}$ is continuous.\\ \noindent \emph{Proof of Sub-Claim 2:} Fix $\varphi \in \mps(W_1,W_2)$ and $m \in \Nb$. If $W_1$ is infinite, $\Map^{\varphi}_m\equiv \infty$, and so is trivially continuous. Next, suppose $|W_1| < \infty$. Then, since $\theta$ and $H \mapsto \max_{v \in V_H} d_H(\root,v)$ are continuous and $\{(H_1,H_2)\in (\wh{\gs}\sp{1,\Xc})^2: \varphi \in I(H_1,H_2)\}$ is closed, it follows that ${\mathcal A}_{m,1}$ and ${\mathcal A}^{\varphi}_{m,2}$ are closed, so ${\mathcal A}_{m}^{\varphi}$ is likewise closed. Let $\{(H^n_1,H^n_2)\}_{n\in\Nb} \subseteq \wh{\gs}\sp{\emksp,\vmksp} \subset {\mathcal A}^\varphi_m$ be any sequence that is convergent in $\wh{\gs}\sp{\emksp,\vmksp}$, and let $(H^{\infty}_1,H^{\infty}_2)$ denote its limit. Then $(H^{\infty}_1,H^{\infty}_2) \in {\mathcal A}^{\varphi}_m$ since ${\mathcal A}^{\varphi}_m$ is closed. Moreover, by conditions 1 and 2 of Definition \ref{msbl:whgspce}, $V_{H^n_1} = V_{H^\infty_1} = W_1$ and $E_{H^n_1} = E_{H^{\infty}_1}$ for all sufficiently large $n$. Then the convergence of $\Map^{\varphi}_m(H^n_1,H^n_2)$ to $\Map^{\varphi}_m(H^\infty_1,H^\infty_2)$ is an immediate consequence of the definition of $\Map^{\varphi}_m$ and conditions 4 and 5 of Definition \ref{msbl:whgspce}. Lastly note that the map $\theta_2: (\wh{\gs}\sp{\emksp,\vmksp\times \Xc})^2 \to (\wh{\gs}\sp{1,\Xc})^2$ defined by $\theta_2(H_1,H_2)\defeq (\theta(H_1),\theta(H_1))$ is continuous and that $A^{\varphi}_m = \theta_2^{-1}(\ov{A}^{\varphi}_m)$ where $\ov{A}^{\varphi}\subseteq (\wh{\gs}\sp{1,\Xc})^2$ consists entirely of pairs of graphs with radius at most $m$. However, it is easily verified by Definition \ref{msbl:whgspce} that (since $\Xc$ is a discrete space) any finite graph is an isolated point in $\wh{\gs}\sp{1,\Xc}$, which implies that $\ov{A}^{\varphi}_m$, and therefore $A^{\varphi}_m$, must also be open. Sub-Claim 2 then follows on noting that $\Map^{\varphi}_m$ is identically equal to infinity and thus trivially continuous on the closed set $({\mathcal A}^{\varphi}_m)^c$. 

\skipLine

Next, define $\Psi^{\te{min}}_{n,m} \defeq \min_{\varphi \in \mps}\Psi_{n,m}(\varphi)$ where $\Psi^{\te{min}}_{n,m} = \infty$ when $I_{n,m} = \emptyset$. Then $\Psi^{\te{min}}_{n,m}$ always exists because $I_{n,m}$ is finite. Note that since $|B_m(G)| + |B_m(G_n)| < \infty$, it follows that 
\[\Psi^{\te{min}}_{n,m} = \min_{\substack{\varphi\in \mps(W_1,W_2):\\|W_1|+|W_2|<\infty}} \Psi_{n,m}(\varphi),\] 
is a minimum over a countable collection of $\wh{\filtm}$-measurable random variables. Sub-Claim 1 then shows that $\Psi^{\te{min}}_{n,m}$ is also $\wh{\filtm}$-measurable. For any open $U\subseteq \mps$, 
\[\{\minset_{n,m}\cap U \neq\emptyset\} = \bigcup_{\substack{\varphi \in \mps(W_1,W_2)\cap U: \\ |W_1|+|W_2|<\infty}} \{\Psi_{n,m}(\varphi) = \Psi^{\te{min}}_{n,m}\}\cap \{\Psi^{\te{min}}_{n,m} < \infty\}.\] 
Since this is a countable union of $\wh{\filtm}$-measurable sets, the key claim follows from Definition \ref{def:mst}. 
\end{proof}

\begin{proof}[Proof of Lemma \ref{MG:repextexist}:] 
Let $\{\ic{(G_n,\x^n)}\}_{n \in \Nb}$ be a random sequence on $\pzspace$ that converges a.s. to $\ic{(G,\x)}$ in $\gs\sp{\emksp,\vmksp}$. By Lemma \ref{MG:iso}, there exists a rep-con sequence $(\{(G_n,\x^n),M_n\}_{n \in \Nb}$, $(G,\x),\{\varphi_{n,m}\}_{n\in\Nb,m\leq M_n})$ of $(\{\ic{(G_n,\x^n)}\}_{n \in \Nb},\ic{(G,\x)})$ defined on the same probability space $\pzspace$. Let $\psi$ be a $\wh{{\mathcal F}}$-measurable driving map $\psi: V_G \to \Nb$ such that $\Nb \setminus \psi(G)$ is infinite (e.g., consider the map $\psi(k) = 2k$). Then, invoking Properties 3 and 4 of Definition \ref{MG:rcs} and repeatedly applying Lemma \ref{msbl:dmap} with $M = M_n$, $G_1 = (\nm{G_*},\x)$, $G_2 = (\nm{G_{n,*}},\x^n)$, $\filtm = \wh{\filtm}$, $\varphi = \varphi_{n,M_n}$, and $\psi_1 = \psi$, for each $n \in \Nb$ we can construct a $\wh{\filtm}$-measurable driving map $\psi_n:G_n \to \Nb$ such that for every $m \leq M_n$ and $v \in \trnc{m}(G)$, $\psi_n(\varphi_{n,m}(v)) = \psi(v)$. Then extending the space $\pzspace$ to add a countable sequence of i.i.d. Poisson processes $\{\poiss_k\}_{k \in \Nb}$ and using the driving maps $\psi_n$, $n \in \Nb$, and $\psi$ to generate the respective $\filt$-driving noise $(G_n,\poiss^{G_n})$, $n\in\Nb,$ and $(G,\poiss^G)$ as in Definition \ref{MG:driving}, we obtain a consistent rep-con extension. 
\end{proof}

\section{Verification of Assumptions} 
\label{aver} 
\subsection{Well-Posedness Under Spatially Heterogeneous Dynamics} 
\label{saver}
In this section, we extend our well-posedness results to spatially heterogeneous dynamics as mentioned in Remark \ref{rem-semiregularity}. Consider IPS with local jump rates that are parameterized by jumps $j \in \jmps$ and graphs $H$ whose vertex labels lie in $\Nb$ (one can trivially replace $\Nb$ by any countable, deterministic set) rather than the slightly more general space $\gsone$ of rooted graphs of radius one considered in Section \ref{mod:dyn}; see Remark \ref{rem-restriction} below for the reason for this restriction. Recalling the Polish spaces $\wh{\gs}$ and $\wh{\gs}\sp{\emksp,\vmksp}$ from Definition \ref{msbl:whgspce}, let $\wgsone \defeq \wh{\gs}\cap \gsone \subset \wh{\gs}$ be the space of rooted (unmarked) graphs with vertex labels in $\Nb$ and with radius one. In this framework, we denote the local jump rates as follows: 
\[\wh{\mathbf{\locrate}} \defeq \{\wh{\locrate}^H_j: \Rb_+\times\cad^{V_H}\times \emksp^{E_H}\times \vmksp^{V_H}\to \Rb_+, H = (V_H,E_H,\root_H)\in \wgsone,j \in \jmps\}.\] 
Then, just as in the homogeneous case, the jump rates ${\bf \rate^G}$ are derived from the local jump rates in a manner analogous to \eqref{eq:standing}: 
\begin{equation} 
\label{eq:altstanding} 
r^{G,v}_j(t,x) = \wh{\locrate}^{H_v}_j(t,x_{H_v},(\ems_e)_{e \in E_{H_v}},(\vms_v)_{v \in V_{H_v}}),\te{ for all } (t,x) \in \Rb_+\times \cad^{V_G}, 
\end{equation} 
where $H_v \defeq (G[\cl{v}],v)$. Note that the jump rates ${\bf \rate^G}$ are now only well defined for $G \in \wh{\gs}\sp{\emksp,\vmksp}$, and the SDE \eqref{mod:infpart} is likewise also only well defined when the initial data $(G,\x)$ is a $\wh{\gs}\sp{\emksp,\vmksp\times\x}$-random element.

\begin{remark}[Vertex labels] 
\label{rem-restriction} 
Considering only graphs with vertex labels in $\Nb$ is not very restrictive in practice and is done for the purely technical reason of simplifying measurability considerations. This restriction allows us to parameterize the dynamics of particles by their labels in $\Nb$, in which case the space $\wgsone$ is countable and discrete (all points are isolated), and the local jump rates are trivially measurable as a function of their inputs and the underlying graph that parameterizes them. 
\end{remark}

We now define the notion of \emph{quasi-regular} local jump rates:

\begin{definition}[Quasi-regularity of local jump rates] 
\label{def-semiregular} 
Given $j \in \jmps$, the family of local jump rates $\wh{\locrate}^H_j: \Rb_+\times \cad^{V_H}\times \emksp^{E_H}\times \vmksp^{V_H}\to \Rb_+$, $H = (V_H,E_H,\root_H) \in \wgsone$, is said to be quasi-regular if for each $H \in \wgsone, j \in \jmps$, and $(\ems,\vms) \in \emksp^{E_H}\times \vmksp^{V_H}$, $\wh{\locrate}^H_j$ is Borel measurable and the map $(t,x)\mapsto \wh{\locrate}^H_j(t,x,\ems,\vms)$ is predictable in the sense that for every $t > 0$ and $x,y \in \cad^{V_H}$, 
\[x(s) = y(s) \te{ for all }s \in [0,t)\quad \Rightarrow \quad \wh{\locrate}^H_j(t,x,\ems,\vms) = \wh{\locrate}^H_j(t,y,\ems,\vms).\] 
\end{definition}

Definition \ref{def-semiregular} only differs from Definition \ref{def-regular} in that the symmetry condition is removed. Thus, we replace the standing assumption with the condition that the local jump rates are quasi-regular.

\begin{assumptio}[A weaker standing assumption] 
\label{assu-newstanding} 
Given any $G = (V,E, \root, \ems,\vms)\in \wh{\gs}\sp{\emksp,\vmksp}$, the jump rates ${\bf \rate^G}$ for the IPS dynamics on the graph $G$ satisfy \eqref{eq:altstanding} for local jump rates $\wh{\mathbf{\locrate}} = \{\wh{\locrate}^{H}_j: \Rb_+ \times \cad^{V_H} \times \emksp^{E_H} \times \vmksp^{V_H} \rightarrow \Rb_+\}_{H = (V_H, E_H, \root_H) \in \wgsone}, j \in \jmps,$ that are quasi-regular in the sense of Definition \ref{def-semiregular}. 
\end{assumptio}

 The definitions of weak/strong solutions to the SDE \eqref{mod:infpart}, and weak/strong well-posedness of \eqref{mod:infpart} are identical to those given in Definition \ref{mod:sol} and Definition \ref{mod:WP}, respectively, with the driving noises $(G,\poiss^G)$ being a simple sub-collection of i.i.d. Poisson processes, as specified in Remark \ref{rem:drivingnoise}.  

We now introduce the analog of Assumption \ref{mod:assu}, which is essentially exchanged except for the fact that the maximum of the local jump rates is now defined over a smaller set of graphs.

\begin{assumptionp}{$\wh{\ref{mod:assu}}$}[Bounds on the heterogenous local jump rates] 
\label{ass:auxrate} 
There exists a family of constants $\mathbf{C} := \{C_{k,T}\}_{k \in \Nb, T \in \Rb_+ \subset (0,\infty)}$ with $(k,T) \mapsto C_{k,T}$ being componentwise non-decreasing such that for every $H \in \wh{\gs}$ and $T \in \Rb_+,$ 
\[\wh{\locrate}^H_j (t,x,\ems, \vms) \leq C_{|V_H|,T}, \quad \mbox{ for all } t \in [0,T], x \in \cad^{V_H}, \ems \in \emksp^{E_H}, \vms \in \vmksp^{V_H}.\] 
\end{assumptionp}

Combined with \eqref{eq:altstanding}, Assumption \ref{assu-newstanding} implies that the jump rate $\rate^{G,v}_j$ may depend on $v$ arbitrarily with no need to respect graph isomorphisms as in the homogeneous case. As noted in Remark \ref{rem-classfn}, removal of the symmetry condition implies that the jump rates do not necessarily satisfy the class property \eqref{mod:rsym0}. Therefore, we now state the main result of the section, noting that well-posedness is now framed as a property of a marked graph rather than an isomorphism class as in Theorem \ref{WP:WP}.

\begin{theorem}[Strong well-posedness of heterogeneous IPS] 
\label{saver:all} 
Fix a complete, filtered probability space $\fpspace$. Let $(G,\x)$ be an $\filtm_0$-measurable $\wh{\gs}\sp{\emksp,\vmksp\times \x}$-random element. If Assumption \ref{ass:auxrate} holds for a family of constants $\mathbf{C}$ and $G$ is a.s. finitely dissociable with respect to $\mathbf{C}$, then the SDE \eqref{mod:infpart} is strongly well-posed for the initial data $(G,\x)$. 
\end{theorem} 
\begin{proof} 
Define the mapping $\psi: \wh{\gs}\sp{\emksp,\vmksp}\to \wh{\gs}\sp{\emksp,\vmksp\times\Nb}$ as follows: 
\[\psi((G,\ems,\vms))\defeq (G,\ems,\vms,w)\te{ where }w_v = v\te{ for }v \in V.\] 
We slightly abuse notation by defining $\psi$ as in the display above for all choices of $(\emksp,\vmksp)$. Note that for any $H \in \wh{\gs}\sp{\emksp,\vmksp}$, the marked graph $\psi(H)$ has a trivial automorphism group. The key steps of the proof are as follows. We first use this fact to construct from the quasi-regular local jump rates ${\bf \wh{\locrate}}$, a collection of regular local jump rates ${\bf \locrate}$ (in the sense of Definition \ref{def-regular}) that act on marks that have been augmented by $\psi$. Next, we consider an auxiliary SDE \eqref{mod:infpart} with jump rates ${\bf \rate^{\psi(G)}}$ defined in terms of the local jump rates ${\bf \locrate}$ via \eqref{eq:standing}, and observe that this is strongly well-posed by Theorem \ref{WP:WP}. Then we establish a one-to-one correspondence between weak solutions to the SDE \eqref{mod:infpart} with jump rates ${\bf \rate^G}$ for the initial data $(G,\x),$ stated in the theorem, and weak solutions to the auxiliary SDE with jump rates ${\bf \rate^{\psi(G)}}$ for the initial data $\psi(G,\x).$ \\

\noindent {\em Step 1: Construction of regular local jump rates on an augmented graph.} Fix $H = (V_H,E_H,\root_H) \in \gsone$ and $(\ems,\vms,w) \in \emksp^{E_H}\times(\vmksp\times \Nb)^{V_H}$. Then define $\varphi_w: V_H \to \Nb$ by $\varphi_w(v) = w_v$, and by some abuse of notation, define $\varphi_{w}:E_H \to \{\{i,j\}: i, j \in \Nb\}$ by $\varphi_{w}(\{u,v\}) = \{\varphi_w(u),\varphi_w(v)\}$ for $\{u,v\} \in E_H$, with the obvious definition for $\varphi_w(A)$ when $A$ is a subset of $V_H$ or $E_H$. Also, when $\varphi_w$ is injective, we define  
\begin{equation} 
\label{def-Hw} 
H_w \defeq (\varphi_w(V_H),\varphi_w(E_H),\varphi_w(\root_H)). 
\end{equation} 
Observe that $H_w$ is the unique element of $\wgsone$ such that $\varphi_w \in I(H,H_w)$ or, in other words, $H_w$ is the unique graph that is isomorphic to $H$ and has vertex labels in $\Nb$ that are compatible with $w$ in the sense made explicit by \eqref{def-Hw}. Note that $H_w$ may fail to be well defined when $\varphi_w$ is not injective, but we consider this case separately. Then for $(t,x) \in [0,\infty)\times \cad^{V_H}$, and $j \in \jmps$, define the auxiliary local jump rates 
\begin{align*} 
&\locrate^H_j(t,x,\ems,\vms,w)\\ 
& \defeq 
\begin{cases} 
\wh{\locrate}^{\nm{H_w}}_j(t, (x_{\varphi_w^{-1}(v)})_{v \in \varphi_w(V_H)},(\ems_{\varphi_{w}^{-1}(e)})_{e \in \varphi_{w}(E_H)}, (\vms_{\varphi^{-1}_w(v)})_{v \in \varphi_w(V_H)}) &\te{ if } \varphi_w \te{ is injective},\\ 
0 &\te{ otherwise.} 
\end{cases} 
\end{align*}

We now show that this auxiliary family of local jump rates satisfies the symmetry property in Definition \ref{def-regular}. Suppose that $\wh{H}\cong H$, and fix some $\wh{\varphi}\in I(\wh{H},H)$. For $w \in \Nb^{V_H}$ (chosen such that $\varphi_w$ is injective), define 
\begin{equation} 
\label{def-whw} 
\wh{w}_v \defeq w_{\wh{\varphi}(v)} \te{ for all } v \in V_{\wh{H}}. 
\end{equation} 
We claim that $H_w = \wh{H}_{\wh{w}}$, where $\wh{H}_{\wh{w}}$ is defined as in \eqref{def-Hw}, but with $H$ and $w$ replaced by $\wh{H}$ and $\wh{w}$, respectively. To see why, note that for any $v \in V_{\wh{H}}$, by the definition of $\varphi_{\wh{w}}$ given above, \eqref{def-whw}, and the choice of $\wh{\varphi}$, $\varphi_{\wh{w}}(v) = \wh{w}_v = w_{\wh{\varphi}(v)} = \varphi_w (\wh{\varphi}(v))$. Thus, we have shown that $\varphi_{\wh{w}} = \varphi_w\circ \wh{\varphi}$. Together with the definition of $\wh{H}_{\wh{w}}$ from \eqref{def-Hw}, this implies that 
\[V_{H_w} = \varphi_w(V_H) = \varphi_w\circ \wh{\varphi}(\wh{\varphi}^{-1}(V_H)) = \varphi_{\wh{w}}(V_{\wh{H}}) = V_{\wh{H}_{\wh{w}}}.\] 
Likewise, $E_{H_w} = E_{\wh{H}_{\wh{w}}}$ and $\root_{H_w} = \root_{\wh{H}_{\wh{w}}}$, and so $H_w = \wh{H}_{\wh{w}}$ as claimed. 

Fix $t > 0$, $j \in\jmps$, $\wh{H}\cong H \in \gsone$, $\wh{\varphi} \in I(\wh{H},H)$, and $(x,\ems,\vms,w) \in \cad^{V_H}\times\emksp^{E_H}\times(\vmksp\times\Nb)^{V_H}$. Again, define $\wh{w}$ as in \eqref{def-whw}. If $\varphi_w$ is injective, applying the definitions of $\locrate^H$ and $\locrate^{\wh{H}}$ in the first and last equalities, respectively, and using $\wh{\varphi}\in I(\wh{H},H)$ and $(\wh{H}_{\wh{w}},\varphi_{\wh{w}}) = (H_w, \varphi_w\circ\wh{\varphi})$ in the second and third equalities, respectively, we conclude that 
\begin{align*} 
\locrate^H_j(t,&(x_v)_{v \in V_H},(\ems_e)_{e \in E_H},(\vms_v,w_v)_{v \in V_H}) \\ 
&= \wh{\locrate}^{\nm{H_w}}_j(t, (x_{\varphi_w^{-1}(v)})_{v \in \varphi_w(V_H)},(\ems_{\varphi_{w}^{-1}(e)})_{e \in \varphi_w(E_H)}, (\vms_{\varphi^{-1}_w(v)})_{v \in \varphi_w(V_H)})\\ 
&= \wh{\locrate}^{\nm{H_{w}}}_j(t, (x_{(\varphi_w\circ\wh{\varphi})^{-1}(v)})_{v \in \varphi_w\circ\wh{\varphi}(V_{\wh{H}})}, (\ems_{(\varphi_{w}\circ\wh{\varphi})^{-1}(e)})_{e \in \varphi_w\circ\wh{\varphi}(E_{\wh{H}})}, (\vms_{(\varphi_w\circ\wh{\varphi})^{-1}(v)})_{v \in \varphi_w\circ\wh{\varphi}(V_{\wh{H}})})\\ 
&= \wh{\locrate}^{\nm{\wh{H}_{\wh{w}}}}_j(t, (x_{(\varphi_{\wh{w}})^{-1}(v)})_{v \in \varphi_{\wh{w}}(V_{\wh{H}})}, (\ems_{(\varphi_{\wh{w}})^{-1}(e)})_{e \in \varphi_{\wh{w}}(E_{\wh{H}})}, (\vms_{(\varphi_{\wh{w}})^{-1}(v)})_{v \in \varphi_{\wh{w}}(V_{\wh{H}})})\\ 
&=\locrate^{\wh{H}}_j(t, (x_{\wh{\varphi}(v)})_{v \in V_{\wh{H}}}, (\ems_{\wh{\varphi}(e)})_{e \in E_{\wh{H}}}, (\vms_{\wh{\varphi}(v)}, w_{\wh{\varphi}(v)})_{v \in V_{\wh{H}}}), 
\end{align*} 
where we make use of the fact that the map $\varphi_w\circ\wh{\varphi}: V_{\wh{H}} \to \Nb$ satisfies $\varphi_w\circ\wh{\varphi}(v) = w_{\wh{\varphi}(v)}$ and is injective if and only if $\varphi_w$ is injective. If $\varphi_w$ is not injective, then 
\[\locrate^H_j(t,(x_v)_{v \in V_H},(\ems_e)_{e \in E_H},(\vms_v,w_v)_{v \in V_H}) = 0 = \locrate^{\wh{H}}_j(t, (x_{\wh{\varphi}(v)})_{v \in V_{\wh{H}}}, (\ems_{\wh{\varphi}(e)})_{e \in E_{\wh{H}}}, (\vms_{\wh{\varphi}(v)}, w_{\wh{\varphi}(v)})_{v \in V_{\wh{H}}}).\] 
Thus, $\mathbf{\locrate} = \{\locrate^H_j\}_{H \in \gsone,j\in\jmps}$ satisfies the symmetry condition of Definition \ref{def-regular}. The predictability condition follows directly from Assumption \ref{assu-newstanding}. Thus, $\{\locrate^H_j\}_{H \in \gsone,j\in\jmps}$ is a family of regular local jump rates. \\

\noindent {\em Step 2: Introducing an auxiliary SDE.} Next, in terms of the specified $\wh{\gs}\sp{\emksp,\vmksp\times\x}$ initial data $(G,\x)$ and the local jump rates $\mathbf{\locrate}$ defined above, we define the jump rates ${\bf \rate^{\psi(G)}}=\{\rate^{\psi(G),v}_j\}_{v \in V_G,j \in \jmps}$ via \eqref{eq:standing}. Then we define the auxiliary SDE to be the SDE \eqref{mod:infpart} with jump rates ${\bf \rate^{\psi(G)}}$ in place of ${\bf \rate^G}.$ Since $G$ is a.s. finitely dissociable by assumption, Definition \ref{findis:decom} trivially ensures that the unmarked version $\nm{G}$ and hence, $\psi(G)$, are also a.s. finitely dissociable. Further, since ${\bf \wh{\locrate}}$ satisfies Assumption \ref{ass:auxrate}, it follows from the definition given above that ${\bf \locrate}$ satisfies Assumption \ref{mod:assu}. Thus, by Theorem \ref{WP:WP}, the auxiliary SDE is strongly well-posed for the initial data $\ is {\psi(G,\x)}$. \\

\noindent {\em Step 3: Establishing a one-to-one correspondence.} For $G = (V,E,\root,\ems,\vms,w)$, $v \in V_G$, $j \in \jmps$, and $(t,x) \in \Rb_+\times\cad^{V_G}$, setting $H^v := \nm{G\subg{\cl{v}}} = (V_{H^v},E_{H^v},v=\root_{H^v})$, and using first \eqref{eq:standing}, then the definition of ${\bf \locrate}$ in Step 1 and the fact that $\varphi_w$ is the identity map, and lastly \eqref{eq:altstanding}, it follows that 
\begin{align*} 
\rate^{\psi(G),v}_j(t,x) &= \locrate^{H^v}_j(t,(x_u)_{u \in V_{H^v}}, (\ems_e)_{e \in E_{H^v}},(\vms_u)_{u \in V_{H^v}},(u)_{u\in V_{H^v}})\\ 
&= \wh{\locrate}^{H^v}_j(t,(x_u)_{u \in V_{H^v}},(\ems_e)_{e \in E_{H^v}},(\vms_u)_{u\in V_{H^v}})\\ 
&= \rate^{G,v}_j(t,x). 
\end{align*} 
It follows immediately from the equivalence above that a $\wh{\gs}\sp{\emksp,\vmksp\times \cad}$-random element $(G,X)$ is an $(\filt,\poiss^G)$-weak solution to the SDE \eqref{mod:infpart} (with jump rates ${\bf \rate^G}$) for the initial data $(G,\x)$ if and only if $\psi(G,X)$ is an $(\filt,\poiss^G)$-weak solution to the auxiliary SDE with jump rates ${\bf \rate^{\psi(G)}}$ for the initial data $\psi(G,\x)$. \\

To conclude the proof, note that from step 2, we know that the auxiliary SDE is strongly well-posed for the initial data $\ic{\psi(G,\x)}$. Therefore, by Remark \ref{isotograph}, it is also strongly well-posed for the initial data $\psi(G,\x)$. The one-to-one correspondence established in Step 3 then shows that the SDE \eqref{mod:infpart} (with jump rates $\mathbf{r}^G$) is strongly well-posed for the initial data $(G,\x)$. This concludes the proof. 
\end{proof}

\subsection{Characterization of Strong Well-Posedness on Random Graphs}
\label{cond}

\begin{proof}[Proof of Lemma \ref{mod:condwp}]
The key issue here is to show that conditioning on the initial data does not change the driving noise structure. Note that this is slightly non-standard as the driving noise is indexed by the vertices of the graph and thus is not completely independent of the initial data. Let $\fpspace$ be a complete, filtered probability space that supports $(G,\x)$ and a \fpp\ $(\filt,\poiss^G)$ such that $(G,X)$ and $(G,Y)$ are two $(\filt,\poiss^G)$-weak solutions of \eqref{mod:infpart} for $(G,\xi)$. To prove the lemma, it suffices to prove the following claim: for $\PP$-a.s. $\omega\in\Omega$, setting $(H,\x^H) \defeq (G(\omega),\x(\omega))$, there exists a \fpp\ $(\filtm^H,\poiss^H)$ on some probability space $(\Omega^H,\filtm^H,\filt^H,\PP^H)$ and two $(\filtm^H,\poiss^H)$-weak solutions $(H,X^H)$ and $(H,Y^H)$ to \eqref{mod:infpart} for $(H,\x^H)$ such that $\law((G,X,Y,\poiss^G)|\filtm_0)(\omega) = \law((H,X^H,Y^H,\poiss^H))$. Indeed, if the claim holds, then a.s. strong well-posedness of \eqref{mod:infpart} for every realization of the random graph $(G,\x)$ implies that $\PP(X = Y|\filtm_0)(\omega) = \PP^H(X^H = Y^H) = 1$ for $\PP$-a.s. $\omega \in \Omega$. Hence $\PP(X=Y) = 1$, which proves strong well-posedness. This claim can be proved via direct verification of Definitions \ref{mod:dnoise}, \ref{mod:sol}. We include the details for completion.

To prove the claim, let $(H,X^H,Y^H,\poiss^H)$ be a random element with law $\law((G,X,Y,\poiss^G)|\filtm_0)(\omega)$ and fix a complete, filtered probability space $(\Omega^H,\filtm^H,\filt^H,\PP^H)$ that supports $(H,X^H,Y^H,\poiss^H)$ where $\filt^H$ is the minimal filtration satisfying the usual conditions such that $X^H,Y^H$ and $\poiss^H$ are all $\filt^H$-adapted in the sense that for any $v \in V_H$, $X^H_v,Y^H_v$ and $\poiss^H_v$ are all $\filt^H$-adapted (point) processes. Since by assumption $\poiss^G$ is a $\filt$-driving noise, condition 1 of Definition \ref{mod:dnoise} immediately implies $\poiss^H$ is a collection of i.i.d. Poisson processes on $\Rb_+^2\times\jmps$ with intensity $\leb^2\otimes\jmps$, indexed by the vertices of $H$, and hence that $(H,\poiss^H)$ is a $\filt^{\poiss^H}$-driving noise.

Since \eqref{mod:infpart} holds a.s., for $\PP$-a.s. $\omega \in \Omega$, $(H,X^H)$ and $(H,Y^H)$ are both graphs with random c\`adl\`ag marks that, together with $(H,\poiss^H)$, $\PP(\cdot|\filtm_0)(\omega)$-a.s. solve \eqref{mod:infpart} for $(H,\x^H)$. Thus, $(H,X^H)$ and $(H,Y^H)$ satisfy conditions 1 and 3 of Definition \ref{mod:sol} $\PP$-a.s. with respect to the filtered probability space $(\Omega^H,\filtm^H,\filt^H,\PP^H)$ (noting that any $\PP(\cdot|\filtm_0)(\omega)$-null event is also $\PP^H$-null), and so it suffices to prove that $\poiss^H$ is $\PP$-a.s. a collection of i.i.d. $\filt^H$-Poisson processes. To do this, we note that condition 2 of Definition \ref{mod:dnoise} implies that for any $t > 0$ and $A \in \borel((t,\infty)\times \Rb_+\times\jmps)$, the random element $(G,\poiss^G(A))$ is conditionally independent of the $\filtm_t$-measurable random element $(G,X[t],Y[t],\poiss^G_t)$ given $\filtm_0$, where $\poiss^G_t = \poiss^G|_{[0,t]\times\Rb_+\times\jmps}$. Thus, $\PP$-a.s., $(H,\poiss^H(A))$ is independent of $(H,X^H[t],Y^H[t],\poiss^H_t)$. Then using a standard approximation argument exploiting the fact that Borel sigma algebras of subsets of Polish spaces are countably generated and that $\filt^H$ is complete, it follows that $\PP$-a.s., $\poiss^H_v(A)$ is independent of $\filtm^H_t$ for all $v \in V_H$, $t \in \Rb_+$ and $A \in \borel((t,\infty)\times\Rb_+\times\jmps)$. Thus, $\{\poiss^H_v\}_{v \in V_H}$ is $\PP$-a.s. a collection of $\filt^H$-Poisson processes. Therefore, $\poiss^H$ is $\PP$-a.s. an $\filt^H$-driving noise in the sense of Definition \ref{mod:dnoise}, so condition 2 of Definition \ref{mod:sol} is also $\PP$-a.s. satisfied. This concludes the proof of the claim. 
\end{proof}

\subsection{Verification of Assumptions for Examples in Section \ref{mod:examples}}
\label{varex} 
In this section, we verify conditions under which the examples in Section \ref{mod:examples} satisfy Assumptions \ref{mod:assu} and \ref{mod:cont}.

\noindent\textbf{Example \ref{mod:Markov}:} If there exist non-decreasing constants $\{C_k\}_{k \in \Nb}$ such that $\sup_{x \in \Xc^{V_H}}|\alt{\rate}^H_j(x)| \leq C_{|H|}$ for all $j \in \jmps$ and $H \in \gsone$, then Assumption \ref{mod:assu} holds with $C_{k,T} \defeq C_k$. Moreover, Assumption \ref{mod:cont} holds trivially for all initial data because the mark spaces $\emksp = \vmksp = \{1\}$ are trivial. Note that in this case, the solution $X^{G,\x}$ to the SDE \eqref{mod:infpart} with initial data $(G, \x)$ is a homogeneous Markov process. 

\noindent\textbf{Example \ref{ex:RE}:} The jump rate is continuous with respect to the initial marks $(\ems,\vms)$, so Assumption \ref{mod:cont} holds for all initial data. This model satisfies Assumption \ref{mod:assu} when (1) there exist deterministic constants $\vms^*$ and $\ems^*$ such that $\sup_{v\in V}\vms_v \leq \vms^* < \infty$ and $\sup_{e\in E}\ems_e \leq \ems^* <\infty$, and (2) the functions $\lambda$ and $\rho$ are bounded from above by respective non-decreasing functions $\wh{\lambda},\wh{\rho}: \Rb_+\to \Rb_+$. Under these conditions, $C_{k,T} = \vms^*\wh{\rho}(T) + (k-1)\ems^*\wh{\lambda}(T)$ for all $k\in\Nb,T \in\Rb_+$.

\noindent\textbf{Example \ref{ex:NMcont}:} Assumption \ref{mod:assu} holds when the functions $\beta$ and $\gamma$ are bounded from above. Note that in the Markov case, $\beta \equiv \lambda$ and $\gamma \equiv 1$ are both constant. Finally, if both $\beta$ and $\gamma$ are also continuous, then Assumption \ref{mod:cont} holds for all initial data.

\section{Well-Posedness for Finite Initial Data}
\label{WPfin}

Under Assumption \ref{mod:assu},  well-posedness of \eqref{mod:infpart} is common knowledge when the initial data is finite, but we establish it here for completeness. In this case, we also show that the trajectories also satisfy the following additional regularity property. Recall the definition of the discontinuity set $\jmp{x}{t}$ given in \eqref{skor:disc}. 

\begin{definition}[Proper trajectories] 
\label{skor:proper}
Given a countable set $W$ and $t \in [0,\infty)$, we say $x\in \cad_t^{W}$ is \emph{proper} if $\jmp{x_w}{t} \cap \jmp{x_v}{t} = \emptyset$ for all distinct $v, w \in W$. Moreover, we say a trajectory $x \in \cad^{W}$ is proper if its restriction $x[t]$ to $[0,t]$ is proper for all $t \in \Rb_+$. 
\end{definition}

\begin{proof}[Proof of Proposition \ref{WP:fin} and trajectories being a.s. proper:]
Let $(G,\x) = (V,E,\root,\ems,\vms,\x)$ any a.s. finite $\sp{\emksp,\vmksp\times \Xc}$-random graph. By Lemma \ref{mod:condwp}, it suffices to prove that \eqref{mod:infpart} is strongly well-posed for $(G,\x)$ with a.s. proper solutions under the additional assumption that $(G,\x)$ is deterministic.

Let $(\filt,\poiss^G)$ be a \fpp\ in the sense of Remark \ref{mod:concise}. First note that for any $(\filt,\poiss^G)$-weak solution $(G,X)$ to \eqref{mod:infpart}, any distinct vertices $u,v \in V$ and any $T > 0$,
\[\jmp{X_v}{T}\cap\jmp{X_u}{T} \subseteq \{s \in [0,T]: \poiss^G_u(\{s\}\times (0,C_{k,T}]\times\jmps)\poiss^G_v(\{s\}\times (0,C_{k,T}]\times\jmps) = 1\} = \emptyset \te{ a.s.,}\]
by \eqref{mod:infpart} and Assumption \ref{mod:assu} where $\{C_{k,T}\}$ is the family of constants from Assumption \ref{mod:assu} and $k \defeq \max\{|\cl{u}|,|\cl{v}|\}$. Since this holds for all $T$, $(G,X)$ is a.s. proper. 

Next, fix a \fpp\ $(\filt,\poiss^G)$. Define the finite collection,
\[\Emc \defeq \{(\tau_n,r_n,j_n,v_n)\in [0,T]\times[0,C_{K,T}]\times\jmps\times V: \poiss^G_{v_n}(\{(\tau_n,r_n,j_n)\}) = 1\},\]
$K = \max_{v\in V} |\cl{v}|$ and $\{\tau_n\}_{n=1}^{|\Emc|}$ is increasing. Note that $\{\tau_n\}$ is, in fact, strictly increasing. define the $\cad^V_T$-random element $X[T]$ by,
\[X_v(t) = \begin{cases}
\x_v &\te{ if } 0 \leq t < \tau_1,\\
X_v(\tau_n) &\te{ if } \tau_n\leq t <\tau_{n+1}, n < |\Emc|,\\
X_v(\tau_n) &\te{ if } \tau_n\leq t \leq T, n = |\Emc|,\\
X_v(\tau_n) &\te{ if } t = \tau_{n+1}, r_{n+1} > \rate^{G,v}_{j_{n+1}}(t_{n+1},X),\\
X_v(\tau_n) + j_{n+1} &\te{ if } t = \tau_{n+1}, r_{n+1} > \rate^{G,v}_{j_{n+1}}(t_{n+1},X).
\end{cases}\]
Then note that for $t \in [0,T]$, $X(t)$ is clearly $\filtm^{G,\x,\poiss^G}_t$-measurable so $X$ is a $\poiss^G$-strong solution to \eqref{mod:infpart} for $(G,\x)$ on the interval $[0,T]$. Furthermore, it is clear that any $(\filt,\poiss^G)$-weak solution to \eqref{mod:infpart} must satisfy the above display, so all $(\filt,\poiss^G)$-weak solutions equal $X$ on the interval $[0,T]$. Because $T$ is arbitrary, and for any $T' > T$ the corresponding solution $X'$ satisfies $X'[T] = X[T]$ a.s., it follows that there exists a $\poiss^G$-strong solution to \eqref{mod:infpart} and that solution is pathwise unique. Therefore \eqref{mod:infpart} is strongly well-posed for $(G,\x)$.
\end{proof}



\bibliographystyle{plain}
\bibliography{reference}

%


\end{document}